\documentclass[11pt]{amsart}
\usepackage{fullpage}

\usepackage{hyperref}
\usepackage{amsmath,amsfonts,amssymb,mathtools}
\usepackage{mathrsfs}
\usepackage{verbatim}
\usepackage{amsthm}

\newtheorem{theorem}{Theorem}[section]
 \newtheorem{corollary}[theorem]{Corollary}
 \newtheorem{lemma}[theorem]{Lemma}
 \newtheorem{proposition}[theorem]{Proposition}
 \theoremstyle{definition}
 \newtheorem{definition}[theorem]{Definition}
 \newtheorem{setting}[theorem]{Setting}
 \newtheorem{hypothesis}[theorem]{Hypothesis}
 \theoremstyle{remark}
 \newtheorem{remark}[theorem]{Remark}
  \newtheorem{ex}[theorem]{Example}
 \numberwithin{equation}{section}

\def \bC {\mathbb C}

\def \bH {\mathbb H}

\def \bN {\mathbb N}

\def \bP {\mathbb P}

\def \bR {\mathbb R}

\def \bV {\mathbb V}

\def \bZ {\mathbb Z}

\def \cA {\mathcal A}
\def \cB {\mathcal B}
\def \cC {\mathcal C}
\def \cD {\mathcal D}

\def \cF {\mathcal F}
\def \cG {\mathcal G}
\def \cH {\mathcal H}

\def \cK {\mathcal K}
\def \cL {\mathcal L}

\def \cO {\mathcal O}

\def \cR {\mathcal R}
\def \cS {\mathcal S}

\def \fg {\mathfrak g}

\def \fU {\mathfrak U}

\def \sL{\mathscr L}

\def \sF{\mathscr F}
\def \tr {{\rm Tr}}

\def \id {\text{\rm I}}
\def \Op {{\rm Op}}
\def \supp{{{\rm supp}}}
\def \Gh {{\widehat G}}
\def \eps {\varepsilon}
\def \vol {{\rm vol}}

\def \sp {{\rm sp}}
\def\dom {\text {\rm Dom}}
\def\IM  {\text {\rm Im}}
\def\RE  {\text {\rm Re}}
\def\diag  {\text {\rm diag}}

\begin{document}
\author[V. Fischer]{V\'eronique Fischer}
\address[V. Fischer]
{University of Bath, Department of Mathematical Sciences, Bath, BA2 7AY, UK} 
\email{v.c.m.fischer@bath.ac.uk}

\author[S. Mikkelsen]{S\o ren Mikkelsen}
\address[S. Mikkelsen]
{University of Helsinki, Department of Mathematics and Statistics, Finland} 
\email{soren.mikkelsen@helsinki.fi}

\title{Semiclassical functional calculus on nilpotent Lie groups and their compact nilmanifolds}



\begin{abstract}
In this paper, we show that the semiclassical calculus recently developed on nilpotent Lie groups and nilmanifolds include the functional calculus of  suitable   subelliptic operators.
Moreover, we obtain Weyl laws for these operators. 
Amongst these operators are sub-Laplacians in horizontal divergence form perturbed with a potential and their generalisations. 
\end{abstract}

\maketitle

\makeatletter
\renewcommand\l@subsection{\@tocline{2}{0pt}{3pc}{5pc}{}}
\makeatother

\tableofcontents

\section{Introduction}

For more than a century, 
the global analysis of elliptic operators has  attracted interest from many branches of mathematics, especially in spectral analysis on manifolds and in theoretical physics. 
Hypoelliptic operators  have been intensively studied since the 60's and the seminal works of Lars H\"ormander on the subject (eg \cite{Hormander67})
and of Rothschild and Stein 
 on the analysis of sums of square of vector fields \cite{RotschildStein}.
The focus has mostly been on 
sub-Laplacians on  CR and contact manifolds or related settings \cite{Beals+Greiner}.
For these, 
the abundance  of works on their spectral properties from both the microlocal and semiclassical viewpoint in the last decade is remarkable \cite{CdVHT,Letrouit+Clo,LetrouitQL,Burq+Sun,LetrouitNodal,arnaiz2023quantumlimitsperturbedsubriemannian}, although some initial results had been known previously
 \cite{zelditch97,Ponge}.

A fundamental tool in spectral analysis is the functional calculus. The latter is described, at least abstractly, by the spectral theorem which  gives a meaning as a well defined (possibly unbounded) operator to $\phi(T)$ 
for  any Borel measurable function $\phi$ of a self-adjoint operator $T$.  In certain situations, it can be difficult to get information about e.g. the spectral properties or behaviour of the integral kernel of $\phi(T)$. For pseudo-differential operators on a Riemannian manifold the situation is much improved: under certain conditions on the smooth function $\phi$ and the pseudo-differential operator $T$, especially its ellipticity, it is well established that $\phi(T)$ is again a pseudo-differential operator. For the case when $\phi$ is a polynomial, this follows immediately from the symbolic properties of the pseudo-differential calculus. The first results for  more general functions $\phi$ was obtained by Seeley in \cite{MR0237943}, where it is established that complex powers of elliptic pseudo-differential operators are again a pseudo-differential operator. This was further generalised by many authors; here is a short list serving a small sample  \cite{MR0310713,MR0618463,MR0474431,MR0547022,MR0811855,MR0642653,MR1385196}.
The semiclassical functional calculus for pseudo-differential operators on $\bR^d$ was developed by Helffer and Robert in \cite{MR724029}, see also the monographs \cite{MR897108,dimassi+Sj}; this generalises to Riemannian manifolds \cite{zworski}.

The objective of this paper is to study the functional calculus of a large class of hypoelliptic operators. 
We will also include the extension to semiclassical pseudo-differential calculus to obtain interesting and new Weyl asymptotics. 
The class of hypoelliptic operators studied in this paper are defined on nilpotent Lie groups and their nilmanifolds. Indeed, in these settings, a pseudo-differential calculus has been defined 
and studied in \cite{R+F_monograph,FFPisa,bologna,FFFMilan,FFFGeomInv,fischer2022semiclassicalanalysiscompactnilmanifolds}.
The importance of these settings stems from  the pioneering work of Stein \cite{RotschildStein} and many of his collaborators in the 1970's and 80's: they developed the idea that the H\"ormander condition on vector fields leads to the analysis of convolution operators acting on nilpotent Lie groups. 
It has transpired since then that in many senses (for instance in the metric geometric sense \cite{MR1421822}), the setting of stratified nilpotent Lie groups of sub-Riemannian geometry corresponds to the flat Euclidean case of Riemannian geometry.   
The functional calculus of sub-Laplacians has therefore mainly concerned left-invariant sub-Laplacians on stratified nilpotent Lie groups.
Folland proved \cite{Folland75} that 
that their heat kernels are Schwartz functions,
allowing for many results in spectral multipliers in this operator, see e.g.  \cite{Christ91,MR0657581,MR1240169,MR3283616,MR3513881,MR3985515,MR1125759} and for generalisation on groups of polynomial growth \cite{MR1172944}. 
Hulanicki generalised \cite{hulanicki} this property not only to the more general class of (left-invariant)
  Rockland operators on graded Lie groups  but also to Schwartz function not necessarily the negative exponential; this has led to study  functional calculi in several commuting left-invariant operators on Lie groups, see eg \cite{MR2772351} and references therein.
 Some holomorphic functional calculus for Rockland operators  were considered in  \cite{cardona2021analyticfunctionalcalculusgrarding}.

 To the best of the authors' knowledge, the  functional calculus of non-left-invariant operators  on Lie groups has not been considered before and our results on semiclassical Weyl laws are  new.
 This is an important step towards the future development of 
  a functional calculus for sub-Laplacians on sub-Riemannian manifolds  and obtaining semiclassical Weyl asymptotics for subelliptic operators. 

\subsection*{Novelty and importance}
Following ideas from Michael Taylor \cite{Taylor86}, 
many groups have a well defined quantization attached to their  group Fourier transform and representation theory,  see e.g. \cite{LinoCLoMe}.
When the group is non-commutative, this leads to a notion of operator-valued symbols for Fourier multipliers.  
The symbol classes on nilpotent Lie groups and its resulting  pseudo-differential calculi on nilpotent Lie groups $G$  and their nilmanifolds $M$ 
have been defined and actively studied in the past ten years by the first author and her collaborators, see eg \cite{R+F_monograph,FFPisa,bologna,FFFMilan,FFFGeomInv,fischer2022semiclassicalanalysiscompactnilmanifolds}. 
In order to give a precise and self-contained presentation in this paper, we have included the definitions for these pseudo-differential calculi as well as many proofs of expected properties that may be known to experts; this explains its length.

The  novelty of  the paper lies in relating the  semiclassical pseudo-differential and functional calculi on $G$ or $M$. Indeed, we give natural conditions on the function $\phi$ and on a semiclassical pseudo-differential operator $T$ so that $\phi(T)$ makes sense functionally and is related to the  pseudo-differential calculus.
The description of the class of functions $\phi$ is simple and traditional: it is the space 
of smooth functions $\phi:\bR \to \bC$ growing at certain rate in the sense of Definition \ref{def_cGm}. 
The conditions on $T$ are imposed on its principal symbols $\sigma_0$ which needs to be non-negative and such that $\id+\sigma_0$ is invertible, see Section \ref{subsec_settingFC}.

The case of such operators $T$ that are also  left-invariant and  obtained as the quantization of their principal symbol 
are  the Rockland operators, and our results then boils down to the ones obtained in  \cite[Section 5.3]{R+F_monograph} and recalled in Theorem \ref{thm_phi(R)} in this paper.
The main example of these is the class of invariant sub-Laplacians on stratified nilpotent Lie groups. 
It shows for instance that the functional calculus of the canonical sub-Laplacian $\cL_{\bH_n}$ on the Heisenberg group $\bH_n$
is in the nilpotent calculus (i.e. the 
pseudo-differential calculus developed in \cite[Section 5]{R+F_monograph}), 
whereas it is known that it is not in the classical H\"ormander calculus.
For instance, the square root of $\cL_{\bH_1}$ 
is an important operator for the study of the wave equations for $\cL_{\bH_1}$, but it 
is not in the  classical H\"ormander calculus of the underlying manifold $\bR^3$  of $\bH_1$
(see eg \cite[p.57]{LetrouitQL}).
However,  $\sqrt{\id + \cL_{\bH_1}}$ is in the nilpotent calculus \cite[Section 5]{R+F_monograph}, and therefore the square root of $\cL_{\bH_1}$ is in the polyhomogeneous nilpotent calculus \cite{FFPisa}. 
This paper produces more examples of subelliptic operators whose functional calculus escapes the classical H\"ormander calculus but lies in the nilpotent calculus. This paves the way for further investigations in less `flat' geometric sub-Riemannian settings. 

The proof of our main functional theorem (Theorem \ref{thm_sclFCG}) 
follows the now traditional lines of reasoning for semiclassical functional calculus in the Euclidean (abelian) setting: it
relies on constructing parametrices for the resolvent of the operator $T$ together with the Helffer--Sj\"ostrand formula. 
All these arguments are done within  the nilpotent calculus. In particular, 
as the symbols are operator valued, we first develop the functional calculus of the principal symbol.
As applications,  we obtain semiclassical Weyl asymptotics (Theorem \ref{Thm:semiclasical_weyl_law}).

\subsection*{An application}
In order to motivate the paper, let us describe  Weyl asymptotics for a particular class of operators that we call sub-Laplacians in divergence form  perturbed with a confining potential
(here, $
C^\infty_{l,b}(G)$  denotes the space of smooth bounded functions with bounded left-derivatives on $G$):

\begin{theorem}
Let $G$ be a stratified nilpotent Lie groups. 
We fix a basis $X_1,\ldots,X_{n_1}$ of the first stratum of its 	Lie algebra $\fg$.
	Let $a_{i,j}\in C_{l,b}^\infty(G)$, $1\leq i,j \leq n_1$, be such that at every point $x\in G$, the resulting matrix $A(x)=(a_{i,j}(x))$ is non-negative.
	Identifying the elements of $\fg$ with left-invariant vector fields,  
	we consider	the sub-Laplacian in divergence form:
	$$
	\cL_A:=-\sum_{1\leq i,j\leq n_1} X_i (a_{i,j}(x) X_j)
	= -\sum_{1\leq i,j\leq n_1} a_{i,j}(x) X_iX_j + (X_i a_{i,j}(x)) X_j.
	$$
	
We assume that  the  minimum and maximal eigenvalues $\lambda_{A(x),min}$ and 	$ \lambda_{A(x),\max}$ of $A(x)$ are positive and satisfy 
  the following ellipticity condition 
$$
\inf_{x\in G} \lambda_{A(x),min}>0
\qquad\mbox{and}\qquad  
\sup_{x\in G} \lambda_{A(x),max}<\infty.
$$	
Then $\cL_A$ is hypoelliptic. 
Moreover, for any  function $V\in C_{l,b}^\infty(G)$, the operator 
$\cL_A+V$ is hypoelliptic.
If  $V$ is non-negative, and if $a<b$ and $\delta>0$ are such that $V^{-1}((a-\delta,b+\delta))$ is compact, then the operator
$\eps^2 \cL_A + V$ for any $\eps\in (0,1]$ is essentially self-adjoint, and the part of its spectrum in the interval $[a,b]$ is discrete. Furthermore, this semiclassical family satisfies  the Weyl asymptotics:
$$
\lim_{\eps \to 0}\eps^{Q} \tr[1_{[a,b]}(\eps^2 \cL_A + V)] = \int_{G\times \Gh} \tr \left(1_{[a,b]}( \sigma_0 (x,\pi) ) \right)dx d\mu(\pi),
$$
where $Q$ is the homogeneous dimension, 
$\mu$ is the Plancherel measure on the unitary dual $\Gh$ of  $G$,
and $\sigma_0$ is the symbol given by:
$$
\sigma_0 (x,\pi)= \sum_{1\leq i,j\leq n_1} a_{i,j}(x)\, \pi( X_i)   \pi( X_j) \  + \ V(x).
$$
We have a similar result on the nilmanifold $M$. 
\end{theorem}

The hypoellipticity is obtained from Lemma \ref{lem_LApar} and the Weyl asymptotics from Corollary \ref{cor1_Thm:semiclasical_weyl_law}.
Note that no hypothesis on the multiplicity of the eigenvalues of $A(x)$ are assumed, so $\cL_A$ is not necessarily a sum of squares of vector fields, and its hypoellipticity is not a direct consequence of H\"ormander's theorem \cite{Hormander67}.

\subsection*{Organisation of the paper}
In order to make the paper self-contained, we have included preliminaries on nilpotent Lie groups $G$  and their nilmanifolds $M$ in Section 
 \ref{sec_prelM}, 
 as well as some generalities 
  in the graded case in Section \ref{sec:Graded_Lie_group}.
  We also recall many definitions and mostly known properties regarding the pseudo-differential calculi on 
$G$ and $M$  in Sections~\ref{Sec:calculi_symbols} and~\ref{sec:semiclassical_calculus_G_M}.
The novel results of this paper are 
the semiclassical functional calculi and its applications to Weyl asymptotics
in 
Sections~\ref{sec_FCG}
and~\ref{sec_WeylR} respectively. 
In appendix to the paper, we give the proof of  semiclassical composition and adjointness in Section \ref{secA_pfthm_sclexp_prod+adj} as well as some properties of almost analytical extensions in Section \ref{Appendix_almost_analytic}.

\subsection*{Acknowledgement}
The authors are  grateful to the Leverhulme Trust for their support via Research Project Grant 2020-037.
We also thank Lino Benedetto for interesting discussions and references. 

\section{Preliminaries on nilpotent Lie groups and   nilmanifolds} 
\label{sec_prelM}

In this section, we set our notation for nilpotent Lie groups and nilmanifolds.
We also recall some elements of harmonic analysis in this setting. 

\subsection{About nilpotent Lie groups and nilmanifolds}
\label{subsec_aboutGnilp}
In this paper, a nilpotent Lie group $G$ is always assumed  connected and simply connected unless otherwise stated. 
It is a smooth manifold which is identified with $\bR^n$ via the exponential mapping and
polynomial coordinate systems. 
This leads to a corresponding Lebesgue measure on its Lie algebra $\fg$ and the Haar measure $dx$ on the group $G$,
hence $L^p(G)\cong L^p(\bR^n)$.
This also allows us \cite[p.16]{corwingreenleaf}
to define the spaces 
$$
\cD(G)\cong \cD(\bR^n)
\quad \mbox{and}\quad  
\cS(G) \cong \cS(\bR^n)
$$
 of test functions which are smooth and compactly supported or Schwartz, 
and the corresponding spaces of distributions 
$$
\cD'(G)\cong \cD'(\bR^n)
\quad \mbox{and}\quad 
\cS'(G)\cong \cS'(\bR^n).
$$
Note that this identification with $\bR^n$ does not usually extend to the convolution: the group convolution, i.e. the operation between  two functions on $G$ defined formally via 
$$
 (f_1*f_2)(x):=\int_G f_1(y) f_2(y^{-1}x) dy,
$$
 is   not commutative in general whereas it is a commutative operation for functions on  the abelian group $\bR^n$.

\subsubsection{Compact nilmanifolds}
A compact nilmanifold is the quotient $M=\Gamma\backslash G$ of 
a  nilpotent Lie group $G$  by a discrete co-compact subgroup $\Gamma$ of $G$.
A concrete example of discrete co-compact subgroup is the natural discrete subgroup of the Heisenberg group, as described in 
\cite[Example 5.4.1]{corwingreenleaf}. 
Abstract characterisations are discussed in  \cite[Section 5.1]{corwingreenleaf}.

An element of $M$ is a class 
$$
\dot x := \Gamma x
$$
 of an element $x$ in $G$. If the context allows it, we may identify this class with its representative $x$. 

The quotient $M$ is naturally equipped with the structure of a compact smooth manifold. 
Furthermore, fixing a Haar measure on the unimodular group $G$, 
$M$ inherits a measure $d\dot x$ which is invariant under the translations  given by
$$
\begin{array}{rcl}
M & \longrightarrow & M\\
\dot x & \longmapsto & \dot x g = \Gamma xg
\end{array}, \quad g\in G.
$$
Recall that the Haar measure $dx$ on $G$ is unique up to a constant and, once it is fixed, $d\dot x$ is the only $G$-invariant measure on $M$ satisfying 
for any  function $f:G\to \mathbb C$, for instance continuous with compact support,
\begin{equation}
\label{eq_dxddotx}
	\int_G f(x) dx = \int_M \sum_{\gamma\in \Gamma} f(\gamma x) \ d\dot x.
\end{equation}
We denote by $\vol (M) = \int_M 1 d\dot x$  the volume of $M$.

\subsubsection{Fonctions on $G$ and $M$}
\label{subsubsec_periodicfcn}

Let $\Gamma$ be a discrete co-compact subgroup of a nilpotent Lie group $G$. 

We say that a function $f:G\rightarrow \mathbb C$  is  $\Gamma$-left-periodic or just  $\Gamma$-periodic  when we have 
$$
\forall x\in G,\;\;\forall \gamma\in \Gamma ,\;\; f(\gamma x)=f(x).
$$
This definition extends readily to measurable functions and to distributions.  

There is a natural one-to-one correspondence between the functions on $G$ which are $\Gamma$-periodic and the functions on $M$.
Indeed, for any map $F$ on $M$, 
the corresponding periodic function on $G$ is $F_G$ defined via
$$	 
F_G(x) := F(\dot x), \quad x\in G,
$$
while if $f$ is a $\Gamma$-periodic function on $G$, 
it defines a function $f_M$ on $M$ via
$$
f_M(\dot x) =f(x), \qquad x\in G.
$$
Naturally, $(F_G)_M=F$ and $(f_M)_G=f$.

We also  define, at least formally, the periodisation $\phi^\Gamma$ of a function $\phi(x)$ of the variable $x\in G$ by:
$$
\phi^\Gamma(x) = \sum_{\gamma \in \Gamma } \phi(\gamma x), \qquad x\in G.
$$

If $E$ is a space of functions or of distributions on $G$, then we denote by $E^\Gamma$ the space of elements in $E$ which are  $\Gamma$-periodic. 
Although 
$\cD(G)^\Gamma = \{0\} = \cS(G)^\Gamma,$
many other periodised functions or functional spaces have interesting descriptions on $M$ \cite{nilmanifold}:
\begin{proposition}
\label{prop_nilmanifold}
\begin{enumerate}
\item
 The periodisation of a Schwartz  function $\phi\in \cS(G)$ is a well-defined function  $\phi^\Gamma$ in $C^\infty(G)^\Gamma$.
Furthermore, the map $\phi \mapsto \phi^\Gamma$ yields a surjective morphism of topological vector spaces from
$\cS(G)$ onto $C^\infty(G)^\Gamma$
and from
$\cD(G)$ onto $C^\infty(G)^\Gamma$.
\item  For every $F\in \cD'(M)$, the tempered distribution  $F_G\in 
\cS'(G)$ is defined by
$$
\forall \phi\in \cS(G),\qquad 
\langle F_G,\phi\rangle  = \langle F , (\phi^\Gamma)_M\rangle.
$$
The map $F\mapsto F_G$ yields an isomorphism of topological vector spaces
from $\cD'(M)$ onto $\cS'(G)^\Gamma=\cD'(G)^\Gamma$.
\item 
For every $p\in [1,\infty]$,
the map $F\mapsto F_G$ is an isomorphism of the topological vector spaces (in fact Banach spaces) from $L^p(M)$ onto $L^p_{loc}(G)^\Gamma$ with inverse $f\mapsto f_M$.
\item 
Let $f\in \cS'(G)^\Gamma$ and $\kappa\in \cS(G)$.
Then $(\dot x, \dot y)  \mapsto \sum_{\gamma\in \Gamma} \kappa (y^{-1} \gamma x)$
is a smooth function on $M\times M$ and  $f*\kappa \in C^\infty(G)^\Gamma$.
Viewed as a function on $M$, 
$$
(f*\kappa)_M(\dot x) = \int_M f_M(\dot y) \ (\kappa (\cdot ^{-1} x)^\Gamma)_M (\dot y) d\dot y
= \int_M f_M(\dot y) \sum_{\gamma\in \Gamma} \kappa (y^{-1} \gamma x) \  d\dot y.
$$
\end{enumerate}
 \end{proposition}

\subsubsection{Operators on $G$ and $M$}
\label{subsubsec_opGM}
A mapping $T:\cS'(G)\to \cS'(G)$
or $\cD'(G) \to \cD'(G)$ is (left-)invariant under an element $g\in G$ when  
$$
\forall f\in \cS'(G) \ (\mbox{resp.} \, \cD'(G)), \qquad T(f(g \, \cdot)) = (Tf)(g \, \cdot).
$$ 
It is invariant under a subset of $G$ if it is invariant under every element of the subset.

\begin{ex}
	Many operators considered in this paper will be (right) convolution operators $T$ on $G$,
by which we mean  operators of the form $Tf=f*\kappa$ for any $f\in \cS(G)$ with $\kappa\in \cS'(G)$.
The distribution $\kappa$  is called the convolution kernel of $T$ and may be denoted by 
$$
\kappa :=T\delta_0.
$$ 
By the Schwartz kernel theorem, a continuous operator $T:\cS(G)\to \cS'(G)$ that is invariant under (left-)translation under $G$ in the above sense is a convolution operator. 
\end{ex}

Consider a linear continuous mapping $T:\cS'(G)\to \cS'(G)$
or $\cD'(G) \to \cD'(G)$ respectively
which is invariant under $\Gamma$. Then it  naturally induces
a linear continuous mapping
 $T_M$ on $M$ given via
$$
T_M F = (TF_G)_M, \qquad F\in \cD'(M).
$$
Consequently, 
if  $T$ coincides with   a smooth differential operator 	on $G$ that is invariant under $\Gamma$, then $T_M$ is a smooth differential operator on $M$.
For convolution operators $T$, 
the results in  Proposition \ref{prop_nilmanifold} yield:
\begin{lemma}
\label{lem_IntKernel_general}
Let $\kappa\in \cS(G)$ be a given convolution kernel, and let us denote by $T$ the associated convolution operator:
$$
T (\phi) = \phi*\kappa, \qquad \phi\in \cS'(G).
$$
The operator $T$ is a linear continuous mapping $\cS'(G)\to \cS'(G)$. 
The corresponding operator $T_M$ maps  $\cD'(M)$ to $\cD'(M)$ continuously and linearly. Its  integral kernel is the smooth function $K$ on $M\times M$ given by
$$
K(\dot x,\dot y) = 
\sum_{\gamma\in \Gamma}  \kappa(y^{-1}\gamma  x).
$$
Consequently, the operator $T_M$ is Hilbert-Schmidt on $L^2(M)$ with Hilbert-Schmidt norm 
$$
\|T_M\|_{HS} = \|K\|_{L^2(M\times M)}.
$$
\end{lemma}

\subsection{Representation theory and Plancherel theorem}

\subsubsection{Representations of $G$ and $L^1(G)$}
In this paper,  we always assume that the representations of the group $G$ 
are strongly continuous and acting unitarily on separable Hilbert spaces. For a representation $\pi$ of $G$, 
we keep the same notation for the corresponding infinitesimal representation
which acts on the universal enveloping algebra $\fU(\fg)$ of the Lie algebra of the group.
It is characterised by its action on $\fg$:
\begin{equation}
\label{eq_def_pi(X)}
\pi(X)=\partial_{t=0}\pi(e^{tX}),
\quad  X\in \fg.
\end{equation}
The infinitesimal action acts on the space $\cH_\pi^\infty$
of smooth vectors, that is, the space of vectors $v\in \cH_\pi$ such that 
 the mapping $G\ni x\mapsto \pi(x)v\in \cH_\pi$ is smooth.

We will use the following equivalent writings for the group Fourier transform of a function  $f\in L^1(G)$
  at $\pi$: 
$$
 \pi(f) \equiv \widehat f(\pi) \equiv \cF_G(f)(\pi)=\int_G f(x) \pi(x)^*dx.
$$ 
The operator $\pi(f)$ is bounded on $\cH_\pi$ with operator norm:
\begin{equation}
	\label{eq_pifnormL1normf}
	\|\pi(f)\|_{\sL(\cH_\pi)} \leq \|f\|_{L^1(G)}.
\end{equation}
 If $\pi_1,\pi_2$ are two equivalent representations of $G$ with $\pi_1=\mathbb U^{-1}\circ  \pi_2 \circ \mathbb U$ for some intertwining operator $\mathbb U$, then 
$$
\pi_1(f) =\mathbb U^{-1}\circ \pi_2(f)\circ \mathbb U .
$$

We denote by $\Gh$ the unitary dual of $G$,
that is, the unitary irreducible representations of $G$ modulo equivalence.
We may often allow ourselves to identify a unitary irreducible representation 
with its class in $\Gh$. Moreover, for $f\in L^1(G)$,  the measurable field of operators $\widehat f = \cF_G f = \{\pi(f), \pi\in \Gh \}$ is understood  modulo intertwiners.

\smallskip
 
When studying square integrable functions on a compact nilmanifold $M=\Gamma\backslash G$, we will consider the following representation and the following group Fourier transform.
\begin{ex}
\label{ex_Regrep}
The regular representation $R:G\to \sL(L^2(M))$ is defined via
$$
R(x_0)f(\dot x) = f(\dot x x_0), \quad f\in L^2(M), \ x_0\in G, \ \dot x\in M. 
$$

Let $\kappa\in L^1(G)$.
Then $R(\kappa)$ is the operator acting on $L^2(M)$ via
$$
R(\kappa) f (\dot x) = \int_G \kappa(y) R(y)^* f (\dot x) 
dy = 
 \int_G \kappa(y)  f (y^{-1}\dot x) 
dy.
$$
\end{ex}
Recall that $R$
decomposes into a discrete direct sum of representation $\pi\in \Gh$ with finite multiplicity $m(\pi)$;
the multiplicity $m(\pi)$ may in fact be described more precisely, see \cite{Richardson}.
Denoting by $\Gamma\backslash\Gh$ the set of these representation,
this means that $L^2(M)$ decomposes into closed $R(G)$-invariant closed vector subspaces:
\begin{equation}
	\label{eq_L2Mdec}
	L^2(M) = \oplus^\perp_{\pi  \in \Gamma\backslash\Gh} L^2_\pi(M), 
\end{equation}
and on each $L^2_\pi(M)$, the representation $R$ is unitarily equivalent to $m(\pi)$ copies of $\pi$.
In this paper, we will use $R$ and its decomposition only via the following statement which gives a better estimate than \eqref{eq_pifnormL1normf}:
\begin{lemma}
\label{lem_RegRep}
We continue with the setting of Example \ref{ex_Regrep} and the above notation. 
Then 
$$
\|R(\kappa) \|_{\sL(L^2(M))}
\leq \sup_{\pi \in \Gamma\backslash\Gh} \|\pi(\kappa)\|_{\sL(\cH_\pi)}.
$$	
\end{lemma}

\begin{proof}
	The  Hilbertian decomposition \eqref{eq_L2Mdec} implies that any $f\in L^2(M)$  decomposes as 
$$
f=\sum_{\pi  \in \Gamma\backslash\Gh} f_\pi, \quad f_\pi\in 
L^2_\pi(M), \quad\mbox{with}\quad
\|f\|_{L^2(M)}^2 = \sum_{\pi  \in \Gamma\backslash\Gh} \|f_\pi\|_{L^2(M)}^2 .
$$
We can also decompose 
$$
R(\kappa)f =
\sum_{\pi  \in \Gamma\backslash\Gh} R(\kappa) f_\pi,
\quad\mbox{with}\quad
R(\kappa) f_\pi\in L^2_\pi(M), 
$$ 
so 
$$
\|R(\kappa) f\|_{L^2(M)}^2 = \sum_{\pi  \in \Gamma\backslash\Gh} 
\| R(\kappa) f_\pi\|_{L^2(M)}^2.
$$
Since the representation $R$ on  $L^2_\pi(M)$ is unitarily equivalent to $m(\pi)$ copies of $\pi$, we have
$$
\| R(\kappa) f_\pi\|_{L^2(M)}\leq \|\pi(\kappa)\|_{\sL(\cH_\pi)}\|f_\pi\|_{L^2(M)}.
$$
Hence
\begin{align*}
\|R(\kappa)f\|_{L^2(M)}^2 
\leq  \sum_{\pi  \in \Gamma\backslash\Gh} \|\pi(\kappa)\|_{\sL(\cH_\pi)}^2
\|  f_\pi\|_{L^2(M)}^2 =\sup_{\pi\in \Gamma\backslash\Gh} \|\pi(\kappa)\|_{\sL(\cH_\pi)}^2 \sum_{\pi  \in \Gamma\backslash\Gh} 
\|  f_\pi\|_{L^2(M)}^2.
\end{align*}
We conclude with $\sum_{\pi  \in \Gamma\backslash\Gh} \|f_\pi\|_{L^2(M)}^2=\|f\|_{L^2(M)}^2$. 
\end{proof}

\subsubsection{The Plancherel measure}
 The unitary dual $\Gh$ is naturally equipped with a structure of a standard Borel space.
The Plancherel measure is the unique positive Borel measure $\mu$ 
on $\Gh$ such that 
for any $f\in C_c(G)$, we have:
\begin{equation}
\label{eq_plancherel_formula}
\int_G |f(x)|^2 dx = \int_{\Gh} \|\cF_G(f)(\pi)\|_{HS(\cH_\pi)}^2 d\mu(\pi).
\end{equation}
Here $\|\cdot\|_{HS(\cH_\pi)}$ denotes the Hilbert-Schmidt norm on $\cH_\pi$.
This implies that the group Fourier transform extends unitarily from 
$L^1(G)\cap L^2(G)$ to $L^2(G)$ onto the Hilbert space
$$
L^2(\Gh):=\int_{\Gh} \cH_\pi \otimes\cH_\pi^* d\mu(\pi),
$$
which we identify with the space of $\mu$-square integrable fields $\sigma$ on $\Gh$ with Hilbert norm 
$$
\|\sigma\|_{L^2(\Gh)}=\sqrt{\int_{\Gh} \|\sigma(\pi)\|_{HS(\cH_\pi)}^2 d\mu(\pi)}.
$$
Consequently \eqref{eq_plancherel_formula} holds for any $f\in L^2(G)$ and may be restated as
$$
\|f\|_{L^2(G)} = \|\widehat {f} \|_{L^2(\Gh)},
$$
this formula is called the Plancherel formula.
It is possible to give an expression for the Plancherel measure $\mu$, see \cite[Section 4.3]{corwingreenleaf}, although we will not need this explicit expression in this paper.
From this, the following inversion formula is deduced \cite{corwingreenleaf}:
\begin{equation}
\label{eq_FI}	
\forall x\in G,\quad
\int_{\Gh} \tr(\pi(x)\cF_G\kappa(\pi) )d\mu(\pi) 
=\kappa(x),
\end{equation}
for any continuous function $\kappa:G\to \bC$ satisfying  $\int_{\Gh} \tr|\cF_G\kappa(\pi)| d\mu(\pi)<\infty$.

\subsubsection{The von Neuman algebra and $C^*$-algebra of $G$}

The von Neumann algebra of the group $G$ may be realised as the von Neumann algebra $\sL(L^2(G))^G$ of $L^2(G)$-bounded operators commuting with the left-translations on $G$.
As our group is nilpotent, the $C^*$-algebra of the group is then the closure of the space of right-convolution operators with convolution kernels in the Schwartz space. 

Dixmier's full Plancherel theorem  \cite[Ch. 18]{Dixmier} states the the von Neumann algebra of $G$ can also be  realised  as
the space
$L^\infty (\Gh)$
of measurable fields of operators that are bounded, that is, 
of measurable fields of operators
 $\sigma=\{\sigma(\pi)\in \sL(\cH_\pi): \pi\in \Gh\}$ such that
 $$
 \exists C>0\qquad \|\sigma(\pi)\|_{\sL(\cH_{\pi})} \leq C 
 \ \mbox{for} \ d\mu(\pi)\mbox{-almost all} \ \pi \in \Gh.
 $$
  The smallest of such constant $C>0$ is the norm $\|\sigma\|_{L^{\infty}(\Gh)}$ of $\sigma$  
  in $L^\infty(\Gh)$. 
 Similarly,  the $C^*$-algebra of the group $C^*(G)$ is then 
the closure of $\cF_G\cS(G)$ for the $L^\infty$-norm, 
and 
$L^\infty (\Gh)$ is the von Neumann algebra generated by the $C^*$-algebra of the group. 

The isomorphism between 
the von Neumann algebras $L^\infty(\Gh)$ and $\sL(L^2(G))^G$ may be described as follows.
We check readily that $f\mapsto \cF_G^{-1} \sigma \widehat f$ is in $\sL(L^2(G))^G$  if $\sigma\in L^\infty(\Gh)$.
The converse is given by \cite[Ch. 18]{Dixmier}:
 if $T\in \sL(L^2(G))^G$, then 
there exists  a unique field $\widehat T \in L^\infty (\Gh)$ such that 
$T$ and $f\mapsto \cF_G^{-1} \widehat T \widehat f$ coincide; moreover, 
\begin{equation}
\label{eq_LinftyGh_L2bdd}
\|\widehat T\|_{L^\infty(\Gh)}=\|T\|_{\sL(L^2(G))}	.
\end{equation}
By the Schwartz kernel theorem, 
the operator $T$ admits a distributional convolution kernel  $\kappa\in \cS'(G)$.
We may also write $\widehat \kappa=\widehat T$ and call this field  the group Fourier transform of $\kappa$ or of $T$. 
It extends the previous definition of the group Fourier transform on $L^1(G)$ and $L^2(G)$. 

\section{Graded nilpotent Lie groups and their nilmanifolds}
\label{sec:Graded_Lie_group}

In the rest of the paper, 
we will be concerned with graded Lie groups and their Rockland operators. 
References on this subject for this section and the next ones include \cite{folland+stein_82} and 
\cite{R+F_monograph}.
Most of the analysis presented here may be already known to experts, but this section will help with setting up the notation. 

\subsection{Graded nilpotent Lie group}
\label{subsec_gradedG}

A graded Lie group $G$  is a connected and simply connected 
Lie group 
whose Lie algebra $\fg$ 
admits an $\bN$-gradation
$\fg= \oplus_{\ell=1}^\infty \fg_{\ell}$
where the $\fg_{\ell}$, $\ell=1,2,\ldots$, 
are vector subspaces of $\fg$
 all equal to $\{0\}$ except a finite number, 
and satisfying 
$[\fg_{\ell},\fg_{\ell'}]\subset\fg_{\ell+\ell'}$
for any $\ell,\ell'\in \bN$.

This implies that the group $G$ is nilpotent.
Examples of such groups are the Heisenberg group
 and, more generally,
all stratified groups (which by definition correspond to the case $\fg_1$ generating the full Lie algebra $\fg$).

For any $r>0$, 
we define the  linear mapping $D_r:\fg\to \fg$ by
$D_r X=r^\ell X$ for every $X\in \fg_\ell$, $\ell\in \bN$.
Then  the Lie algebra $\fg$ is endowed 
with the family of dilations  $\{D_r, r>0\}$
and becomes a homogeneous Lie algebra in the sense of 
\cite{folland+stein_82}.
We re-write the set of integers $\ell\in \bN$ such that $\fg_\ell\not=\{0\}$
into the increasing sequence of positive integers
 $\upsilon_1,\ldots,\upsilon_n$ counted with multiplicity,
 the multiplicity of $\fg_\ell$ being its dimension.
 In this way, the integers $\upsilon_1,\ldots, \upsilon_n$ become 
 the weights of the dilations.

 We construct a basis $X_1,\ldots, X_n$  of $\fg$ adapted to the gradation,
by choosing a basis $\{X_1,\ldots X_{n_1}\}$ of $\fg_1$ (this basis is possibly reduced to $\emptyset$), then 
$\{X_{n_1+1},\ldots,  X_{n_1+n_2}\}$ a basis of $\fg_2$
(possibly $\{0\}$ as well as the others).
We have $D_r X_j =r^{\upsilon_j} X_j$, $j=1,\ldots, n$.

 The associated group dilations are defined by
$$
D_r(x)=
rx
:=(r^{\upsilon_1} x_{1},r^{\upsilon_2}x_{2},\ldots,r^{\upsilon_n}x_{n}),
\quad x=(x_{1},\ldots,x_n)\in G, \ r>0.
$$
In a canonical way,  this leads to the notions of homogeneity for functions, distributions and operators and we now give a few important examples. 

The Haar measure is $Q$-homogeneous, where
$$
Q:=\sum_{\ell\in \bN}\ell \dim \fg_\ell=\upsilon_1+\ldots+\upsilon_n,
$$
 is called the homogeneous dimension of $G$.

Identifying the element of $\fg$ with left invariant vector fields, 
each  $X_j$ is a $\upsilon_j$-homogeneous differential operator of degree one. More generally, the differential operator 
$$
X^{\alpha}=X_1^{\alpha_1}X_2^{\alpha_2}\cdots
X_{n}^{\alpha_n}, \quad \alpha\in \bN_0^n$$
is homogeneous with degree
$$
[\alpha]:=\alpha_1 \upsilon_1 + \cdots + \alpha_n \upsilon_n.
$$

The unitary dual $\Gh$ inherits  a dilation from the one on $G$ \cite[Section 2.2]{FFPisa}: we denote by $r\cdot \pi$ the element of $\Gh$ obtained from $\pi$ through dilatation by $r$, that is, $r \cdot \pi(x) = \pi (rx)$, 
$r>0$ and $x\in G$.
The Plancherel measure is $Q$-homogeneous for these dilations in the sense that 
\begin{equation}
	\label{eq_muQhom}
	\int_{\Gh} \tr \sigma(r \cdot \pi)
d\mu(\pi) = r^{-Q} \int_{\Gh} \tr \sigma(\pi)
d\mu(\pi),
\end{equation}
for any symbol $\sigma$ such that $\int_{\Gh} \tr |\sigma(\pi)|
d\mu(\pi)$ is finite.

An important class of homogeneous maps are the homogeneous quasi-norms, 
that is, a $1$-homogeneous non-negative map $G \ni x\mapsto \|x\|$ which is symmetric and definite in the sense that $\|x^{-1}\|=\|x\|$ and $\|x\|=0\Longleftrightarrow x=0$.
In fact, all the homogeneous quasi-norms are equivalent in the sense that if $\|\cdot\|_1$ and $\|\cdot\|_2$ are two of them, then 
$$
\exists C>0 \qquad \forall x\in G
\qquad C^{-1} \|x\|_1 \leq \|x\|_2 \leq C \|x\|_1.
$$
Examples may be constructed easily, such as
\begin{equation}
\label{eq_qnormp}
|(x_1,\ldots,x_n)|_p=\big  (\sum_{j=1}^n |x_j|^{p/\upsilon_j}\big)^{1/p},
\ \mbox{for any}\ p\geq 1.
\end{equation}

\subsection{Rockland symbols and operators on $G$}
\label{subsec_cR}
Let us briefly review the definition and main properties of positive Rockland operators. 
References on this subject include \cite{folland+stein_82} and 
\cite{R+F_monograph}.

A Rockland operator
 $\cR_G$ on $G$ is 
a left-invariant differential operator 
which is homogeneous of positive degree and satisfies the Rockland condition, that is, 
for each unitary irreducible representation $\pi$ on $G$,
except for the trivial representation, 
the operator $\pi(\cR_G)$ is injective on the space  $\cH_\pi^\infty$ of smooth vectors of the infinitesimal representation.
Equivalently, the symbol $\widehat \cR = \{\pi(\cR_G): \pi\in \Gh\}$ is said to be Rockland.

Recall
that Rockland operators are hypoelliptic.
In fact, they are equivalently characterised as the left-invariant differential operators which are hypoelliptic. 
If this is the case, then lower order term may be added in the sense that the operator $\cR_G + \sum_{[\alpha]< \nu}c_\alpha X^\alpha$, where $c_\alpha\in \bC$ and $\nu$ is the homogeneous degree of $\cR_G$, is hypoelliptic.

A Rockland operator is \emph{positive} when 
$$
\forall f \in \cS(G),\qquad
\int_G \cR_G f(x) \ \overline{f(x)} dx\geq 0.
$$

Positive Rockland operators may be viewed as generalisations of the natural sub-Laplacians on Carnot groups and 
they are easily constructed on any graded Lie group, see \cite[Section 4.1.2]{R+F_monograph}:

\begin{ex}
\label{ex_Rockland}
\begin{enumerate}
\item Any sub-Laplacian with the sign convention $-(X_1^2+\ldots+X_{n_1}^2)$ of a stratified Lie group  is a positive Rockland operator; here $X_1,\ldots, X_{n_1}$ form a basis of the first stratum $\fg_1$, and we identify them with the corresponding left-invariant vector fields.
\item 
More generally, 
$-\sum_{1\leq i,j\leq n_1} c_{i,j} X_i X_j$
where $(c_{i,j})$ is a positive definite matrix 
is a sub-Laplacian and a positive Rockland operators. 
Indeed, we define the scalar product induced by the $X_1,\ldots, X_{n_1}$ on $\fg_1$ and we 
 consider an orthonormal basis $Z_1,\ldots, Z_{n_1}$ of eigenvectors of $(c_{i,j})$ of $\fg_1$, with corresponding positive eigenvalues $\lambda_1,\ldots,\lambda_{n_1}$; we may now write 
 $$
 -\sum_{1\leq i,j\leq n_1} c_{i,j} X_i X_j
=-\sum_{1\leq i,j\leq n_1} \lambda_i Z_i ^2
=-\sum_{1\leq i,j\leq n_1} (\sqrt{\lambda_i} Z_i) ^2.
 $$

	\item Let $G$ be a graded Lie group, and $X_1,\ldots, X_n$ an adapted basis to its graded Lie algebra.
	If $\nu_0$ is a common multiple of the weights $\upsilon_j$, $j=1,\ldots,n$ of the dilations,  then 
	$\cR = \sum_{j=1}^{n} (-1)^{\nu_0 / \upsilon_j} X_j^{2 \nu_0 /\upsilon_j}$ is a positive Rockland operator of homogeneous degree $2\nu_0$.  It  is also symmetric: $\cR^t = \cR$.
	
\end{enumerate}
\end{ex}

The above examples will be generalised  in Proposition \ref{prop_exRgeneral}.

\medskip

A positive Rockland operator is essentially self-adjoint on $L^2(G)$ and we keep the same notation for their self-adjoint extension.
Its spectrum is $\sp(\cR_G)$ included in $[0,+\infty)$ and the point 0 may be neglected in its spectrum \cite{FFPisa}.

For each unitary irreducible representation $\pi$ of $ G$, 
the operator 
$\pi(\cR_G)$ is  essentially self-adjoint on $\cH^\infty _\pi$
and we keep the same notation for this self-adjoint extension.
Its spectrum $\sp(\pi(\cR_G))$ is a discrete subset of $(0,\infty)$ if the representation $\pi$  is not trivial, i.e. $\pi\not =1_{\Gh}$,  while $\pi(\cR_G)=0$ if $\pi=1_{\Gh}$.

Let us denote by $E$ and $E_\pi$ the spectral measures respectively
of 
\begin{equation}
\label{eq_spectralmeas}
\cR_G = \int_\bR \lambda dE_\lambda
\quad\mbox{and}\quad 
\pi(\cR_G) = \int_\bR \lambda dE_\pi(\lambda), \ \pi\in \Gh.	
\end{equation}
Then $\widehat E(\pi)=E_\pi$ in the sense that 
for any interval $I\subset \bR$, 
the group Fourier transform 
$\widehat {E(I)}$ of the projection $E(I)\in \sL(L^2(G))^G$ coincides with the field $\{\pi(E(I))=E_\pi(I), \pi\in \Gh\}$.

If $\psi:\bR^+\to \bC$ is a measurable function,
the spectral multiplier $\psi(\cR_G) = \int_\bR \psi(\lambda) dE_\lambda$ is well defined as a possibly unbounded operator on $L^2(G)$.
If the domain of $\psi(\cR_G)$  contains $\cS(G)$ 
and defines a continuous map $\cS(G)\to \cS'(G)$, then 
it is invariant under left-translation. Its convolution kernel $\psi(\cR_G)\delta_0 \in \cS'(G)$ (in the sense of Section \ref{subsec_aboutGnilp}) satisfies the following homogeneity property:
\begin{equation}
\label{eq_homogeneitypsiR}	
 \psi(r^\nu \cR_G) \delta_0 (x) =r^{-Q} \psi(\cR_G)\delta_0(r^{-1}x),
 \quad x\in G.
\end{equation}
Furthermore, for each unitary irreducible representation $\pi$ of $G$,  
 the domain of the operator
$ \psi(\pi(\cR_G))= \int_\bR \psi(\lambda) dE_\pi (\lambda) $
contains $\cH_\pi^\infty$ and we have
$$
\widehat  {\psi(\cR_G)}(\pi)=
 \psi (\pi(\cR_G)) .
$$

\smallskip

The following statement is the famous result due to Hulanicki \cite{hulanicki}:
\begin{theorem}[Hulanicki's theorem]
\label{thm_hula}
	Let $\cR_G$ be a positive Rockland operator on $G$.
	If $\psi\in \cS(\bR)$ then the convolution kernel $\psi(\cR_G)\delta_0$ of the operator $\psi(\cR_G)$ is a Schwartz function, i.e.
	 $\psi(\cR_G)\in \cS(G)$.
\end{theorem}

For instance,  the heat kernels 
$$
p_t:=e^{-t\cR_G}\delta_0, \quad t>0,
$$
 are Schwartz - although this property is in fact used in the proof of Hulanicki's Theorem. 

\medskip

The following result describes the isometry 
  $\psi\mapsto \psi(\cR_G)\delta_0$  from $L^2((0,\infty), c_0 \lambda^{Q/2} d\lambda/\lambda)$ to  $L^2(G)$ for some constant $c_0>0$.
This was mainly obtained by Christ for sub-Laplacians on stratified groups \cite[Proposition 3]{Christ91} and readily extended   to positive Rockland operators in \cite{nilmanifold}, see also \cite{Martini}   and the references to Hulanicki's works therein. 

\begin{theorem}
\label{thm_christ}
	Let $\cR_G$ be a positive Rockland operator of homogeneous degree $\nu$ on $G$. 
	If the measurable function  $\psi:\bR^+\to \bC$ is in  $L^2(\bR^+,  \lambda^{Q/\nu} d\lambda/\lambda)$, 
	then $\psi(\cR_G)$  defines a continuous map $\cS(G)\to \cS'(G)$ whose convolution kernel  $\psi(\cR_G)\delta_0$ is in $L^2(G)$. Moreover,   we have 
$$
\|\psi(\cR_G)\delta_0\|_{L^2(G)}^2
 	=c_0\int_0^\infty |\psi (\lambda)|^2  \lambda^{\frac Q \nu} \frac{d\lambda}{\lambda},
 $$
 where $c_0 = c_0(\cR_G)$ is a positive constant of $\cR_G$ and $G$.  
 
 Consequently, 
 we have for any $\psi\in \cS(\bR)$
$$
\psi(\cR_G)\delta_0(0)
 	=c_0\int_0^\infty \psi (\lambda)  \lambda^{\frac Q \nu} \frac{d\lambda}{\lambda}.
 $$

\end{theorem}

\subsection{Sobolev spaces on $G$ and $M$}

If $\cR_G$ be a positive Rockland operator of homogeneous degree $\nu$ and $s\in \bR$, 
then we define the \emph{Sobolev spaces} $L^2_{s}(G)$ as the completion of the domain $\dom (\id+\cR_G)^\frac s\nu$
of $(\id+\cR_G)^\frac s\nu$, 
for the Sobolev norm 
$$
\|f\|_{L^2_s(G),\cR}:= \|(\id+\cR_G)^\frac s\nu f\|_{L^2(G)}.
$$
They satisfy  the following natural properties (see \cite[Section 4.4]{R+F_monograph}):
\begin{theorem}
\label{thm_sobolev_spacesG}
\begin{enumerate}
\item The Sobolev spaces $L^2_{s}(G)$ are independent of a choice of a positive Rockland operator $\cR_G$.
Different choices of positive Rockland operators yields equivalent Sobolev norms, and these equip 
the  spaces $L^2_s(G)$, $s\in \bR$, of a structure of Banach spaces.
These Hilbert spaces  satisfy the classical properties of duality (in the sense of  \cite[Lemma 4.4.7]{R+F_monograph}) and interpolation 
(in the sense of 
 \cite[Theorem 4.4.9 and Proposition 4.4.15]{R+F_monograph}).
 \item For $s,a\in \bR$, 
the operator $(\id+\cR_G)^{\frac a \nu}$ maps continuously 
$L^2_s(G)$ to $L^2_{s-a}(G)$.
 f
\label{item_thm_sobolev_spaces_continuity_Xalpha}
For any $\alpha\in \bN_0^n$, $X^\alpha$ maps continuously 
$L^2_s(G)$ to $L^2_{s-[\alpha]}(G)$
 for any $s\in \bR$.
Moreover, if $s\in \bN$ is a  common multiple of the weights $\upsilon_1,\ldots,\upsilon_n$ of the dilations, 
then  
$f\mapsto \sum_{ [\alpha]\leq s}
\|X^\alpha f\|_{L^2(G)}$ is an equivalent norm on $L^2_s(G)$.

 \item
 We have the continuous inclusions
$$
s_1 \geq s_2 \ \Longrightarrow \ 
L^2_{s_1} (G)\subset L^2_{s_2} (G) ,
  $$
  and
  $$
 L^2_s(G)\subset C_b(G), \ s>Q/2, \qquad \mbox{(Sobolev embeddings) }
$$
where $C_b(G)$ denotes the Banach space of continuous and bounded functions on $G$.
\end{enumerate}
\end{theorem}

Many parts of the theorem above are a consequence  of the following property that we will also use later on:
\begin{lemma}
\label{lem_R1R2bdd}
If $\cR_1$ and $\cR_2$ are two positive Rockland operators of homogeneous degrees $\nu_1$ and $\nu_2$ on $G$,
then for any $a\in \bR$, 
the operator
$(\id+\cR_1)^{a/\nu_1} (\id +\cR_2)^{-a/\nu_2}$ is well defined and extends naturally into  a bounded operator on $L^2(G)$.
\end{lemma}

\begin{remark}
\label{rem_lem_R1R2bdd}
	The proof of Lemma \ref{lem_R1R2bdd} shows that 
	the $L^2$-operator norm of $(\id+\cR_1)^{a/\nu_1} (\id +\cR_2)^{-a/\nu_2}$ is bounded by a constant $C(G, \cR_1,\cR_2, a)$ depending on the structural constants:
$$
\|(\id+\cR_1)^{a/\nu_1} (\id+\cR_2)^{-a/\nu_2}\|_{\sL(L^2(G))}
\leq C(G, \cR_1,\cR_2, a).
$$
We can be even more precise by fixing an adapted basis $(X_1,\ldots, X_n)$ and writing  
$$
\cR_i=\sum_{[\alpha_i]=\nu_i}c_{\alpha_i,i} X^{\alpha_i},
$$
$i=1,2$.
The constant may be bounded by 
$$
C(G, \cR_1,\cR_2, a) \leq C\left (\max \left( [|a|]+1,(c_{\alpha_i,i})_{[\alpha_i]=\nu_i, i=1,2}, \dim G, \upsilon_{\dim G}, \|[\cdot,\cdot]\|_{\sL(\fg\times \fg, \fg)}\right)\right ),
$$
where $C:(0,\infty)\to (0,\infty)$ is an increasing function
and $ \|[\cdot,\cdot]\|_{\sL(\fg\times \fg, \fg)}$ denotes the operator norm  of the bilinear map $[\cdot,\cdot]:\fg\times \fg \to \fg$ with $\fg$ equipped with the norm that made $(X_1,\ldots, X_n)$ orthonormal.
However, the function $C$ does not depend on the group $G$ or the specific choice of $(X_1,\ldots, X_n)$.
\end{remark}

In order to distinguish the Sobolev spaces $L^2_s(G)$ on the graded group $G$ from the usual Sobolev spaces on the underlying $\bR^n$, 
we  denote by $H^s=H^s(\bR^n)$ the Eulidean Sobolev spaces on $\bR^n$.
The spaces $H^s$ and $L^2_s(G)$ are not comparable globally
(we assume that $G$ is not abelian), but they are locally.
Let us recall the definition of local Sobolev space:
\begin{definition}
	\begin{enumerate}
		\item On $\bR^n$, for any $s\in \bR$, the local Sobolev space $H^s_{loc}=H^s_{loc}(\bR^n)$  is the Fr\'echet space of distributions $f\in \cD'(\bR^n)$ that are locally in $H^s$, that is, such that for any $\chi\in \cD(\bR^n)$, we have $f\chi \in H^s$ . 
		\item Similarly, on a graded Lie group $G$, for any $s\in \bR$, the local Sobolev space $L^2_{s,loc}(G)$  is the Fr\'echet space of distributions $f\in \cD'(G)$ that are locally in $L^2_s(G)$, that is, such that for any $\chi\in \cD(G)$, we have $f\chi \in L^2_s(G)$. 
	\end{enumerate}
\end{definition}

We have the continuous inclusions
$$
H^s_{loc} \subset L^2_{s\upsilon_1,loc} 
\qquad\mbox{and}\qquad
L^2_{s,loc} \subset H^{s/\upsilon_n}_{loc};
$$
recall that $\upsilon_1\leq \ldots\leq \upsilon_n$
are the dilation's weights in increasing order.

\smallskip

Let us define the Sobolev spaces adapted to a compact nilmanifold:
\begin{definition}
Let $s\in \bR$. The Sobolev space $L^2_s(M)$ on the compact nilmanifold $M = \Gamma \backslash G$ is the space of distributions $f\in \cD'(M)$ whose corresponding $\Gamma$-periodic distribution $f_G\in \cS'(G)$ is in $L^2_{s,loc}(G)$.
\end{definition}

Routine checks and  Theorem \ref{thm_sobolev_spacesG} imply the following properties of the Sobolev spaces on $M$. Recall that if a differential $T$ is left invariant (for instance if $T=\cR_G$ a positive Rockland operator on $G$), then we associate the corresponding differential operator $T_M$ on $M$,  see Section \ref{subsubsec_opGM}.
Recall also that  fundamental domains for $\Gamma \backslash G$ are described in \cite[Section 5.3]{corwingreenleaf} or \cite[Proposition 2.4]{nilmanifold}.
\begin{proposition}
\label{prop_sobolev_spacesM}
Let $M=\Gamma\backslash G$ be a compact nilmanifold. 
	\begin{enumerate}
\item Let $\chi\in \cD(G)$ such that $\chi\equiv 1$ on a fundamental domain of $M$ in $G$. 
For any $s\in \bR$ and choice of norm $\|\cdot\|_{L^2_s(M)}$ on $L^2_s(G)$, the quantity
$$
\|f\|_{L^2_{s},\chi}:= \| f_G \chi\|_{L^2_s(G)}
$$
defines a Sobolev norm on $L^2_s(M)$. 
Different choices of $\chi$ and $\|\cdot\|_{L^2_s(M)}$ yield equivalent  norms on $L^2_s(M)$, and these equip 
the  spaces $L^2_s(M)$, $s\in \bR$, of a structure of Hilbert spaces.
These Hilbert spaces  satisfy the  properties of duality  and interpolation in the same sense as in Theorem \ref{thm_sobolev_spacesG}.
 \item
 \label{item_SobMcompres} If $\cR_G$ is a positive Rockland operator on $G$ of homogeneous degree $\nu$, then 
the operator $(\id+\cR_M)^{\frac a \nu}$ maps continuously 
$L^2_s(M)$ to $L^2_{s-a}(M)$ 
for any $s,a\in \bR$. 
Moreover, 
$$
\|f\|_{L^2_{s},\cR_M}:= \| (\id+\cR_M)^{s/\nu}f\|_{L^2(M)}
$$
defines an equivalent Sobolev norm on $L^2_s(M)$.

\item 
For any $\alpha\in \bN_0^n$, $X_M^\alpha$ maps continuously 
$L^2_s(M)$ to $L^2_{s-[\alpha]}(M)$
 for any $s\in \bR$.
Moreover, if $s\in \bN$ is a  common multiple of the weights $\upsilon_1,\ldots,\upsilon_n$ of the dilations, 
then  
$f\mapsto \sum_{ [\alpha]\leq s}
\|X_M^\alpha f\|_{L^2(G)}$ is an equivalent norm on $L^2_s(M)$.

 \item
 We have the continuous inclusions
 $$
s_1 \geq s_2 \ \Longrightarrow \ 
L^2_{s_1} (M)\subset L^2_{s_2} (M) ,
  $$
  and
  $$
 L^2_s(M)\subset C(M), \ s>Q/2,\qquad \mbox{(Sobolev embeddings) }
$$
where $C(M)$ denotes the Banach space of continuous functions on $M$.
Moreover, if $s_1>s_2$, the above embedding is compact. 
\end{enumerate}
\end{proposition}

We also have the following properties for negative powers of $\id+\cR_M$:
 \begin{proposition}
 \label{prop_I+cR_comptrHS}
 	Let $\cR_G$ be a positive Rockland operator on $G$ of homogeneous degree $\nu$.
 	Let $\cR_M$ be the corresponding operator on $M$. 
 	Then  the operator $(\id+\cR_M)^{-1}$ is compact $ L^2_{s}(M)\to L^2(M)$ for any $s<\nu.$
 	Moreover, $(\id+\cR_M)^{-s/\nu}$ is trace-class for any $s>Q$ and Hilbert-Schmidt class for $s>Q/2$.  
 \end{proposition}
\begin{proof}
The compactness of $(\id+\cR_M)^{-1}$   follows from \cite[Section 3]{nilmanifold}  where it is proved that the spectrum of $\cR_M$ is discrete and its eigenspaces are finite dimensional. 

By functional calculus, the operator $e^{-t\cR_M}$ are non-negative, and we have:
$$
0\leq \tr (\id+\cR_M)^{-\frac s \nu}
\leq
\frac 1{\Gamma (s/\nu)}\int_0^\infty t^{\frac s \nu -1}
e^{-t}\tr ( e^{-t\cR_M}) dt.
$$
By \cite[Section 4]{nilmanifold}, 
there exists $C>0$ such that  (2 here is arbitrary)
$$
\forall t\in (0,2]\qquad
0\leq \tr (e^{-t\cR_M})\leq C t^{Q/\nu}.
$$
As $\|e^{-t\cR_M}\|_{\sL(L^2(M))}\leq 1$ for any $t>0$ by functional calculus, we have for any $t\geq 1$
$$
\tr (e^{-t\cR_M})
\leq 
\|e^{-(t-1)\cR_M}\|_{\sL(L^2(M))}
\tr (e^{-\cR_M})
\lesssim 1.
$$
The integral formula above allows us to conclude that $\tr (\id+\cR_M)^{-\frac s \nu}$ is finite for $s>Q$.
As $\|(\id+\cR_M)^{-\frac s \nu}\|_{HS(L^2(M))}^2 = \tr (\id+\cR_M)^{-\frac {2s} \nu}$, we obtain the result regarding the Hilbert-Schmidt classes. 
\end{proof}

\subsection{Further examples of positive Rockland operators}
\label{subsec_furtherexR}
We can generalise Example \ref{ex_Rockland} in the following way:

\begin{proposition}
\label{prop_exRgeneral}	
\begin{enumerate}
	\item Let $G$ be a stratified Lie group and $X_1,\ldots,X_{n_1}$ a basis of the first stratum. 
	Fix $\nu_1\in \bN$ and set $\nu'_1=\#\{\alpha\in \bN_0^{n_1} \ : \ |\alpha|=\nu_1\}$.
	Let $a_{\alpha,\beta}\in \bR$ with $\alpha,\beta\in \bN_0^{n_1}$, $|\alpha|=|\beta|=\nu_1$. 
If the matrix $A:=(a_{\alpha,\beta})_{\alpha,\beta}$ is non-negative, then 
the differential operator 
$$
\cR_A:=\sum_{|\alpha|=|\beta|=\nu_1} a_{\alpha,\beta} (X^\beta)^t X^\alpha 
$$
 is symmetric, non-negative and homogeneous of degree $2\nu_1$.
If furthermore  $A$ is positive definite, 
	then 
	$\cR_A$ is a  positive Rockland operator.
	\item Let $G$ be a graded Lie group, and $X_1,\ldots, X_n$ an adapted basis to its graded Lie algebra.
	Fix two common multiple $\nu_0,\nu_1$  of the weights $\upsilon_j$, $j=1,\ldots,n$ of the dilations.
	Set $Y_j:= X_j^{ \nu_0 /\upsilon_j}$, $j=1,\ldots, n$ and set $\nu'_1=\#\{\alpha\in \bN_0^{n} \ : \ [\alpha]=\nu_1\}$.
	Let $a_{\alpha,\beta}\in \bR$ with $\alpha,\beta\in \bN_0^{n}$, $[\alpha]=[\beta]=\nu_1$. 
	If the matrix $A:=(a_{\alpha,\beta})_{\alpha,\beta}$ is non-negative, then 
the differential operator 
$$
\cR_A:=
\sum_{[\alpha]=[\beta]=\nu_1}
a_{\alpha,\beta} (X^\beta)^t X^\alpha 
$$
 is symmetric, non-negative and homogeneous of degree $\nu_1 \nu_0$.
If furthermore  $A$ is positive definite, 
	then 
	$\cR_A$ is a  positive Rockland operator.
\end{enumerate}	
\end{proposition}

\begin{proof}
Let us prove Part (1). Clearly, $\cR_A$ is a homogeneous differential operator of homogeneous degree $2\nu_1$.
We compute the formal transpose of $\cR_A$:
$$
\cR_A^t
=
\sum_{|\alpha|=|\beta|=\nu_1} a_{\alpha,\beta} (X^\alpha)^t X^\beta 
=\cR_A,
$$
so  $\cR_A$ is symmetric. 
As $A$ is non-negative, there exists an orthogonal $\nu'_1\times\nu'_1$-matrix $P$ such that $PAP^{-1} = \diag(\lambda_1,\ldots,\lambda_{\nu_1})$ with $\lambda_j\geq 0$, 
and we have for any $v=(v_\alpha)\in \bR^{\nu'_1}$
$$
\sum_{|\alpha|=|\beta|=\nu_1} a_{\alpha,\beta} v_\alpha v_\beta =
\sum_{j=1}^{\nu_1} \lambda_j [Pv]_{\nu_1 e_j}^t [Pv]_{\nu_1 e_j},
$$
where $e_j =(0,\ldots, 0,1,0,\ldots ,0) \in \bN_0^{\nu_1}$ is the index with 0 everywhere except at the $j$-th place where it is equal to 1; 
note that $|\nu_1 e_j|=\nu_1$.
We denote by $Z$ the $\nu'_1$-tuple of left-invariant differential operators given as the $P$-linear combination of the $X^\alpha$, $|\alpha|=\nu_1$, i.e. $Z:= P(X^\alpha)_\alpha$.
We may re-write $\cR_A$ as 
$$
\cR_A = \sum_{j=1}^{\nu_1} \lambda_j [Z]_{\nu_1 e_j}^t [Z]_{\nu_1 e_j},
$$
where  $[Z]_{\nu_1 e_j}$ the invariant differential operator corresponding to the index $\nu_1 e_j$ in $Z$.

We can now readily check that $\cR_A$ is non-negative:
$$
\int_G \cR_A f (x) \ \overline{f(x)} dx
=
\sum_{j=1}^{\nu_1} \lambda_j \|
[Z]_{\nu_1 e_j}f \|_{L^2(G)}^2.
\geq 0
$$
Moreover, for any $\pi\in \Gh$, we have
$$
\pi(\cR_A) = \sum_{j=1}^{\nu_1} \lambda_j \pi([Z]_{\nu_1 e_j})^* \pi( [Z]_{\nu_1 e_j})
$$
so if $v\in \cH_\pi^\infty$
with $\pi(\cR_A)v=0$, 
then 
$$
0=(\pi(\cR_A)v,v) _{\cH_\pi}
=\sum_{j=1}^{\nu_1} \lambda_j \|\pi([Z]_{\nu_1 e_j}) v\|^2.
$$
We now assume  $A>0$, which means that every $\lambda_j$ is positive.
For any $\pi\in \Gh\setminus\{1\}$ and $v$ as above, we have $\pi([Z]_{\nu_1 e_j}) v=0$ for $j=1,\ldots,\nu_1$.
As we may write matricially
$$
(X^\alpha)_{|\alpha|=\nu_1} = A^{-1} P^{-1} \diag (\lambda_1,\ldots,\lambda_{\nu_1}) Z
= A^{-1} P^{-1} \diag ([Z]_{\nu_1 e_1},\ldots, [Z]_{\nu_1 e_{\nu_1}})
$$
and similarly for $\pi(X)^\alpha$ and the entries of $\pi(Z)$, this implies that $\pi(X^\alpha)v=0$ for any $\alpha$ with $|\alpha|=\nu_1$, in particular for  $\pi(X_j)^{\nu_1} v=0$, $j=1,\ldots n_1$. Proceeding as in \cite[Lemma 2.18]{R+F_monograph}, this implies recursively $\pi(X_j)^{\nu_1} v=0$, and therefore $v=0$ by proceeding as in \cite[Lemma 2.17]{R+F_monograph}. This shows Part (1). Part (2) is proved in a similar manner. 
	\end{proof}

The canonical examples will correspond to the case where $\nu_1=1$ and the coefficients $a_{\alpha,\beta}$ form the identity matrix $\id$.
Indeed, we then recognise the sub-Laplacian 
$$
\cR_\id=-\sum_{j=1}^{n_1} X_j^2
$$
 in the stratified case, 
 and  
 $$
 \cR_\id=\sum_{j=1}^{n} (-1)^{\nu_0 / \upsilon_j} X_j^{2 \nu_0 /\upsilon_j}
 $$
in the graded case 
as in Example \ref{ex_Rockland} (3).
We can  be more precise in the $L^2$-boundedness between two positive Rockland operators described in Lemma \ref{lem_R1R2bdd} (see also Remark \ref{rem_lem_R1R2bdd}) with respect to the case of the identity matrix: 

\begin{corollary}
\label{cor_exRgeneral}
We continue with the setting of Proposition \ref{prop_exRgeneral} with $A>0$.
\begin{enumerate}
	\item 	We consider the case of 
	a stratified Lie group $G$. 
For any $a\in \bR$, the maximum of the operator norms 
$$ \|(\id+ \cR_A)^{a/\nu_1} (\id+\cR_\id)^{-a/\nu_1}\|_{\sL(L^2(G))},\  \| (\id+\cR_\id)^{a/\nu_1} (\id+\cR_A)^{-a/\nu_1}\|_{\sL(L^2(G))}
 $$
 is bounded 
 by $
 C(\max ([|a|]+1,\lambda_{A,1},\lambda_{A,n_1}))$
  where $\lambda_{A,1}$ and $\lambda_{A,n_1}$ are the lowest and highest eigenvalue of $A$, and 
 $C:(0,\infty)\to (0,\infty)$ is an increasing function that depend on the structure of the graded Lie group $G$ and on the scalar product  that makes $(X_1,\ldots, X_{n_1})$ orthonormal. 

Above, $\cR_\id= \sum_{|\alpha|=\nu_1} (X^\alpha)^t X^\alpha $ corresponds to the case of 	$A=\id$ being   the identity matrix.
	\item We consider the case of 
	a graded Lie group $G$. For any $a\in \bR$, the maximum of the operator norms 
$$ \|(\id+ \cR_A)^{a/\nu_1} (\id+\cR_\id)^{-a/\nu_1}\|_{\sL(L^2(G))},\  \| (\id+\cR_\id)^{a/\nu_1} (\id+\cR_A)^{-a/\nu_1}\|_{\sL(L^2(G))}
 $$
 is bounded 
 by $
 C(\max ([|a|]+1,\lambda_{A,1},\lambda_{A,n})),
 $
where $\lambda_{A,1}$ and $\lambda_{A,n}$ are the lowest and highest eigenvalue of $A$, and 
 $C:(0,\infty)\to (0,\infty)$ is an increasing function depending on the structure of the graded Lie group $G$ and the scalar product that makes $(X_1,\ldots, X_{n})$ orthonormal. 
 Above, $\cR_\id= \sum_{[\alpha]=\nu_1} (Y^\alpha)^t Y^\alpha $ corresponds to the case of 	$A=\id$ being   the identity matrix.
\end{enumerate}	
\end{corollary}

\section{Pseudo-differential calculi on $G$ and $M$}\label{Sec:calculi_symbols}

In this section, we recall the construction of the pseudo-differential calculus on a graded nilpotent Lie group $G$ as presented in \cite{R+F_monograph}.
We also explain how it induces a pseudo-differential calculus on its compact nilmanifold $M$.

\subsection{Symbol classes on $G\times \Gh$}

\subsubsection{Invariant symbols and $L^\infty_{a,b}(\Gh)$}
\begin{definition}
An invariant symbol $\sigma$ on $G$ is a  measurable field of operators $\sigma= \{\sigma (\pi):\cH_\pi^{+\infty} \to \cH_\pi^{+\infty} : \pi\in \widehat G\}$ over $\widehat G$. 
Here, $\cH_\pi^{+\infty}$ denotes the spaces of smooth vectors in $\cH_\pi$.
We denote by $\sF_{G}$ the set of all invariant symbols on $G$.
\end{definition}

We set
$$
\cK_{a,b}(G) := \{\kappa\in \cS'(G), (f\mapsto f\star_G\kappa) \in \sL(L^2_a(G),L^2_b(G))\},
$$
Fixing a positive Rockland operator $\cR$ of homogenenous degree $\nu$, if $\sigma$ is an invariant symbol, then  so is the symbol
$$
(\id+\widehat \cR)^{b/\nu}  \sigma (\id+\widehat \cR)^{-a/\nu} 
=\{(\id+\pi (\cR))^{b/\nu}  \sigma(\pi) (\id+\pi (\cR))^{-a/\nu} : \pi\in \widehat G\}.
$$
We then define the following subspace of invariant symbols
$$
L^\infty_{a,b}(\widehat G)
:= \big\{ \sigma \in \sF_G \, \big| \, (\id+\widehat \cR)^{b/\nu}  \sigma\, (\id+\widehat \cR)^{-a/\nu} \in L^\infty(\widehat G) \big\}.
$$
Equipped with the norm 
$$
\|\sigma\|_{L^\infty_{a,b}(\widehat G),\cR}
:= \|(\id+\widehat \cR)^{b/\nu}  \sigma\, (\id+\widehat \cR)^{-a/\nu}\|_{L^\infty(\widehat G)}, 
$$
it is a Banach space, isomorphic and isometric to the subspace of left-invariant operators in $\sL(L^2_a(G),L^2_b(G))$. 
The properties    of the Rockland operators \cite[Section 5.1]{R+F_monograph} imply that $L^\infty_{a,b}(\widehat G)$ is independent of the choice of a positive Rockland operator $\cR$. 
Moreover, 
the Schwartz kernel theorem allows us to extend the group Fourier transform into a bijection
$\cF_G :\cK_{a,b}(G) \to L^\infty_{a,b}(G)$. 

\subsubsection{Difference operator}

\begin{definition}
	Let $q\in C^\infty(G)$ and let $\sigma$ be an invariant symbol on $\widehat G$. 
	We say that $\sigma$ is $\Delta_q$-differentiable when $\sigma\in L^\infty_{a,b}(\widehat G)$ and $q \cF_G^{-1}\sigma\in \cK_{c,d}(G)$ for some $a,b,c,d\in \bR$, and we set
	$$
	\Delta_q \sigma = \cF_G (q \cF_G^{-1}\sigma).
	$$
\end{definition}

If $\kappa\in \cS'(G)$ is in some classes of functions where the Fourier transform makes sense, e.g. in $L^1(G)$ or $L^2(G)$ or $\cK_{a,b}(G)$, then we say that the symbol $\sigma:=\widehat \kappa$  admits $\kappa$ as convolution kernel. 
If $q$ is a smooth function of polynomial growth, then the distribution $q\kappa\in \cS'(G)$ makes sense; if the Fourier transform of $q\kappa$ makes sense, then we write
$$
\Delta_q \sigma = \cF_G (q\kappa).
$$
We assume that a basis $X_1,\ldots, X_n$ of $\fg$ adapted to the gradation has been fixed. 
We then denote by $x_1,\ldots,x_n$ the corresponding coordinate functions on $G$, that is, 
$\exp \sum_{j=1}^n x_j X_j \mapsto x_j$.
We also set
$$
x^\alpha=x_1^{\alpha_1}\ldots x_n^{\alpha_n}, \quad\mbox{and}\quad
\Delta_\alpha=\Delta_{x^\alpha},\qquad 
\alpha=(\alpha_1,\ldots,\alpha_n)\in \bN_0^n.
$$

\subsubsection{The class $S^m(G\times \Gh)$}
We can now define our classes of symbols.
\begin{definition}
\label{def_SmGGh}
Let $m\in \bR$.
	A symbol $\sigma$ is in $S^m(G\times \Gh)$ when 
	\begin{enumerate}
		\item for every $x\in G$, 
	$\sigma(x,\cdot)$ is an invariant symbol 
	such that $ \sigma (x,\cdot) \in L^\infty_{0,-m} (\Gh)$, 
	\item the map $x\mapsto \cF_G^{-1} \sigma(x,\cdot)$ is  a smooth map from $G$ to the topological vector space $\cS'(G)$ of tempered distributions on $G$,
	\item for any $\alpha,\beta\in \bN_0^n$ and $\gamma\in \bR$, we have $\Delta_{\alpha}X_x^\beta \sigma (x,\cdot) \in L^\infty_{\gamma,[\alpha]+\gamma-m} (\Gh)$ for any $x\in G$, 
	with 
	$$
	\sup_{x\in G} \|\Delta_{\alpha}X_x^\beta \sigma (x,\cdot)\|_{L^\infty_{\gamma,[\alpha]+\gamma-m} (\Gh)}<\infty.
	$$
	\end{enumerate}
\end{definition}
In other words, if a positive Rockland operator $\cR$ is fixed on $G$,
 $\sigma$ belongs to $S^m(G\times \Gh)$ if and only if the following quantities  
\begin{align*}
\|\sigma\|_{S^m, a,b,c}
&:=
\max_{[\alpha]\leq a, [\beta]\leq b, |\gamma|\leq c }
\sup_{x\in G} \|X^\beta_x \Delta_{\alpha} \sigma(x,\cdot )\|_{L^\infty _{\gamma,[\alpha]+\gamma-m}(\Gh)}
\\&=\max_{[\alpha]\leq a, [\beta]\leq b, |\gamma|\leq c }
\sup_{x\in G,\pi\in \Gh}
\|\pi(\id+\cR)^{\frac{[\alpha]+\gamma-m}\nu}
X_x^\beta \Delta_{\alpha} \sigma(x,\pi)
\pi(\id+\cR)^{-\frac{\gamma}\nu}\|_{\sL(\cH_\pi)}	
\end{align*}
 are finite for any $a,b,c\in \bN_0$.

 In Definition \ref{def_SmGGh}, 
 Part (1) ensures that $\cF_G^{-1}\sigma(x,\cdot)=\kappa_{\sigma,x}$ has meaning as a tempered  distribution over $G$, 
 and Part (2), that $\kappa_\sigma:x\mapsto \kappa_{\sigma,x}$ is a smooth map $G\to \cS'(G)$.
 We call it \emph{the convolution kernel} of $\sigma$. 
  
The topology on the space  $S^m(G\times \Gh)$ is given by the semi-norms $\|\cdot \|_{S^m, a,b,c}$, $a,b\in \bN_0$. We may  even take `$c=0$'
\cite[Sections 5.2 and 5.5]{R+F_monograph}:
  \begin{theorem}
\label{thm_toposymbolG}
Let $m\in \bR$.
		The space $S^m(G\times \Gh)$ equipped with the semi-norms $\|\cdot\|_{S^m, a,b,c}$, $a,b,c\in \bN_0$, is a Fr\'echet space. 
			An equivalent topology is given by the semi-norms 
		$\|\cdot\|_{S^m , a,b,0}$, $a,b\in \bN_0$.
	We have the continuous inclusions
$$
m_1 \leq m_2 \ \Longrightarrow \ 
S^{m_1}(G\times \Gh)
\subset 
S^{m_2}(G\times \Gh).
$$
\end{theorem} 

It will be handy to set a notation for the class of symbols of any order:
$$
S^{\infty}(G\times \Gh):= \cup_{m\in \bR} S^m(G\times \Gh).
$$

\subsubsection{Smoothing symbols}
\label{subsubsec_smoothingsymbol}
 The class of {\it smoothing symbols}
$$
S^{-\infty}(G\times \Gh):= \cap_{m\in \bR} S^m(G\times \Gh).
$$
is equipped  with the induced topology of projective limit.

The smoothing symbols are characterised by their convolution kernels
\cite[Section 5.4]{R+F_monograph}:

\begin{proposition}
\label{prop_smoothing}
The map 
$\sigma \mapsto \kappa_\sigma$ is a continuous injective morphism
from $S^{-\infty}(G\times \Gh)$ to $C^\infty(G:\cS(G))$.
Its image is  the vector space of  maps $ x\mapsto \kappa_x $ in $C^\infty(G:\cS(G))$ such that $\kappa_x$ and all its left invariant derivatives $X^\beta_x \kappa_x$ in $x$ are 
Schwartz functions and form a bounded subset of the Fr\'echet space $\cS(G)$ as $x$ runs over $G$.
\end{proposition}

We denote the space of smooth scalar-valued functions which are bounded as well as all their left-derivatives by 
$$
C^\infty_{l,b}(G):= \{f \in C^\infty (G) \ : \ \sup_{x\in G} |X^\beta f(x)|<\infty \ \mbox{for any}\ \beta\in \bN_0^n\}.
$$
This extends to functions $f$ valued in a topological vector space. 
Proposition \ref{prop_smoothing} may be rephrased as saying that the map 
$\sigma \mapsto \kappa_\sigma$ is an isomorphism of topological vector spaces from $S^{-\infty}(G\times \Gh)$ onto 
$C^\infty_{l,b}(G : \cS(G))$.

In the following sense, smoothing symbols are dense in any $S^m(G\times \Gh)$:
\begin{lemma}
\label{lem_densitysmoothing}
Let $\sigma\in S^m(G\times \Gh)$ with $m\in \bR$. 
Then we can construct a sequence $\sigma_\ell \in S^{-\infty}(G\times \Gh)$, $\ell\in \bN$, that converges to $\sigma$ in $S^{m_1}(G\times \Gh)$ as $\ell\to \infty$ for any $m_1>m$, 
and satisfies for any semi-norm $\|\cdot \|_{S^m,a,b,c}$
$$
\limsup_{\ell \to \infty} \|\sigma_\ell\|_{S^m,a,b,c}
\geq \|\sigma\|_{S^m,a,b,c}.
$$
\end{lemma}
\begin{proof}
The first property is proved in 
	\cite[Section 5.4.3]{R+F_monograph}.
	The second follows with the same construction. 
\end{proof}

\subsubsection{Main properties of the symbol classes on $G\times \Gh$}  
\label{subsec_propSm}
The symbol classes $S^m(G\times \Gh)$, $m\in \bR$,
form an $*$-algebra \cite[Sections 5.2]{R+F_monograph} in the sense that 
the composition map 
$$
\left\{\begin{array}{rcl}
S^{m_1}(G\times \Gh) \times S^{m_2}(G\times \Gh)
&\longrightarrow & S^{m_1+m_2}(G\times \Gh)
\\
(\sigma_1,\sigma_2)& 	\longmapsto& \sigma_1\sigma_2
\end{array}
\right., \quad m_1,m_2\in \bR,
$$
and the adjoint map
$$
\left\{\begin{array}{rcl}
S^{m}(G\times \Gh)
&\longrightarrow & S^{m}(G\times \Gh)
\\
\sigma& 	\longmapsto& \sigma^*
\end{array}
\right. ,\quad m\in \bR, 
$$
are continuous.

 \begin{ex}
The invariant symbol 
	$\widehat X^\alpha = \{\pi(X)^\alpha : \pi\in \Gh\}$ 
	is in $S^m(G\times \Gh)$ with $m=[\alpha]$.	
\end{ex}

Another important example of classes of symbols are given by the multipliers in the symbols of a positive Rockland operator.
This often requires  the multiplier has enough regularity and decay. For this we define, the class of functions of growth at most $m$ in the following sense:
\begin{definition}
\label{def_cGm}
Let $\cG^m(\bR)$ be the space of smooth functions $\phi:\bR \to \bC$ growing at rate $m\in \bR$ in the sense that   
$$
\forall k\in \bN_0,\qquad \exists C=C_{k,\phi},\qquad
\forall \lambda\in \bR, \qquad |\partial_\lambda^k \phi(\lambda)|\leq C (1+|\lambda|)^{m-k}.
$$	
\end{definition}
This is a Fr\'echet space when equipped with the semi-norms given by
$$
\|\phi\|_{\cG^m, N} := \max_{k=0,\ldots, N} \sup_{\lambda\in \bR} 
(1+|\lambda|)^{-m+k}|\partial_\lambda^k \phi(\lambda)|, 
\qquad N\in \bN_0. 
$$
Many spectral multipliers in positive Rockland operators are in our symbol classes
\cite[Proposition 5.3.4]{R+F_monograph}:
\begin{theorem}
\label{thm_phi(R)}
Let $\phi\in \cG^m (\bR)$ and let $\cR$ be a positive Rockland operator on $G$ of homogeneous degree $\nu$.
	Then $\phi(\widehat \cR)\in S^{m\nu}(G\times \widehat G)$. 
	Moreover, the map 
	$$ \cG^m (\bR)\ni \phi\mapsto \phi(\widehat \cR)\in S^{m\nu}(G\times\widehat G)$$
	is continuous. 
\end{theorem}

\subsection{The quantization and the pseudo-differential calculus on $G$}
\label{subsec_PsiG}

By \cite[Section 5.1]{R+F_monograph},
for any $\sigma\in S^\infty(G\times \Gh)$, $f\in \cS(G)$ and $x\in G$, the formula 
$$
\Op_G (\sigma)f(x) = \int_{\Gh} \tr\left( \pi(x)\sigma(x,\pi) \widehat f(\pi)\right)d\mu(\pi)
$$
defines a smooth function of $x$; it is equal to $f*\kappa_x(x)$ where $\kappa$ is the convolution kernel of $\sigma$. 

We denote by $\Psi^m(G)=\Op_G(S^m(G\times \Gh))$, $m\in \bR$, the spaces of operators $\Op_G(\sigma)$, $\sigma\in S^m(G\times \Gh)$.
It inherits naturally a Fr\'echet structure.

\begin{ex}
For any $\alpha\in \bN_0$,  $\Op_G (\widehat X^\alpha)=X^\alpha \in \Psi^{[\alpha]}(G)$. 	
More generally, for any $N$, given $c_\alpha\in C_{l,b}^\infty(G)$, $\alpha\in \bN_0^n$, $[\alpha]\leq N$, the symbol 
$\sigma=\sum_{[\alpha]\leq N} c_\alpha(x)\widehat X^\alpha$ is in $S^N(G\times \Gh)$, 
therefore $\Op_G(\sigma)=\sum_{[\alpha]\leq N} c_\alpha(x) X^\alpha$ is in $\Psi^N(G)$. 

If $\widehat \cR$ is the symbol of a positive Rockland operator, as it does not depend on $x$, we have $\cR=\Op_G (\widehat \cR)$. 
More generally, this is true for the spectral multipliers:
$$
\forall \phi\in \cG^m (\bR) \qquad
\Op_G (\phi (\widehat \cR)) = \phi(\cR).
$$
\end{ex}
This example implies that the resulting class of operators
$$
\Psi^\infty(G):=
\cup_{m\in \bR}\Psi^m(G)
$$ contains the left-invariant differential calculus on $G$ and the spectral multipliers in positive Rockland operators.  
It forms a pseudo-differential calculus \cite[Chapter 5]{R+F_monograph} in the following sense:
\begin{theorem}
\label{thm_PDOGcomp+adj}
\begin{enumerate}
	\item If $T\in \Psi^m(G)$ with $m\in \bR$ then $T$ is continuous $L^2_s(G)\to L^2_{s - m}(G)$ for any $s\in \bR$, $\cS(G)\to \cS(G)$ and $\cS'(G)\to \cS'(G)$.
	Moreover, $T\mapsto T$ is continuous $\Psi^m(G)\to \sL(L^2_s(G), L^2_{s - m}(G))$.
	\item If $T_1\in \Psi^{m_1}(G)$ and $T_2\in \Psi^{m_2}(G)$ with $m_1,m_2\in \bR$, then the composition $T_1T_2$ is in $\Psi^{m_1+m_2}(G)$.	Moreover, the map $(T_1,T_2)\mapsto T_1 T_2$ is continuous $\Psi^{m_1}(G) \times \Psi^{m_2}(G)\to \Psi^{m_1+m_2}(G)$.
	\item  If $T\in \Psi^{m}(G)$ with $m\in \bR$, then its formal adjoint $T^*$ is in $\Psi^{m}(G)$.
	Moreover, the map $T\mapsto T^*$ is continuous $\Psi^{m}(G)\to \Psi^{m}(G)$.
\end{enumerate}	
\end{theorem}

\begin{remark}
\label{rem_thm_PDOGcomp+adj_L2bddness}
Regarding the proof of Part (1) in 	Theorem \ref{thm_PDOGcomp+adj}, it suffices to show the case of an operator of order $m=0$ and its boundedness on $L^2(G)$.
The other orders and actions on Sobolev spaces then follow from the properties of composition (Part  (2)) and of spectral multipliers in a positive Rockland operator or symbol. 

For the case of $m=0$, the proof given for   \cite[Theorem 5.4.17]{R+F_monograph} shows that 
there exists $C>0$ so that for any $T\in \Psi^0(G)$, we have
$$
\|T\|_{\sL(L^2(G))} \leq C\left( \max_{[\beta]\leq 1+Q/2} \sup_{(x,\pi)\in G\times \Gh}\|X^\beta_x \sigma(x,\pi)\|_{\sL(\cH_\pi)} \ + \ \sup_{x\in G} \||\cdot|_p^{pr}\kappa_x \|_{L^2(G)}\right).
$$
Above, $\kappa_x$ denotes the convolution kernel of $T$ and $|\cdot|_p$ is the quasinorm defined in \eqref{eq_qnormp}
with $p\in \bN$ such that $p/2$ is the smallest common multiple of the dilations' weights $\upsilon_1,\ldots,\upsilon_n$. The integer $r$ is chosen so that $pr>Q/2$.
By the kernel estimates (see Theorem \ref{thm_kernelG} below), this implies that the quantity
$\sup_{x\in G} \||\cdot|_p^{pr}\kappa_x \|_{L^2(G)}$ is indeed finite. Moreover, it defines a continuous semi-norm on $\Psi^0(G)$.  
\end{remark}

In fact, the proofs of the properties of composition and adjoint in Theorem \ref{thm_PDOGcomp+adj}
 also give asymptotic expansions in the following sense:
\begin{definition}
\label{def_asymp}
A symbol $\sigma\in S^m(G\times \Gh)$ admits an asymptotic expansion in $S^m(G\times \Gh)$ when there exists a sequence of symbol $\tau_j$ with 
$$
\tau_j\in S^{m-j}(G\times \Gh)\ \mbox{for any} \, j\in \bN_0,\
\quad\mbox{and for all}\ N\in \bN_0,\quad
\sigma-\sum_{j\leq N} \tau_j \in S^{m-(N+1)}. 
$$
We then write
$$
\sigma \sim \sum_{j\in \bN_0} \tau_j \quad\mbox{in}\ S^m(G\times \Gh).
$$
\end{definition} 

Given an asymptotic expansion $\sum_{j\in \bN_0} \tau_j$, with $\tau_j\in S^{m-j}(G\times \Gh)$, $j\in \bN_0$, then there exists a symbol $\sigma\in S^m(G\times \Gh)$ admitting this asymptotic expansion; $\sigma$ is unique modulo $S^{-\infty}(G\times \Gh)$ \cite[Theorem 5.5.1]{R+F_monograph}.

To describe the asymptotic expansions for composition and adjoint, 
it is handy to adopt the notation 
$$
\Op_G(\sigma_1\diamond\sigma_2) = \Op_G(\sigma_1)\Op_G(\sigma_2)
\quad\mbox{and}\quad 
\Op_G(\sigma^{(*)}) = (\Op_G(\sigma))^* .
$$
Moreover, we denote by  $(q_\alpha)$  the basis of homogeneous polynomials dual to $(X^\alpha)$:
$$
X^{\alpha'}q_\alpha (0) = \delta_{\alpha=\alpha'},
\alpha,\alpha'\in \bN_0^n.
$$
We also denote by  $\Delta^\alpha$  the corresponding difference operators for $q_\alpha(\cdot^{-1}):x\mapsto q_\alpha (x^{-1})$:
$$
\Delta^\alpha := \Delta_{q_\alpha(\cdot^{-1})},
$$

\begin{theorem}
\label{thm_asymp}
For $\sigma_1 \in S^{m_1}(G\times \Gh)$,  $\sigma_2 \in S^{m_2}(G\times \Gh)$, and 
$\sigma\in S^m(G\times \Gh)$, the asymptotic expansions of the symbols for composition and adjoint are given by: 
$$
\sigma_1\diamond\sigma_2 \sim \sum_\alpha \Delta^\alpha \sigma_1 \, X^\alpha \sigma_2
\ \mbox{in}\ S^{m_1+m_2}(G\times \Gh),
\quad 
\sigma^{(*)} \sim \sum_\alpha \Delta^\alpha  X^\alpha \sigma^*
\ \mbox{in}\ S^m(G\times \Gh).
$$
Moreover, in this case, the  maps are continuous: 
\begin{align*}
	(\sigma_1,\sigma_2)&\longmapsto 
\sigma_1\diamond\sigma_2 - \sum_{[\alpha]\leq N} \Delta^\alpha \sigma_1 \, X^\alpha \sigma_2,
\ \ 
	S^{m_1}(G\times \Gh)\times S^{m_2}(G\times \Gh) \longrightarrow S^{m_1+m_2- N}(G\times \Gh),\\
	\sigma &\longmapsto \sigma^{(*)} - \sum_{[\alpha]\leq N}  \Delta^\alpha  X^\alpha \sigma^*, \ \ 
S^{m}(G\times \Gh)\longrightarrow S^{m- N}(G\times \Gh),
\end{align*}
for any $N\in \bN_0$.
\end{theorem}

Theorem \ref{thm_asymp} is proved by studying the convolution kernels of $\sigma_1\diamond\sigma_2 $ and $\sigma^{(*)}$ which are formally given by 
\begin{align}
	\kappa_{\sigma_1\diamond\sigma_2,x}(y)
	&=\int_G \kappa_{\sigma_2, xz^{-1}}(yz^{-1})\kappa_{1,x}(z) dz \label{eq_kappa12}\\
	\kappa_{\sigma^{(*)},x}(y)
	&=\bar\kappa_{\sigma,xy^{-1}}(y^{-1})\label{eq_kappaadj}
\end{align}
We will generalise the proof of Theorem \ref{thm_asymp} by adding a semiclassical parameter $\eps$ in Theorem \ref{thm_sclexp_prod+adj}.

\subsection{Kernel estimates}

The  convolution kernel associated with a symbol in some $S^m(G\times \widehat G)$ will be Schwartz away from the origin 0, but may have a singularity at 0 \cite[Theorem 5.4.1]{R+F_monograph}:
\begin{theorem}
\label{thm_kernelG}
	Let $\sigma\in S^m(G\times \widehat G)$ and denote its convolution kernel by $\kappa_x =\cF_G^{-1} \sigma(x,\cdot)$.
	Then $\kappa_x(y)$ is smooth away from the origin $y=0$. Moreover,  fixing a quasi-norm $|\cdot|$ on $G$, we have the following kernel estimates:
	\begin{enumerate}
		\item The convolution kernel $\kappa$ decays faster than any polynomial away from the origin:
	\begin{align*}
	    	\forall N\in \bN_0, \quad 
	\exists C=C_{\sigma, N}>0:\quad \forall x,y\in G, \\ |y|\geq 1\Longrightarrow
	|\kappa_x(y)|\leq C |y|^{-N}.
	\end{align*}
	\item 
	If $Q+m<0$ then $\kappa$ is continuous and bounded on $G\times G$:
 $$
 \exists C=C_\sigma>0, \qquad 
 \sup_{x,y\in G} |\kappa_x(y)| \leq C.
 $$

 \item If $Q+m>0$, then 
	\begin{align*}
	\exists C=C_\sigma>0:
        \quad \forall x,y\in G , \qquad 0<|y|\leq 1\Longrightarrow
	|\kappa_x(y)|\leq  C|y|^{-(Q+m) }.
	\end{align*}

 \item If $Q+m=0$, then 
	\begin{align*}
	\exists C=C_\sigma>0:
        \quad \forall x,y\in G , \qquad 0<|y|\leq 1/2\Longrightarrow
	|\kappa_x(y)|\leq -  C \ln |y|.
	\end{align*}
	\end{enumerate}

In all the estimates above, 
 the constant $C$ may be chosen of the form 
$C = C_1 \|\sigma\|_{S^m, a,b,c}$
with $C_1>0$, $a,b,c\in \bN_0$ independent of $\sigma$. 
\end{theorem}

\begin{corollary}
	\label{cor_sigmacompop}
If $\sigma\in S^m(G\times \Gh)$ with $m<-Q$, 
then we may realise
		$\sigma(x,\pi)$ as a compact operator for each $(x,\pi)\in G\times \Gh$.
\end{corollary}
\begin{proof}
For each $x\in G$, by Theorem \ref{thm_kernelG}, $\cF_G^{-1} \sigma(x,\cdot)$ defines an integrable function on $G$. Recall that if $\kappa\in L^1(G)$ and $\pi$ is a unitary irreducible representation, then $\pi( \kappa)$  is a compact operator on $\cH_\pi$ \cite{Dixmier}. Consequently, $\sigma(x,\pi)$ is a well defined compact operator for each $(x,\pi)\in G\times \Gh$.
\end{proof}

\begin{corollary}
\label{cor_trHSG}
Let $\sigma\in S^m(G\times \Gh)$.
We assume that $\sigma$ is compactly supported in $x\in G$,
		that is, if there exists a compact subset $\cC \subset G $ such that $\sigma(x,\pi)=0$ for all $(x,\pi)\in G\times \Gh$ with $x\not \in \cC$.
\begin{enumerate}
		\item If $m<-Q$, then 
		  we have
		$$
		\int_G \int_{\Gh} \tr |\sigma(x,\pi)| 
		d\mu(\pi)dx \leq C |\cC| \|\sigma\|_{S^m,0,0,0} ,
		$$
		where $C$ is a constant depending on $m$ and the Rockland operator $\cR$ fixed to consider the associated  $S^m(G\times \Gh)$ semi-norms, and  $|\cC|$ denotes the volume of $\cC$ for the Haar measure. Furthermore,  $\Op_G(\sigma)$ is trace-class with
		$$
		\tr (\Op_G(\sigma))  = \int_G \int_{\Gh} \tr (\sigma(x,\pi)) d\mu(\pi)dx.
		$$
		\item If $m<-Q/2$, then $\Op_G(\sigma)$ is Hilbert Schmidt with Hilbert Schmidt norm satisfying
		$$
		\|\Op_G(\sigma)) \|_{HS}^2 = \int_G \int_{\Gh} \|\sigma(x,\pi)\|_{HS(\cH_\pi)}^2 d\mu(\pi)dx
		\leq C' |\cC| \|\sigma\|_{S^m,0,0,0}^2,	
		$$
		with $C'>0$ a constant depending on $\cR$ and $m$. 
	\end{enumerate}
\end{corollary}

\begin{proof}
Denoting by $\nu$ the homogeneous degree of $\cR$, let us recall that the convolution kernel $\cB_a$ of $(\id+\cR)^{-a/\nu}$ is integrable for $a>0$ and square integrable for $a>Q/2$
\cite[Corollary 4.3.11 (ii)]{R+F_monograph}. Consequently, by the Plancherel formula, we have for any $a>Q/2$:
$$
\int_{\Gh} \|(\id+\pi(\cR))^{-a/\nu}\|^2_{HS(\cH_\pi)}d\mu(\pi)
=
\|\cB_a\|_{L^2(G)} ^2
<\infty.
$$

Assume  $m<-Q$. We have
$$
\tr |\sigma(x,\pi)| 
\leq 
\|(\id +\cR)^{m/\nu}\sigma(x,\pi)\|_{\sL(\cH_\pi)} \tr |(\id +\cR)^{-m/\nu}|
\leq 
\|\sigma\|_{S^m,0,0,0} \|(\id+\pi(\cR))^{-m/2\nu}\|^2_{HS(\cH_\pi)},
$$
and 
$$
		\int_G \int_{\Gh} \tr |\sigma(x,\pi)| 
		d\mu(\pi)dx \leq  |\cC| \|\sigma\|_{S^m,0,0,0} \int_\Gh \|(\id+\pi(\cR))^{-m/2\nu}\|^2_{HS(\cH_\pi)} d\mu(\pi).
		$$
This shows the first part of Part (1) with $C_\cR=\|\cB_{m/2}\|_{L^2(G)}^2$.
	Denoting by $\kappa$ the  convolution kernel of $\sigma\in S^m(G\times \Gh)$, the Fourier inversion formula yields
	$$
	\int_G \int_{\Gh} \tr (\sigma(x,\pi)) d\mu(\pi)dx = \int_G \kappa_x(0) dx.
	$$
	We observe that  $(x,y)\mapsto \kappa_x(y)$ is continuous on $G\times G$ and compactly supported in $x$ and that the integral kernel of $\Op_G(\sigma)$ is $(x,y)\mapsto \kappa_x(y^{-1}x)$.
	Hence $\Op_G(\sigma)$ is trace-class with trace given by $\int_G \kappa_x(0) dx$.
	This concludes the proof of Part (1).

As $(x,y)\mapsto \kappa_x(y^{-1}x)$ is the integral kernel of $\Op_G(\sigma)$, we have by the Plancherel formula:
$$
		\|\Op_G(\sigma)) \|_{HS}^2 = 
		\int_{G\times G} |\kappa_x(y^{-1}x)|^2 dx dy
		=
		\int_G \int_{\Gh} \|\sigma(x,\pi)\|_{HS(\cH_\pi)}^2 d\mu(\pi)dx
		\leq |\cC| \|\sigma\|	.	
		$$
		Since we have for any $(x,\pi)\in G\times \Gh$
	\begin{align*}
	\|\sigma(x,\pi)\|_{HS(\cH_\pi)} 
	&\leq \|(\id +\cR)^{m/\nu}\sigma(x,\pi)\|_{\sL(\cH_\pi)} \|(\id +\cR)^{-m/\nu}\|_{HS(\cH_\pi)}
	\\&\leq \|\sigma\|_{S^m,0,0,0} \|(\id+\pi(\cR))^{-m/\nu}\|_{HS(\cH_\pi)},	
	\end{align*}
we obtain the estimate
$$
\int_G \int_{\Gh} \|\sigma(x,\pi)\|_{HS(\cH_\pi)}^2 d\mu(\pi)dx
		\leq |\cC| \|\sigma\|_{S^m,0,0,0} 
		\int_{\Gh} \|(\id+\pi(\cR))^{-m/\nu}\|_{HS(\cH_\pi)}d\mu(\pi).
	$$	
	If $m<-Q/2$, this shows Part (2) with $C' = \|\cB_{m}\|_{L^2(G)}$.
\end{proof}

\subsection{Symbols and quantization  on $M$}

We now consider a compact nilmanifold $M=\Gamma \backslash G$.

\subsubsection{Symbol classes on $M\times \Gh$}

\begin{definition}
\label{def_SmMGh}
Let $m\in \bR\cup {-\infty}$.
\begin{enumerate}
	\item A symbol $\sigma$ is in $S^m(M\times \Gh)$ when 
	\begin{itemize}
		\item for every $\dot x\in M$, 
	$\sigma(\dot x,\cdot)$ is an invariant symbol 
	 in $L^\infty_{m,0} (\Gh)$, and 
	\item the symbol $\sigma_G $ given by $\sigma_G(x,\pi)= \sigma(\dot x,\pi)$ is in $S^m(G\times \Gh)$. 
	\end{itemize}
\item A symbol $\sigma\in S^m(G\times \Gh)$ is $\Gamma$-\emph{periodic} when $\sigma(\gamma x,\pi) =\sigma(x,\pi)$ for any $\gamma\in \Gamma$, $x\in G$ and $\pi\in \Gh$.
\end{enumerate}
	\end{definition}

If $\sigma\in S^m(G\times \Gh)$ is $\Gamma$-periodic, 
we denote by $\sigma_M$ the corresponding symbol on $M\times \Gh$.
Clearly, $\sigma\to \sigma_G$ is a bijection from $S^m(M\times \Gh)$ to 
the space $S^m(G\times \Gh)^\Gamma$ of  $\Gamma$-periodic symbols in $S^m(G\times \Gh)$, with inverse given by $\tau \to \tau_M$.
As the space of $\Gamma$-periodic symbols in $S^m(G\times \Gh)$ is closed, $S^m(M\times \Gh)$ inherits a Fr\'echet structure. 
For $m\in \bR$, it is given by the semi-norms:
$$
\|\sigma\|_{S^m, a,b,c}
:=
\|\sigma_G\|_{S^m, a,b,c}
=
\max_{[\alpha]\leq a, [\beta]\leq b, |\gamma|\leq c }
\sup_{\dot x\in M} \|X_M^\beta \Delta_{\alpha} \sigma(\dot x,\pi)\|_{L^\infty _{\gamma, m -[\alpha]-\gamma}(\Gh)},\quad 
a,b,c\in \bN_0.
$$
By Theorem \ref{thm_toposymbolG}, 
an equivalent topology is given by the semi-norms 
		$\|\cdot\|_{S^m , a,b,0}$, $a,b\in \bN_0$, 
and we have the continuous inclusions
$$
m_1 \leq m_2 \ \Longrightarrow \ 
S^{m_1}(M\times \Gh)
\subset 
S^{m_2}(M\times \Gh).
$$
Moreover, the smoothing symbols are dense in $S^m(M\times \Gh)$ in the same sense as in Lemma \ref{lem_densitysmoothing}.

As a consequence of Proposition \ref{prop_smoothing}, 
the smoothing symbols on $M\times \Gh$ are  described by their convolution kernels:
\begin{corollary}
	The map 
$\sigma \mapsto \kappa_\sigma$ is an isomorphism of topological vector spaces 
 $S^{-\infty}(M\times \Gh)\to C^\infty(M:\cS(G))$.
\end{corollary}

The group case implies that 
the symbol classes $S^m(M\times \Gh)$, $m\in \bR$,
form a $*$-algebra in the sense that 
the composition map 
$$
\left\{\begin{array}{rcl}
S^{m_1}(M\times \Gh) \times S^{m_2}(M\times \Gh)
&\longrightarrow & S^{m_1+m_2}(M\times \Gh)
\\
(\sigma_1,\sigma_2)& 	\longmapsto& \sigma_1\sigma_2
\end{array}
\right., \quad m_1,m_2\in \bR,
$$
and the adjoint map
$$
\left\{\begin{array}{rcl}
S^{m}(M\times \Gh)
&\longrightarrow & S^{m}(M\times \Gh)
\\
\sigma& 	\longmapsto& \sigma^*
\end{array}
\right. ,\quad m\in \bR, 
$$
are continuous.

The space $S^\infty (M\times \Gh):=\cup_{m\in \bR}S^m(M\times \Gh)$ contains the symbols $\sigma$ which are invariant such as $\pi(X)^\alpha$ or the spectral multiplier in $\widehat \cR$ in Theorem \ref{thm_phi(R)}, and also $a(\dot x)\sigma$ with $a\in C^\infty(M)$. 

\subsubsection{Quantization on $M$}

The quantization on $M$ will follow from the quantization on $G$ and the following observation:
\begin{lemma}
\label{lem_obsQuantizationM}
	If $\sigma\in S^m(G\times \Gh)$ then for any $f\in \cS'(G)$ and $\gamma\in G$, we have:
	$$
	\left (\Op_G (\sigma) f\right )\, (\gamma\,  \cdot ) 
	= \Op_G \left ( \sigma(\gamma\,  \cdot ,\pi)\right ) \left (f(\gamma\,  \cdot )\right ) .
	$$
	\end{lemma}
	\begin{proof}
	By the density of $\cS(G)$ in $\cS'(G)$ and the continuity of $\Op_G (\sigma)$ on $\cS'(G)$
	(see Theorem \ref{thm_PDOGcomp+adj} (1)), 
	it suffices to prove the property for $f\in \cS(G)$.
	Denoting by $\kappa$ the convolution kernel of $\sigma$, we have for any $x\in G$ and $\gamma\in G$,
$$
	(\Op_G (\sigma) f) (\gamma x ) 
	=
	f*\kappa_{\gamma x}(\gamma x) 
	=\int_G f(y) \kappa_{\gamma x}(y^{-1}\gamma x) dy
	=\int_G f(\gamma z) \kappa_{\gamma x}(z^{-1}x) dz
	$$
after the change of variable $y=\gamma z$.	
We recognise $f(\gamma\,  \cdot ) *\kappa_{\gamma x} (x)$, and the statement follows. 
	\end{proof}

If $f\in \cD'(M)$ and $\sigma\in S^m(M\times \Gh)$ with $m\in \bR$, then 
$\Op_G (\sigma_G) f_G$ is $\Gamma$-periodic by Lemma \ref{lem_obsQuantizationM} and we may  set
$$
\Op_M (\sigma) f := \left( \Op_G (\sigma_G) f_G \right)_M.
$$
This formula defines an operator $\Op_M (\sigma)$ acting continuously on $\cD'(M)$; this gives a quantization on $M$.

\begin{lemma}
\label{lem_IntKernel}
Let $\sigma\in S^m(M\times \Gh)$. 
Then its convolution kernel  	$\kappa$ viewed as a distribution on $G\times G$ that is $\Gamma$-invariant satisfies the kernel estimates in Theorem \ref{thm_kernelG}.
The integral integral of $\Op_M (\sigma)$  is the distribution $K$ on $M\times M$ given by
$$
K(\dot x,\dot y) = 
\sum_{\gamma\in \Gamma}  \kappa_{\dot  x}(y^{-1}\gamma  x) = (\kappa (\cdot^{-1} x))_M^\Gamma(\dot y),
$$
with the notation of Section \ref{subsubsec_periodicfcn}.
If $m=-\infty$, then $K$ is smooth. 
\end{lemma}
\begin{proof}
The case of $\sigma\in S^{-\infty}(M\times \Gh)$ follows from the results in  Section \ref{subsubsec_periodicfcn}.	
The density of $\cD(M)$ in $\cD'(M)$ implies the result for any $\sigma \in S^m(M\times \Gh)$.
\end{proof}

We denote by $\Psi^m(M)=\Op_M(S^m(M\times \Gh))$, $m\in \bR$, the spaces of operators $\Op_G(\sigma)$, $\sigma\in S^m(G\times \Gh)$.
It inherits naturally a Fr\'echet structure.

\begin{ex}
For any $\alpha\in \bN_0$,  $\Op_M (\widehat X^\alpha)=X_M^\alpha \in \Psi^{[\alpha]}(M)$. 	
More generally, for any $N$, given $c_\alpha\in C^\infty(M)$, $\alpha\in \bN_0^n$, $[\alpha]\leq N$, the symbol 
$\sigma=\sum_{[\alpha]\leq N} c_\alpha(\dot x)\widehat X^\alpha$ is in $S^N(M\times \Gh)$, 
therefore $\Op_M(\sigma)=\sum_{[\alpha]\leq N} c_\alpha(\dot x) X^\alpha$ is in $\Psi^N(M)$. 
\end{ex}
This example implies that $\Psi^\infty(M)=\cup_{m\in \bR}\Psi^m(M)$ contains the left-invariant differential calculus on $M$, that is, the $C^\infty(M)$-module generated by $X^\alpha$, $\alpha\in \bN_0^n$.

\subsection{The pseudo-differential calculus on $M$}

The spaces $\Psi^m(M)$, $m\in \bR$, form a calculus in the following sense:
\begin{theorem}
\label{thm_PDOGcomp+adjM}
\begin{enumerate}
	\item If $T\in \Psi^m(M)$ with $m\in \bR$ then $T$ is continuous $L^2_s(M)\to L^2_{s - m}(M)$ for any $s\in \bR$, $C^\infty(M)\to C^\infty(M)$ and $\cD'(M)\to \cD'(M)$.
	Moreover, $T\mapsto T$ is continuous $\Psi^m(M)\to \sL(L^2_s(M), L^2_{s - m}(M))$.
	\item If $T_1\in \Psi^{m_1}(M)$ and $T_2\in \Psi^{m_2}(M)$ with $m_1,m_2\in \bR$, then the composition $T_1T_2$ is in $\Psi^{m_1+m_2}(M)$.	Moreover, the map $(T_1,T_2)\mapsto T_1 T_2$ is continuous $\Psi^{m_1}(M) \times \Psi^{m_2}(M)\to \Psi^{m_1+m_2}(M)$.
	\item  If $T\in \Psi^{m}(M)$ with $m\in \bR$, then its formal adjoint $T^*$ is in $\Psi^{m}(M)$.
	Moreover, the map $T\mapsto T^*$ is continuous $\Psi^{m}(M)\to \Psi^{m}(M)$.
\end{enumerate}	
\end{theorem}

In fact, Parts (2) and (3) follow from Theorem \ref{thm_PDOGcomp+adj} applied to operators with symbols that are $\Gamma$-periodic. Indeed, it implies that 
if $\sigma_1\in S^{m_i}(M\times \Gh)$, $i=1,2$,
then $\Op_M(\sigma_1)\Op_M(\sigma_2) \in \Psi^{m_1+m_2}(M)$
with 
$$
\Op_M(\sigma_1)\Op_M(\sigma_2)=
\Op_M(\sigma_1\diamond\sigma_2) 
\quad\mbox{where}\quad 
\sigma_1\diamond\sigma_2 := 
((\sigma_1)_G \diamond(\sigma_2)_G)_M,
$$
since we check readily from \eqref{eq_kappa12}
that the convolution kernel of $(\sigma_1)_G \diamond(\sigma_2)_G$ is $\Gamma$-periodic in $x$.
Similarly, if $\sigma\in S^m(M\times \Gh)$ then 
$(\Op_M(\sigma))^* \in \Psi^m(M)$
with 
 convolution kernel that is given by \eqref{eq_kappaadj} and clearly  $\Gamma$-periodic in $x$, and we have:
$$
\Op_M(\sigma^{(*)}) = (\Op_M(\sigma))^* 
\quad\mbox{where}\quad 
\sigma^{(*)} := (\sigma_G^{(*)})_M.
$$
We define a notion of asymptotic expansion in $S^m(M\times \Gh)$ in a similar way as in Definition \ref{def_asymp}, 
and we obtain similar asymptotic expansions for composition and adjoint as described on $G$ in Section \ref{subsec_PsiG}.

\begin{proof}[Proof of Theorem \ref{thm_PDOGcomp+adjM}]
	Parts (2) and (3) follow from the observations above together with Theorem \ref{thm_PDOGcomp+adj}.
	By construction, $T\in \Psi^m(M)$ acts continuously on $\cD'(M)$. It remains to show the continuity on the Sobolev spaces as this will imply the continuity on $C^\infty(M)$.
	Since $(\id +\widehat \cR)^{s/\nu}$ is in $S^s(M\times \Gh)$ for any $s\in \bR$ and positive Rockland operator $\cR$ (with homogeneous degree $\nu$), it suffices to show the case of symbols of order $m=0$ and Sobolev order 0.
	That is, it remains to show that if $\sigma\in S^0(M\times \Gh)$ then $\Op_M(\sigma)$ is bounded on $L^2(M)$, with operator norm bounded up to a constant of $M$ by a semi-norm of $\sigma\in S^0(M\times \Gh)$.
	We may assume $\sigma$ smoothing, 
	see Section \ref{subsubsec_smoothingsymbol}.  
	
Let $\sigma\in S^{-\infty}(M\times \Gh)$.
Its convolution kernel $\kappa$ satisfies  $\kappa_{\dot x} \in \cS(G)$. We have for any $f\in C^\infty(M)$ and $\dot x\in M$ (since $f_G\in C_b^\infty (G)$)
\begin{align*}
\Op_M(\sigma) f(\dot x)
&=f_G *\kappa_{\dot x}(x),	
\\
|\Op_M(\sigma) f(\dot x)|
&\leq \sup_{\dot x_1\in M}|f_G *\kappa_{\dot x_1}(x)|
\lesssim_M \sum_{[\beta] \leq \nu_0}
\|X^\beta_{M,\dot x_1} f_G *\kappa_{\dot x_1}(x)\|_{L^2(M,d\dot x_1)},
\end{align*}
by the Sobolev embedding (see Proposition \ref{prop_sobolev_spacesM}), where $\nu_0$ is the smallest common multiple of the dilation weights $\upsilon_1,\ldots,\upsilon_n$ satisfying $\nu_0>Q/2$. 
We observe that 
$$
X^\beta_{M,\dot x_1} f_G *\kappa_{\dot x_1}(x)
=
f_G *X^\beta_{M,\dot x_1}\kappa_{\dot x_1}(x)
=
R(X^\beta_{M,\dot x_1}\kappa_{\dot x_1}) f(\dot x), 
$$
with the notation of Example \ref{ex_Regrep}.
Therefore, we have
$$
\int_M |\Op_M(\sigma) f(\dot x)|^2 d\dot x
\lesssim_M \sum_{[\beta] \leq \nu_0}
\int_M \int_M 
| R(X^\beta_{M,\dot x_1}\kappa_{\dot x_1}) f(\dot x)|^2 d\dot x_1 d\dot x,
$$
and the right-hand side is equal to 
$$
 \sum_{[\beta] \leq \nu_0}
\int_M 
\| R(X^\beta_{M,\dot x_1}\kappa_{\dot x_1}) f\|_{L^2(M)}^2 d\dot x_1 \\
 \leq \sum_{[\beta] \leq \nu_0}
\int_M 
\sup_{\pi\in \Gh} \|\pi(X^\beta_{M,\dot x_1}\kappa_{\dot x_1})\|_{\sL(\cH_\pi)}^2 \|f\|_{L^2(M)}^2 d\dot x_1 ,
$$
by Lemma 	\ref{lem_RegRep}.
Hence, 
we have obtained:
\begin{align*}
\|\Op_M(\sigma) f\|^2_{L^2(M)} 
&\lesssim_M
\sum_{[\beta] \leq \nu_0}\int_M \|X^\beta_{M,\dot x_1} \sigma(\dot x_1, \cdot)\|_{L^\infty(\Gh)}^2 d\dot x_1\
\|f\|_{L^2(M)}^2\\
&\quad \leq |\vol (M)|
\sum_{[\beta] \leq \nu_0}\sup_{\dot x_1\in M}\|X^\beta_{M,\dot x_1} \sigma(\dot x_1, \cdot)\|_{L^\infty(\Gh)}^2 \
\|f\|_{L^2(M)}^2.
\end{align*}
This implies
\begin{equation}
\label{eq_OpMsigmaL2bdd}
	\|\Op_M(\sigma)\|_{\sL(L^2(M))}
\lesssim_M
\max_{[\beta] \leq \nu_0}\sup_{(\dot x_1,\pi)\in M\times \Gh }\|X^\beta_{M,\dot x_1} \sigma(\dot x_1, \pi)\|_{\sL(\cH_\pi)} =
 \|\sigma\|_{S^0,0,\nu_0,0},
\end{equation}
and concludes the proof. 
\end{proof}

In this paper, we will not discuss the characterisation of pseudo-differential operators by commutators with certain multipication and differentiation (also called Beals' characterisation). However, we can describe smoothing operators via their kernels or their action on Sobolev spaces on $M$:
\begin{proposition}
\label{prop_char_smoothingM}
	\begin{enumerate}
		\item 
		Given a distribution $K\in \cD'(M\times M)$, the operator  operator $T_K$ with integral kernel $K$ is a smoothing pseudo-differential operator on $M$  if and only if $K$ is smooth. Moreover, $K\mapsto T_K$ is an isomorphism of topological vector spaces $C^\infty (M\times M) \to \Psi^{-\infty}(M)$.
		\item 	If $T\in \sL(L^2(M))$ maps continuously $L^2_s(M)\to L^2_{s+N}(M)$ for any $s\in \bR$ and $N\in \bN$, then $T\in \Psi^{-\infty}(M)$ is a smoothing pseudo-differential operator. 
		Moreover, the converse is true, and the topology induced by $\cap_{s\in \bR, N\in \bN} \sL(L^2_s(M),L^2_{s+N}(M))$ coincides with the topology of $\Psi^{-\infty}(M)$.
	\end{enumerate}
\end{proposition}

We could obtain similar results for 
localised operators, i.e. $\chi_1 T \chi_2$
with cut-off functions $\chi_1,\chi_2$, characterised by their actions on local Sobolev spaces. However, as we will not use these results in the paper, we have not included them here.

\begin{proof}[Beginning of the proof of Proposition \ref{prop_char_smoothingM} (1)]
Here, we prove the implication  $K\in C^\infty (M\times M)\Rightarrow T_K \in \Psi^{-\infty}(M)$.
We fix $\chi\in \cD(G)$ such that 
$0\leq \chi \leq 1$ and $\chi=1$ on a fundamental domain $F_0$ of $M=\Gamma \backslash G$.
Then the function $K_\chi$ defined via
$$
K_\chi (\dot x,\dot y) = \sum_{\gamma\in \Gamma} \chi (y^{-1} \gamma x), \qquad \dot x, \, \dot y\in M,
$$
is smooth on $M\times M$ by the proof of Lemma \ref{lem_IntKernel}, with  $K_\chi \geq 1$.
We denote by $K_{\chi,G\times G}$ the corresponding function on $G\times G$.

Let $K\in C^\infty(M\times M)$. 
Denoting by $K_{G\times G} \in C^\infty (G\times G)^{\Gamma\times \Gamma}$ the corresponding function on $G\times G$, 
we set
$$
\kappa_{\dot x}(y) := \frac{K_{G\times G} }{K_{\chi,G\times G} }(x, xy^{-1}) \, \chi(y), 
\qquad x,y\in G.
$$
We check readily that this defines a smooth function in $\dot x,y$, supported in the  compact support of $\chi$  in $y$. Moreover, it satisfies for any $x,y\in G$
$$
\sum_{\gamma\in \Gamma} \kappa_{\dot x}(y^{-1}\gamma x) = \sum_{\gamma\in \Gamma} 
\frac{K_{G\times G} }{K_{\chi,G\times G} }(x, y)\, 
\chi(y^{-1}\gamma x) = K(\dot x,\dot y).
$$
This implies readily that $T_K=\Op_M(\sigma)$ 
is a smoothing pseudo-differential with  symbol $\sigma\in S^{-\infty}(M\times \Gh)$ given by $\sigma(\dot x,\pi)=\pi(\kappa_{\dot x})$.
\end{proof}

\begin{proof}[End of the proof of Proposition \ref{prop_char_smoothingM}]
Let $T\in \sL(L^2(M))$ be a continuous map $L^2_s(M)\to L^2_{s+N}(M)$ for any $s\in \bR$ and $N\in \bN$.
Let $K\in \cD'(M\times M)$ be its integral kernel. Then $K\in L^2(M\times M)$ with 
	$$
	\|K\|_{L^2(M\times M)} = \|T\|_{HS(L^2(M))}
	\leq C_s\|(\id+\cR_M)^{\frac s \nu}T\|_{\sL(L^2(M))}
	$$
	where $\cR$ is a positive Rockland operator on $G$ of homogeneous degree $\nu$, 
	and the constant
$C_s:=	\|(\id+\cR_M)^{-\frac s \nu} \|_{HS(L^2(M))}$
 is  finite by Proposition \ref{prop_I+cR_comptrHS} for $s>Q/2$.
We may apply the same reasoning to the integral kernel $X_M^\alpha (X_M^\beta)^t K\in L^2(M\times M)$ of $X_M^\alpha T X_M^\beta$.
This implies that $K\in C^\infty(M\times M)$ and therefore $T_K$ is smoothing by the beginning of Proposition \ref{prop_char_smoothingM} (1) already proven. 

The converse is true by the properties of the pseudo-differential calculus.
These properties also imply that the 
bijective map $T\mapsto T$ is continuous $\Psi^{-\infty}(M) \to \cap_{s\in \bR, N\in \bN} \sL(L^2_s(M),L^2_{s-N}(M))$;
it is therefore an isomorphism of topological vector spaces. This shows Part (2).

The arguments above show that 
if $K\in \cD'(M\times M)$ is the integral kernel of a smoothing operator $T_K$, then 
$K$ is smooth, 
and that the bijective map $K\mapsto T_K$ is an isomorphism of topological vector spaces $C^\infty (M\times M) \to \Psi^{-\infty}(M)$.
The end of Part (1) follows. 
\end{proof}

\subsection{Parametrices}
\label{subsec_param}

First let us recall the notion of left parametrix:
\begin{definition}
	An operator $A$ admits a \emph{left parametrix} in $\Psi^m(G)$ 
	when $A\in \Psi^m(G)$ and there exists $B\in \Psi^{-m}(G)$ such that 
	$BA - \id \in \Psi^{-\infty}(G)$.
	
	We have a similar definition on $\Psi^m(M)$.
\end{definition}

The existence of a left parametrix is an important quality as it implies the following regularity and spectral properties: 
\begin{proposition}
\label{prop_cqLpar}
If $A\in \Psi^m(G)$ admits a left-parametrix, then $A$ is hypoelliptic and satisfies sub-elliptic estimates 
$$
\forall s\in \bR, \ \forall N\in \bN_0, \ \forall f\in \cS(G)\quad
\|f\|_{L^2_s(G)}\lesssim \|Af\|_{L^2_{s-m}(G)} +\|f\|_{L^2_{-N}(G)}.
$$
The implicit constant above depends on $s,N,A$ and the realisations of the Sobolev norms, but not on $f$.

 On $M$, we have a similar property.
\end{proposition}
\begin{proof}
In the group setting, the hypoellipticity and subelliptic estimates were proved in \cite[Section 5.8.3]{R+F_monograph}.
By considering periodic symbols, similar properties hold on $M$. 
\end{proof}

We are able to construct left parametrices  when  symbols are invertible for high frequencies in the following sense:
\begin{definition}
\label{def_symbinv}
A symbol $\sigma$ 
is \emph{invertible}  in $S^m(G\times \Gh)$  for the  frequencies of a Rockland symbol $\widehat \cR$ higher than $\Lambda\in\bR$
(or just invertible for high frequencies)
when  for any $\gamma\in \bR$,  $x\in G$, $\mu$-almost every $\pi\in \Gh$, and $v\in \cH_{\pi,\cR,\Lambda}$,  we have 
$$
\|(\id+\pi(\cR))^{\frac{\gamma}\nu} \sigma(x,\pi)v_\pi\|_{\cH_\pi}
\geq C_\gamma \|(\id+\pi(\cR))^{\frac{\gamma+m}\nu} v_\pi\|_{\cH_\pi}
	$$
	with $C_{\gamma}=C_{\gamma,\sigma,G}>0$ a constant independent of $x$ and $\pi$.
	Above, $\nu$ is the homogeneous degree of $\cR$ and  $\cH_{\pi,\cR,\Lambda}$ is the subspace $\cH_{\pi,\cR,\Lambda}:= E_\pi[\Lambda ,\infty)\cH_\pi$ of $\cH_\pi$ where $E_\pi$ is the  spectral decomposition of $
\pi(\cR_G) = \int_\bR \lambda dE_\pi(\lambda)$ as in \eqref{eq_spectralmeas}.

We have a similar definition on $S^m(M\times \Gh)$.
\end{definition}

In the group case, 
we recognise the property defined in \cite[Definition 5.8.1]{R+F_monograph}, that was called ellipticity there. Using the vocabulary of ellipticity  is  cumbersome for this as it does not coincide with the usual notion of ellipticity considered in the abelian case $G=(\bR^n,+)$. Indeed, ellipticity in the abelian sense is defined for symbols with homogeneous asymptotic expansions. 
In any case, the properties explained in \cite[Section 8.1]{R+F_monograph} hold.
For instance, in Definition \ref{def_symbinv}, it suffices to prove the estimate for a sequence of $\gamma=\gamma_\ell$, $\ell\in \bZ$,   with $\lim_{\ell\to \pm \infty} \gamma_\ell =\pm \infty$. 
From the examples given in \cite[Section 5.8]{R+F_monograph}, if $\cR$ is a positive Rockland operator of  homogeneous degree $\nu$, then $\cR\in \Psi^\nu(G)$, $(\id+\cR)^{m/\nu}\in \Psi^m(G)$, 
$\cR_M\in \Psi^\nu(M)$, $(\id+\cR_M)^{m/\nu}\in \Psi^m(M)$ are operators with invertible symbols for the high frequencies of $\widehat \cR$.

As mentioned above, 
the invertibility of a symbol (even modulo lower order terms) allows us to construct a left parametrix for the corresponding operator:
\begin{theorem}
\label{thm_Lparexists}
\begin{itemize}
	\item Let $A\in \Psi^m(G)$. Assume that the symbol of $A$ may be written as 
	$$
	(\Op_G)^{-1}(A)=\sigma_0+\sigma_1, 
	\quad\mbox{with} \quad  \sigma_0\in S^m (G\times \Gh), \quad  \sigma_1\in S^{m_1} (G\times \Gh), \quad  m_1<m,
	$$
	 and that $\sigma_0$ is invertible in $S^m (G\times \Gh)$ for high frequencies. 
Then $A$ admits a left parametrix; therefore the properties in Proposition \ref{prop_cqLpar} hold.
\item We have a similar property on $M$. Moreover,  if $m>0$, then $A$ has compact resolvent on $L^2(M)$. Consequently, it has discrete spectrum and its eigenspaces are finite dimensional. 
\end{itemize}
\end{theorem}

\begin{proof}
For the  case of $\sigma_1=0$, this is showed in the proof of \cite[Theorem 5.8.7]{R+F_monograph}. 
Let us recall its main lines. Let $\sigma_0$ 
be invertible  in $S^m(G\times \Gh)$  for the  frequencies of a Rockland symbol $\widehat \cR$ higher than $\Lambda\in\bR$. 
We fix $\Lambda_1,\Lambda_2\geq 0$ with $\Lambda <\Lambda_1<\Lambda_2$, and  $\psi\in C^\infty (\bR)$ with $\psi =0$ on $(-\infty, \Lambda_1)$ and $\psi=1$ on $(\Lambda_2,\infty)$.
Then $\psi (\widehat \cR)\sigma_0^{-1}$ is a well defined symbol in $S^{-m}(G\times \Gh)$ \cite[Proposition 5.8.5]{R+F_monograph}
and $\Op_G(\psi (\widehat \cR)\sigma_0^{-1}) \Op_G(\sigma_0) = \id $ mod $\Psi^{-1}(G)$ 
by the symbolic  properties of the calculus. 
We then conclude in the usual way  
as in the proof of \cite[Theorem 5.8.7]{R+F_monograph}.

When $\sigma_1$ is not necessarily 0, we observe that 
$$
\Op_G(\psi (\widehat \cR)\sigma_0^{-1}) A =
\Op_G(\psi (\widehat \cR)\sigma_0^{-1}) \Op_G(\sigma_0)
+
\Op_G(\psi (\widehat \cR)\sigma_0^{-1}) \Op_G(\sigma_1)
$$
is still equal to $\id $ mod $\Psi^{-\min(m_1,1)}(G)$, and we can proceed as above.  
This shows the existence of a left parametrix for $A\in \Psi^m(G)$. 

If $m>0$ and $\lambda\in \bC$ is in the resolvent set of $A\in \Psi^m(M)$, then applying the above properties to $A-\lambda$, 
we have
$B(A-\lambda)=\id +R$ for some $B\in \Psi^{-m}(M)$ and $R\in \Psi^{-\infty}(M)$ (depending on $\lambda$), so that  $(A-\lambda)^{-1} = B-R	(A-\lambda)^{-1}$ is continuous $L^2(M)\to L^2_m(M)$.
 By Proposition \ref{prop_I+cR_comptrHS}, 
 the operator $(A-\lambda)^{-1}$  is therefore also compact $L^2(M)\to L^2_{m_1}(M)$ for $m_1< m$, in particular for $m_1=0$. 
\end{proof}

Recall that if $T\in \Psi^\infty(G)$ is a pseudo-differential operator, then $T^*$ denotes its formal adjoint, that is, the operator acting continuously on $\cS(G)$ and on $\cS'(G)$ given by 
$$
\int_G T^*f(x) \, \bar g(x) \, dx =\int_G f(x)\, \overline {Tg(x)}\, dx
\qquad f,g\in\cS'(G), \ \mbox{with a least one of them in}\ \cS(G).
$$
This may be different from the $L^2$-adjoint operator $T^\dagger$ of the (possibly unbounded) operator $T$ densely defined on $\cS(G)\subset L^2(G)$.
 A pseudo-differential operator $T\in \Psi^\infty(G)$  is symmetric when $T$ coincides with its formal adjoint $T^*$, or equivalently when
$$
(Tf,f)_{L^2}=(f,Tf)_{L^2}\quad 
\forall f\in \cS(G)\ (\mbox{or more generally} \ \cap_{s\in \bR} L^2_s(G)).
$$
We have similar properties and vocabulary on $M$. 

\begin{corollary}
\label{cor_Lparexists}	
We continue with the setting of Theorem \ref{thm_Lparexists}. We assume furthermore that  $m>0$	 and that $A$ is symmetric with formally self-adjoint symbol $\sigma=\sigma_0^*$. 
Then $A$ is essentially self-adjoint on $\cS(G)\subset L^2(G)$. 

We have a similar property on $M$. 
\end{corollary}
\begin{proof}
We will start the proof with the easier case of the compact nilmanifold. 
Let $g\in L^2(M)$ such that $((A+i)f , g)_{L^2}=0$ for any $f\in \cD(M)$. This implies that $(A+i)^* g=0$.
By the properties of the symbolic calculus, since $\sigma_0=\sigma^*$ and $m>0$, $(A+i)^*$ satisfies the hypotheses of Theorem \ref{thm_Lparexists} and is therefore hypoelliptic. Consequently, $g\in \cD(M)$. We have $0=((A+i)g,g)_{L^2} = (Ag,g)_{L^2} + i\|g\|_{L^2}^2$, implying $g=0$ by symmetry of $A$. The same property with $A-i$ holds, and this implies that the symmetric operator $A$ is essentially self-adjoint. 

The proof on $G$ is similar up to a few modifications. 
Let $g\in L^2(G)$ such that $((A+i)f , g)_{L^2}=0$ for any $f\in \cS(G)$, and more generally for any 
$f\in \cap_{s\in \bR} L^2_s(G)$. 
This implies that $(A+i)^* g=0$ in the sense of distributions. 
 By the symbolic properties of the  calculus, since $\sigma_0=\sigma^*$ and $m>0$, the pseudo-differential operator $(A+i)^*$ satisfies the hypotheses of Theorem \ref{thm_Lparexists}; 
 consequently, it satisfies subelliptic estimates which imply  that $g\in \cap_{s\in \bR} L^2_s(G)$. 
 The equality above and the symmetry of $A$ implies  $g=0$. The same property with $A-i$ holds, and  $A$ is essentially self-adjoint. 
	\end{proof}

We will discuss examples of invertible symbols and operators admitting left parametrices, starting with sub-Laplacians in horizontal divergences in Section \ref{subsec_cL_a} below
and their generalisations in Section \ref{subsec_generalisationLA}.

\subsection{Case of  symbols $\sigma_0\geq 0$ with $\id+\sigma_0$ invertible}

In many applications in this paper, we will consider a symbol $\sigma_0$ that is non-negative in the sense explained below, and such that $\id+\sigma_0$ is invertible in the sense of Definition \ref{def_symbinv}.
This is the case for the sub-Laplacians in horizontal divergences in Section \ref{subsec_cL_a} below
and their generalisations in Section \ref{subsec_generalisationLA}.

\smallskip

For a differential or pseudo-differential operator $T$ on $G$ or $M$, 
being a non-negative operator has an unambiguous definition:
	\begin{equation}
		\label{eq_nonnegT}
		\forall f\in \cD(G)\ \mbox{or} \ \cD(M)\qquad (Tf,f)_{L^2}\geq 0.
	\end{equation}
As our symbols act on the space of smooth vectors, we have a similar notion:
\begin{definition}
\label{def_nonnegsymb}
Let $\sigma$ be a symbol in $S^\infty(G\times \Gh)$. 
	It is \emph{non-negative} when for almost all $(x,\pi)\in G\times \Gh$, we have
	$$
	\forall v\in \cH_\pi^{\infty} \qquad (\sigma(x,\pi)v_\pi ,v_\pi)_{\cH_\pi}\geq 0.
	$$
\end{definition}

Our first observation in the following statement is that for 
a non-negative symbol $\sigma_0$  such that $\id+\sigma_0$ is invertible, 
 $\sigma_0(x,\pi)$ is a well defined operator for each $(x,\pi)\in G\times \Gh$.
In particular, we avoid the complication of viewing $\sigma_0$ as a measurable field of operators modulo a null set for the Plancherel measure:
\begin{lemma}
\label{lem_sigma01prop}
Let $\sigma_0\in S^m(G\times \Gh)$. 
We assume that $m\neq 0$, that $\sigma$ is a non-negative symbol and that $\id +\sigma_0$ is invertible for all  frequencies.
Then $(\id+\sigma_0)^{-1}\in S^{-m}(G\times \Gh)$ and,
for each  $(x,\pi)\in G\times \Gh$, $\sigma_0(x,\pi)$ may be realised as an essentially self-adjoint on $\cH_\pi^\infty \subset \cH_\pi$, with a discrete spectrum in $[0,\infty)$ and finite dimensional eigenspaces. 

Similar properties hold true on $M$.
\end{lemma} 
\begin{proof}
	By \cite[Proposition 5.8.5]{R+F_monograph}, 
	$\id +\sigma_0$ being invertible for all the frequencies of $\widehat \cR$
	 implies that $(\id+\sigma_0)^{-1}\in S^{-m}(G)$.
Consequently, $(\id+\sigma_0)^{-N}\in S^{-mN}(G)$ for any $N\in \bN$.
If $m>0$ and $-Nm<-Q$, then by Corollary \ref{cor_sigmacompop}
we may view $(\id+\sigma_0(x,\pi))^{-N}$ as a well defined compact operator on $\cH_\pi$ for every point $(x,\pi)$ in $G\times \Gh$.
If $m<0$, $(\id+\sigma_0)^{N}\in S^{-mN}(G)$ and we conclude the same way when $Nm<-Q$. 
\end{proof}

As each $\sigma_0(x,\pi)$ is essentially self-adjoint 
on $\cH_\pi^\infty \subset \cH_\pi$, 
we can  define its functional calculus, at least in an  abstract way. Let us show that it is in fact given by symbols in some classes in $S^{\infty}(G\times \Gh)$ or 
$S^{\infty}(M\times \Gh)$:

\begin{theorem}
\label{thm_psi(sigma0)}
Let $\sigma_0\in S^m(G\times \Gh)$. 
We assume that $m> 0$, that $\sigma$ is a non-negative symbol and that $\id +\sigma_0$ is invertible for all  frequencies.
\begin{enumerate}
	\item If $\psi\in \cS(\bR^n)$ then $\psi(\sigma_0)\in S^{-\infty}(G\times \Gh)$. 
	Furthermore, the map $\psi\mapsto \psi(\sigma_0)$ is continuous $\cS(\bR^n)\to S^{-\infty}(G\times \Gh)$.
	\item If $\psi\in \cG^{m'}(\bR)$ with $m'\in \bR$ then $\psi(\sigma_0)\in S^{mm'}(G\times \Gh)$. 
	Furthermore, the map $\psi\mapsto \psi(\sigma_0)$ is continuous $\cG^{m'}(\bR)\to S^{mm'}(G\times \Gh)$.
\end{enumerate}	
Similar properties hold true on $M$.
\end{theorem}

The proof relies	 on the following resolvent bounds:
\begin{proposition}
	\label{prop_sigma0Res}
	Let $\sigma_0\in S^m(G\times \Gh) $.
We assume that $m > 0$, that $\sigma$ is a non-negative symbol and that $\id +\sigma_0$ is invertible for all  frequencies.
For any $(x,\pi)\in G\times \Gh$ and $z\in \bC\setminus\bR$, the operator 
$(z-\sigma_0(x,\pi))^{-1}$ is  bounded  on $\cH_\pi$.
 The resulting field  of operators 
$(z-\sigma_0)^{-1}$ is in $S^{-m}(G)$. 
Moreover,
	for any semi-norm $\|\cdot\|_{S^{-m},a,b,c}$, there exist constant $C>0$ and powers $p\in \bN$ such that we have
$$
\forall z\in \bC\setminus\bR \qquad 
    \| (z-\sigma_0)^{-1}\|_{S^{-m},a,b,c}
	\leq C \left(1+\frac {1+|z|}{|\IM \, z|}\right)^p.
$$
Similar properties hold true on $M$.
\end{proposition}

We make the following two observations. 
The first one concerns the dependence of the constants in $\sigma_0$:
\begin{remark}
\label{remprop_sigma0Res} 
\begin{enumerate}
	\item  The proof below shows that the constant $C$ in Proposition \ref{prop_sigma0Res} may be chosen to depend on $\sigma_0$ in the following way:  
		$$
		C = C' \max_{|\gamma_0|\leq k} \| (\id+ \sigma_0)^{\gamma_0}\|_{ S^{m\gamma_0},a',b',c'}
		$$
		for some constant $C'>0$ and some integers $k,a',b',c'\in \bN_0$ depending on the seminorm $\| \cdot  \|_{S^{-m},a,b,c}$ but not on $\sigma_0$.
		\item We can in fact improve the bounds given  in Proposition \ref{prop_sigma0Res}:
For any semi-norm $\|\cdot\|_{S^{-m},a,b,c}$, there exist constant $C>0$ and powers $p\in \bN$ such that we have
$$
\forall z\in \bC\setminus\bR \qquad 
    \| (z-\sigma_0)^{-1}\|_{S^{-m},a,b,c}
	\leq C \left(\frac {1+|z|}{|\IM \, z|}\right)^{p+1}.
$$
However, as we will not use these more precise bounds, we do not present the proof. 
\end{enumerate}
\end{remark}

The proofs of Proposition \ref{prop_sigma0Res} and Theorem \ref{thm_psi(sigma0)} will be given below in the next section.

\subsection{Proofs of Proposition \ref{prop_sigma0Res} and Theorem \ref{thm_psi(sigma0)}} 
\label{subsec_pfprop+thm_sigma0}

\subsubsection{Proof of Proposition \ref{prop_sigma0Res}}
We assume that $m > 0$, that $\sigma$ is a non-negative symbol and that $\id +\sigma_0$ is invertible for all the frequencies of a positive Rockland symbol $\widehat \cR$ with homogeneous degree $\nu_\cR=\nu$.
Let us first establish some estimates. 
We have 
by functional analysis:
\begin{equation}
\label{eq1pf_prop_sigma0Res}
    \|(z-\sigma_0(x,\pi))^{-1}(\id+ \sigma_0(x,\pi)) \|_{\sL(\cH_\pi)} \leq \sup_{\lambda \geq0} \Big| \frac{1+\lambda}{z-\lambda} \Big| 
    =\sup_{\lambda \geq0} \Big| -1+ \frac{1+z}{z-\lambda} \Big| 
    \leq 1+ \frac{1+|z|}{|\IM z|}. 
\end{equation}
By the properties of the pseudo-differential calculus,  for any $\gamma_0\in \bZ$, the following quantity is finite:
$$
\sup_{(x,\pi)\in G\times \Gh}\left(
\|\pi(\id+\cR)^{-\frac {\gamma_0m} \nu} (\id+ \sigma_0(x,\pi))^{\gamma_0} \|_{\sL(\cH_\pi)}, 
\|(\id+ \sigma_0(x,\pi))^{\gamma_0} \pi(\id+\cR)^{-\frac {\gamma_0m} \nu} \|_{\sL(\cH_\pi)}\right)
 \leq C(\gamma_0),
$$
since $(\id+ \sigma_0)^{\gamma_0} \in S^{m\gamma_0}(G\times \Gh)$ while 
$\pi(\id+\cR)^{-\frac {\gamma_0m} \nu}
\in S^{-m\gamma_0}(G\times \Gh)$; 
the constant $C(\gamma_0)$ may be described as a seminorm in $(\id+\sigma_0)^{\gamma_0}$ up to a constant of $\gamma_0$.

To estimate the seminorm $\|(z-\sigma_0)^{-1}\|_{S^{-m},0,0,c}$, we consider
\begin{align*}
	&\|\pi(\id+\cR)^{\frac {m +\gamma} \nu} 
(z-\sigma_0(x,\pi))^{-1}\pi(\id+\cR)^{-\frac {\gamma} \nu} \|_{\sL(\cH_\pi)}\\
&\quad= \|\pi(\id+\cR)^{\frac {m +\gamma} \nu} (\id+ \sigma_0(x,\pi))^{-\frac{m+\gamma}{m}}
(z-\sigma_0(x,\pi))^{-1}(\id+ \sigma_0(x,\pi))\\
&\phantom{\quad= \|\pi(\id+\cR)}{} \times (\id+ \sigma_0(x,\pi))^{\frac{\gamma}{m}}  \pi(\id+\cR)^{-\frac {\gamma} \nu} \|_{\sL(\cH_\pi)} 
\\
&\quad\leq  \|\pi(\id+\cR)^{\frac {m +\gamma} \nu} (\id+ \sigma_0(x,\pi))^{-\frac{m+\gamma}{m}} \|_{\sL(\cH_\pi)} 
\|(z-\sigma_0(x,\pi))^{-1}(\id+ \sigma_0(x,\pi)) \|_{\sL(\cH_\pi)} 
\\
&\phantom{\quad= \|\pi(\id+\cR)}{} \times \|(\id+ \sigma_0(x,\pi))^{\frac{\gamma}{m}}  \pi(\id+\cR)^{-\frac {\gamma} \nu} \|_{\sL(\cH_\pi)} \\
&\quad\leq C(-1 -\frac{\gamma}m) C(\frac \gamma m)  \left(1+ \frac{1+|z|}{|\IM z|}\right),
\end{align*}
when  $\gamma\in m\bZ$, by \eqref{eq1pf_prop_sigma0Res}.
By interpolation we obtain for any $\gamma\in \bR$ that
\begin{equation}
\label{eq2pf_prop_sigma0Res}
	\sup_{(x,\pi)\in G\times \Gh}\|\pi(\id+\cR)^{\frac {m +\gamma} \nu} 
(z-\sigma_0(x,\pi))^{-1}\pi(\id+\cR)^{-\frac {\gamma} \nu} \|_{\sL(\cH_\pi)}
\leq 
C'_\gamma \left(1+ \frac{1+|z|}{|\IM z|}\right).
\end{equation}
Consequently we have the estimate for the seminorms $\|(z-\sigma_0)^{-1}\|_{S^{-m},0,0,c}$, $c\in \bR$. 

Now we turn our attention to the seminorm $\|(z-\sigma_0)^{-1}\|_{S^{-m},a,b,c}$  where $a$ and/or $b$ are different from zero. We start with the case where one of them is one and the other zero. 
If $X$ is a left-invariant vector field on $G$ or $|\alpha|=1$, 
then we compute
$$
X(z-\sigma_0)^{-1} = (z-\sigma_0)^{-1} X\sigma_0\ (z-\sigma_0)^{-1},
\quad\mbox{and}\quad
\Delta_\alpha(z-\sigma_0)^{-1}=(z-\sigma_0)^{-1} \Delta_\alpha\sigma_0 \ (z-\sigma_0)^{-1},
$$
so, 
\begin{align*}
&	\|\pi(\id+\cR)^{\frac {m +\gamma} \nu} X_x(z-\sigma_0(x,\pi))^{-1}\pi(\id+\cR)^{-\frac {\gamma} \nu}\|_{\sL(\cH_\pi)} 
\\&\quad = 
\|\pi(\id+\cR)^{\frac {m +\gamma} \nu} 
(z-\sigma_0(x,\pi))^{-1} X\sigma_0(x,\pi)\ (z-\sigma_0(x,\pi))^{-1}
\pi(\id+\cR)^{-\frac {\gamma} \nu}\|_{\sL(\cH_\pi)} 
\\&\quad \leq 
\|\pi(\id+\cR)^{\frac {m +\gamma} \nu} 
(z-\sigma_0(x,\pi))^{-1}\pi(\id+\cR)^{-\frac {\gamma} \nu} \|_{\sL(\cH_\pi)} 
\\
&\qquad\qquad\times 
\| \pi(\id+\cR)^{\frac {\gamma} \nu}  X\sigma_0(x,\pi)   \pi(\id+\cR)^{-\frac {m+\gamma} \nu}\|_{\sL(\cH_\pi)} 	
\\
&\qquad\qquad\times 
\| \pi(\id+\cR)^{\frac {m+\gamma} \nu}   (z-\sigma_0(x,\pi))^{-1} \pi(\id+\cR)^{-\frac {\gamma} \nu}\|_{\sL(\cH_\pi)} 	.
\end{align*}
For the first and third term above, we use \eqref{eq2pf_prop_sigma0Res}, while the second term is bounded by a seminorm in $\sigma_0\in S^{m}(G\times \Gh)$.
We therefore obtain:
$$
\sup_{(x,\pi)\in G\times \Gh}
\|\pi(\id+\cR)^{\frac {m +\gamma} \nu} X_x(z-\sigma_0(x,\pi))^{-1}\pi(\id+\cR)^{-\frac {\gamma} \nu}\|_{\sL(\cH_\pi)}
\lesssim_\gamma \left(1+ \frac{1+|z|}{|\IM z|}\right)^2.
$$
We obtain a similar bound for $\Delta _\alpha(z-\sigma_0(x,\pi))^{-1}$.
Modifying the proof of \cite[Proposition 5.8.5]{R+F_monograph}, 
we obtain recursively formulae for $X^\beta \Delta_\alpha (z-\sigma_0)^{-1}$. These expressions and the types of estimates used above lead to the estimates for any $S^{-m}$-semi-norms given in the statement.
This concludes the proof of Proposition \ref{prop_sigma0Res}.

\subsubsection{Proof of Theorem \ref{thm_psi(sigma0)} }
\label{subsubsec_pfthm_psi(sigma0)}
The proof of Theorem \ref{thm_psi(sigma0)} (1) relies on 
the  properties of the resolvent of $\sigma_0$ proved in Proposition \ref{prop_sigma0Res} and on
the Helffer-Sj\"ostrand formula we now recall \cite{zworski,davies_bk,dimassi+Sj}.

If   $T$ is a self-adjoint operator densely defined on a separable Hilbert space $\cH$,  then 
the spectrally defined operator $\psi(T)\in \sL(\cH)$    is given  at least formally by  the Helffer-Sj\"ostrand formula:
\begin{equation}
\label{eq_HS}
\psi(T) = \frac 1{\pi} \int_\bC \bar \partial \tilde \psi(z)\
(T-z)^{-1} L(dz).	
\end{equation}
Above, $\tilde \psi$ is a suitable almost analytic extension of $\psi$,
and  $L(dz)=dxdy$, $z=x+iy$, is the Lebesgue measure on $\bC$.
Recall that an almost analytic extension of a function 
 $\psi\in C^\infty(\bR)$ is any function  $\tilde\psi\in C^\infty(\bC)$ satisfying 
    $$
    \tilde \psi\big|_{\bR} = \psi \qquad\text{and}\qquad \bar{\partial} \tilde\psi\big|_{\bR}=0, 
    \qquad
    \mbox{where}\quad \bar \partial =\frac12( \partial_x +i\partial_y).
    $$
Such almost analytic extensions with compact support are constructed for $\psi\in C_c^\infty (\bR)$ in \cite{zworski,davies_bk,dimassi+Sj}; in addition, they satisfy  $\bar \partial \tilde \psi(z)= O(|\IM\, z|^N)$ for a given $N$. For these, the Helffer-Sj\"ostrand formula holds \cite{zworski,davies_bk,dimassi+Sj}:
\begin{lemma}
\label{lem_HSCcinfty}
	For any $\psi\in C_c^\infty (\bR)$, 
	if $\tilde \psi\in C_c^\infty(\bC)$ is an almost analytical extension  of $\psi$ satisfying 
	$\tilde \psi (z) = O(|\IM z|^2)$, 
	then the Helffer-Sj\"ostrand formula in \eqref{eq_HS} holds for any self-adjoint operator $T$ with 
\begin{equation}
\label{eq_HSbound}
\|\psi(T)\|_{\sL(\cH_\pi)} = \frac 1{\pi}
\left\|\int_\bC \bar \partial \tilde \psi(z)
(T-z)^{-1} L(dz)\right\|_{\sL(\cH_\pi)}\leq  
\frac 1{\pi}\int_{\bC} |\bar \partial \tilde \psi(z)| |\IM\, z|^{-1} L(dz).
\end{equation}
\end{lemma}

It is also possible to construct almost analytic extensions for larger classes of functions $\psi$ for which the Helffer-Sj\"ostrand formula holds together with further properties. 
This is the case for  functions $\psi$ in $\cG^{m'}(\bR)$, $m'<-1$:
\begin{lemma}
\label{lem_tildepsicGm'}
Let $\psi\in \cG^{m'}(\bR)$ with $m'<-1$. Then we can construct an almost analytic extension $\tilde {\psi} \in C^\infty(\bC)$ to $\psi$ 
for which the Helffer-Sj\"ostrand formula in \eqref{eq_HS} holds for any self-adjoint operator $T$ with \eqref{eq_HSbound}.
Moreover,   we have for all $N\in\bN_0$,
$$
\int_{\bC} \big| \bar\partial \tilde \psi( z) \big| \left(\frac{1+|z| }{|\IM \, z|}\right)^N L(dz) \leq C_{N}  \|\psi\|_{\cG^{m'},N+3},
$$
with  the constant  $C_{N}>0$ depending on $N$ (and on the construction and on $m'$), but not on $\psi$.
\end{lemma}

 This is a standard construction, and its proof is  postponed  to Appendix~\ref{Appendix_almost_analytic}.

Proposition \ref{prop_sigma0Res} and the Helffer-Sj\"ostrand formula readily imply the following property:
\begin{corollary}
\label{cor_prop_sigma0Res}
We continue with the setting of Proposition \ref{prop_sigma0Res}.
	If $\psi\in \cG^{m'}(\bR)$ with $m'<-1$, then $\psi(\sigma_0)\in S^{-m}(G\times \Gh)$.
	Moreover, for any seminorm $\|\cdot\|_{S^{-m},a,b,c} $ and any $m'<-1$, there exists $C>0$ and 
	a seminorm $\|\cdot\|_{\cG^{m'},N}$ such that 
	$$
	\forall \psi\in \cG^{m'}(\bR), \ \forall t\in (0,1]\qquad 
	\|\psi(t\sigma_0)\|_{S^{-m},a,b,c} \leq C t^{-1} \|\psi\|_{\cG^{m'},N}.
	$$
	Similar properties hold true on $M$.	
\end{corollary}

\begin{proof} [Proof of Corollary \ref{cor_prop_sigma0Res}]
By the Helffer-Sj\"ostrand formula, we have
$$
\psi(t\sigma_0) = \frac 1{\pi} \int_\bC \bar \partial \tilde \psi(z)\
(t\sigma_0-z)^{-1} L(dz)
=
\frac {t^{-1}}{\pi}\int_\bC \bar \partial \tilde \psi(z)\
(\sigma_0-t^{-1}z)^{-1} L(dz).
$$
We consider here an almost analytic extension $\tilde \psi$ of $\psi\in \cG^{m'} (\bR)$ with $m'<-1$ as in Lemma \ref{lem_tildepsicGm'}.
By  Proposition \ref{prop_sigma0Res}, we have for any   $t\in (0,1]$,
\begin{align*}
\|\psi(t\sigma_0)\|_{S^{-m},a,b,c}
&\leq \frac {t^{-1}}{\pi}	\int_\bC |\bar \partial \tilde \psi(z)|\
\|(t^{-1}z-\sigma_0)^{-1}\|_{S^{-m},a,b,c}L(dz)
 \\
&\lesssim t^{-1}	\int_\bC |\bar \partial \tilde \psi(z)|
 \Big(1+ \frac{1+|t^{-1}z|}{|\IM\, t^{-1} z|} \Big)^p L(dz)\\
 &\lesssim  t^{-1}	\int_\bC |\bar \partial \tilde \psi(z)|
 \Big(1+ \frac{1+|z|}{|\IM\,  z|} \Big)^p L(dz)\\ 
 &\lesssim  t^{-1} \|\psi\|_{\cG^{m'},N},
\end{align*}
for some $N\in \bN_0$.
\end{proof}

This allows us to show Part (1) of Theorem \ref{thm_psi(sigma0)}.
\begin{proof}[Proof of Theorem \ref{thm_psi(sigma0)} (1)]
Let $\psi\in \cS(\bR)$. 
As $\sigma_0\geq 0$, we may assume that $\psi$ is supported in $(-\frac 12, +\infty)$.
Then for any $N\in \bN_0$, $\psi_N (\lambda) := (1+\lambda)^N \psi(\lambda)$ defines a Schwartz function on $\bR$.
The properties of the pseudo-differential calculus implies that  $(\id+\sigma_0)^{-N} \in S^{-Nm}(G\times \Gh)$ and we have
$$
\|\psi(\sigma_0)\|_{S^{(-N-1)m},a,b,c}
=
\|(\id+\sigma_0)^{-N} \psi_N(\sigma_0)\|_{S^{(-N-1)m},a,b,c}
\lesssim \|\psi_N(\sigma_0)\|_{S^{-m},a',b',c'}
\lesssim \|\psi_N\|_{\cG^{-2},N'},
$$
for some $a',b',c',N'\in \bN$, by Corollary \ref{cor_prop_sigma0Res} applied to $\psi_N$, $t=1$ and (to fix the ideas) $m'=-2$.	
\end{proof}

The proof of Part (2) of Theorem \ref{thm_psi(sigma0)} will require the following consequence of Corollary \ref{cor_prop_sigma0Res}. The latter is shown  with similar ideas as in the proof of Part (1) of Theorem \ref{thm_psi(sigma0)} above:
\begin{lemma}
\label{lem_psitsigma0}
We continue with the setting of Proposition \ref{prop_sigma0Res}.
For any $m_1\in \bR$ and any seminorm $\|\cdot\|_{S^{m_1 m} ,a,b,c}$, 
there exist a constant $C>0$ and a number $k_0\in \bN_0$ such that 
$$
\forall \psi\in C_c^\infty(\frac 12,2),\quad\forall t\in (0,1]\qquad
	\|\psi(t\sigma_0)\|_{S^{m_1 m},a,b,c} \leq C t^{m_1} \max_{k=0,\ldots,k_0} 
	\sup_{\lambda \geq 0} |\psi^{(k)}(\lambda)|.
$$
\end{lemma}
\begin{proof}
Let $\psi\in C_c^\infty(\frac 12,2)$.	
Corollary \ref{cor_prop_sigma0Res} gives the case of $m_1=-1$.
For any $N\in \bN$, consider $\psi_N(\lambda):=\lambda^{-N}\psi(\lambda)$.
We have
$$
\|\psi(t\sigma_0)\|_{S^{(N-1) m},a,b,c} 
=
t^N \|	\sigma_0^N \psi_N (t\sigma_0)\|_{S^{(N-1) m},a,b,c} 
\lesssim t^N \|	\psi_N (t\sigma_0)\|_{S^{- m},a',b',c'} ,
$$
by the properties of the pseudo-differential calculus. Applying Corollary \ref{cor_prop_sigma0Res} to $\psi_N$ gives the case of $m_1=N-1$. This is so for any $N\in \bN$.

Let $\chi\in C_c^\infty(\frac 14,4)$ be such that $\chi=1$ on $(1/2,2)$. Since $\psi=\chi\psi$, we have
$$
\|\psi(t\sigma_0)\|_{S^{-2 m},a,b,c}
=\|\chi(t\sigma_0)\psi(t\sigma_0)\|_{S^{-2 m},a,b,c}
\lesssim 
\|\chi(t\sigma_0)\|_{S^{-m},a_1,b_1,c_1}
\|\psi(t\sigma_0)\|_{S^{- m},a_2,b_2,c_2}, 
$$
by the properties of the pseudo-differential calculus. Applying Corollary \ref{cor_prop_sigma0Res} to $\psi$ and $\chi$ gives the case for $m_1=-2$. Recursively, we obtain the case of any $m_1=-1,-2,-3,\ldots.$.
Hence, the statement is proved for any $m_1\in \bZ$, and we conclude by interpolation. 
\end{proof}

We can now show Part (2) of Theorem \ref{thm_psi(sigma0)}. The proof will use the Cotlar-Stein Lemma.

\begin{proof}[Proof of Theorem \ref{thm_psi(sigma0)} (2)]
	Let $\psi\in \cG^{m'}(\bR^n)$. 
	Without loss of generality, 
	we may assume that $\psi$ is real-valued
	and that it is supported in $(2,\infty)$
	by Theorem \ref{thm_psi(sigma0)} (1).
	Let $(\eta_j)$ be a dyadic decomposition of $[0,+\infty)$, that is, $\eta_{-1}\in C_c^\infty(-1,1)$ and $\eta_0\in C_c^\infty (\frac 12,2) $ with 
	$$
	\sum_{j=-1}^\infty \eta_j(\lambda)=1 \ \mbox{for all}\ \lambda\geq 0, \quad\mbox{where}\quad \eta_j(\lambda):=\eta_0(2^{-j}\lambda).
	$$
	We may write for any $\lambda\geq 0$
	$$
	\psi(\lambda) = \sum_{j=0}^\infty 2^{jm'} \psi_j (2^{-j}\lambda), 
	\quad\mbox{where}\quad \psi_j (\mu):= 2^{-jm'} \psi(2^j \mu) \eta_0(\mu).
	$$
	We observe that 
	$$
	\psi_j \in C_c^\infty (\frac 12,2)
	\qquad \mbox{and} \qquad
	\sup_{\lambda\geq 0} |\psi_j ^{(k)}(\lambda)| \lesssim_k \|\psi\|_{\cG^{m'},k}
		$$
		for any $k\in \bN_0$
with an implicit constant independent of $j$.	
Let $\alpha,\beta\in \bN_0^n$ and $(x,\pi)\in G\times \Gh$.
For each $j\in \bN_0$, 
let us consider the operator 
$$
T_j (\alpha,\beta,\gamma; x,\pi):=
T_j := 2^{jm'} \pi(\id+\cR)^{\frac{-mm' +[\alpha]+\gamma}\nu}
\Delta_\alpha X^\beta_x
\psi_j (2^{-j}\sigma_0 (x,\pi ) )
\pi(\id+\cR)^{-\frac{\gamma}\nu}.
$$
We observe that since $\sigma_0$ is self-adjoint, so is $\psi_j(2^{-j}\sigma_0 (x,\pi ) )$ and $(\Delta_\alpha X_x^\beta\psi_j(2^{-j}\sigma_0 (x,\pi ) ))^*$ is equal to $\Delta_\alpha X_x^\beta \psi(2^{-j}\sigma_0 (x,\pi ) )$ up to a sign. 
In any case, we have
\begin{align*}
\|T_i^* T_j\|_{\sL(\cH_\pi)}
&=2^{(i+j)m'}\|\pi(\id+\cR)^{-\frac{\gamma}\nu}
\Delta_\alpha X^\beta_x
\psi_i (2^{-i}\sigma_0 (x,\pi ) )
\pi(\id+\cR)^{\frac{-2mm' +[\alpha]+\gamma}\nu}
\times \\
&\qquad \qquad \qquad\qquad\qquad \times 
\pi(\id+\cR)^{\frac{[\alpha]+\gamma }\nu}
\Delta_\alpha X^\beta_x
\psi_j (2^{-j}\sigma_0 (x,\pi ) )
\pi(\id+\cR)^{-\frac{\gamma}\nu}\|_{\sL(\cH_\pi)}	\\
&\leq 2^{(i+j)m'} \|\psi_i (2^{-i}\sigma_0 )\|_{S^{2mm'},[\alpha],[\beta],c_1}
\|\psi_j (2^{-j}\sigma_0 )\|_{S^{0},[\alpha],[\beta],c_2},
\end{align*}
for some computable indices $c_1,c_2\in \bN_0$.
Applying Lemma \ref{lem_psitsigma0}, we obtain 
$$
\|T_i^* T_j\|_{\sL(\cH_\pi)} \lesssim 2^{(i+j)m'} 2^{-i (2m')} \max_{k=0,\ldots,k_1} 
	\sup_{\lambda \geq 0} |\psi_i^{(k)}(\lambda)|
	\max_{k=0,\ldots,k_2} 
	\sup_{\lambda \geq 0} |\psi_j^{(k)}(\lambda)|,
$$
for some $k_1,k_2\in \bN_0$.
Therefore, with $k=\max (k_1,k_2)$,
$$
\|T_i^* T_j\|_{\sL(\cH_\pi)} \lesssim 2^{(j-i)m'} \|\psi\|_{\cG^{m'},k}^2.
$$
Similarly, we also have with other indices $c'_1,c'_2$ 
\begin{align*}
\|T_i^* T_j\|_{\sL(\cH_\pi)}
&\leq 2^{(i+j)m'} \|\psi_i (2^{-i}\sigma_0 )\|_{S^{0},[\alpha],[\beta],c'_1}
\|\psi_j (2^{-j}\sigma_0 )\|_{S^{2mm'},[\alpha],[\beta],c'_2}\\
& \lesssim 2^{(i-j)m'} \|\psi\|_{\cG^{m'},k}^2,
\end{align*}
after possibly modifying $k\in \bN_0$. 
We obtain similar operator bounds for $T_i T_j^*$, leading to
$$
\max(\|T_i^* T_j\|_{\sL(\cH_\pi)},\|T_i T_j^*\|_{\sL(\cH_\pi)})
\lesssim   
2^{-|i-j| |m'|} \|\psi\|_{\cG^{m'},k}^2.
$$
When $m'\neq 0$,
this allows us to apply the Cotlar-Stein Lemma \cite[\S VII.2]{Steinbk}: $\sum_j T_j$ converges in the strong operator topology to a bounded operator on $\cH_\pi$ with norm $\lesssim \|\psi\|_{\cG^{m'},k}$. The implicit constant is independent of $(x,\pi)\in G\times \Gh$.

With $x\in G$ fixed, this implies that the field of bounded operators 
$$
T(\alpha,\beta,\gamma;x,\pi):=
\sum_{j=0} T_j (\alpha,\beta,\gamma; x,\pi),
 \ \pi\in \Gh, 
 $$ is a well defined symbol in $L^\infty (\Gh)$.
 When $m'\leq 0$, $\psi$ is bounded and
 $\psi(\sigma_0(x,\pi))$ is a well defined bounded operator on $\cH_\pi$  by functional analysis. 
 Moreover, if $m'<0 $, the following operators coincide:
 $$
 T(0,0,0;x,\pi)=\pi(\id+\cR)^{\frac{-mm'}\nu}\psi(\sigma_0(x,\pi)).
 $$ 
 Considering the convolution kernels of 
 $T(\alpha,\beta,\gamma;x,\cdot )$
 and $T_j (\alpha,\beta,\gamma; x,\cdot)$, 
 routine checks 
imply readily that $\psi(\sigma_0)\in S^{mm'}(G\times \Gh)$. Furthermore, the estimate above implies the continuity $\cG^{m'}(\bR)\to S^{mm'}(G\times \Gh)$ of the map
$\psi\mapsto \psi(\sigma_0)$.

This proves the result for $m'< 0$. 
As  $(\id+\sigma_0)\in S^{m}(G\times \Gh)$, 
the properties of composition of the pseudo-differential calculus allow us to extend the result to any $m'\in \bR$.
\end{proof}

\subsection{Sub-Laplacians in horizontal divergence form}
\label{subsec_cL_a}
	Let $G$ be a stratified nilpotent Lie group. 
	We fix a basis $X_1,\ldots, X_{n_1}$ of the first stratum of its Lie algebra $\fg$, completed in a basis $X_1,\ldots,X_n$ of $\fg$ adapted to the stratification.
	We now identify elements of $\fg$ with left-invariant vector fields. 
	Let $a_{i,j}\in C^\infty_{l,b}(G)$, $1\leq i,j \leq n$. We consider	the differential operator  $\cL_A\in \Psi^2 (G)$.
	$$
	\cL_A:=-\sum_{1\leq i,j\leq n_1} X_i (a_{i,j}(x) X_j)
	= -\sum_{1\leq i,j\leq n_1} a_{i,j}(x) X_iX_j + (X_i a_{i,j}(x)) X_j
	$$
Its symbol may be written as 
$$
S^2(G\times \Gh)\ni \widehat \cL_A := -\sum_{1\leq i,j\leq n_1} a_{i,j} \widehat X_i \widehat X_j + (X_i a_{i,j}) \widehat X_j = \sigma_0 +\sigma_1
$$
where
$$
S^2(G\times\Gh)\ni  \sigma_0 := -\sum_{1\leq i,j\leq n_1} a_{i,j} \widehat X_i \widehat X_j
\quad\mbox{and}\quad 
\sigma_1  = -\sum_{1\leq i,j\leq n_1}  (X_i a_{i,j}) \widehat X_j\in S^{1}(G\times \Gh).
$$
Assuming that the matrix $A(x)=(a_{i,j}(x))$ is symmetric at every $x\in G$, 
	we say that  
	$\cL_A$
	is a sub-Laplacian in divergence form. 
	We observe that When  $A=(a_{i,j})$ is the identity matrix $\id$, then 
$\cL_\id:=-\sum_{j=1}^{n_1}  X_j^2$ is the canonical sub-Laplacian in this context, which is known to be a positive Rockland operator of homogeneous degree 2.

If the matrix $A(x)=(a_{i,j}(x))$ is non-negative at every point $x\in G$, then  the differential operator $\cL_A$ and the symbol $\sigma_0$ are  non-negative.
Under some additional natural condition on the matrices $A(x)=(a_{i,j}(x))$, the sub-Laplacian in divergence form $\cL_A$ admits a left parametrix. 
\begin{lemma}
\label{lem_LApar}
We continue with the above setting on $G$.	
We assume in addition the following hypothesis of uniform ellipticity: 
$$
c:=\inf_{x\in G} \lambda_{A(x),1}>0
\qquad\mbox{and}\qquad  
C:=\sup_{x\in G} \lambda_{A(x),n}<\infty,
$$
where $\lambda_{A(x),1}$ and $\lambda_{A(x),n}$ denote the  smallest and largest eigenvalues of the non-negative matrix $A(x)$.
Then $\id+\sigma_0$ is invertible for all the frequencies of $\widehat \cL_I$, and $\cL_A$ admits a left parametrix. 
Moreover, $\cL_A+A_1$ for any $A_1\in \Psi^1(G)$ admits a left parametrix, and
therefore the properties in Proposition \ref{prop_cqLpar} hold. 
If in addition $A_1$ is symmetric, then $\cL_A+A_1$ is essentially self-adjoint. 
\end{lemma}
\begin{proof}
By \eqref{eq_LinftyGh_L2bdd}, we have:
	\begin{align*}
	&\sup_{(x,\pi)\in G\times \Gh}
	\|(\id +\pi(\cL_\id))^{\frac{\gamma+2}2} (\id+\sigma_0(x,\pi))^{-1}(\id +\pi(\cL_\id))^{-\frac{\gamma}2})\|_{\sL(\cH_\pi)}
	\\&\qquad=
	\sup_{x\in G} \|(\id +\cL_\id)^{\frac{\gamma+2}2} (\id-\sum_{1\leq i,j\leq n_1} a_{i,j}(x) X_iX_j )^{-1}(\id +\cL_\id)^{-\frac{\gamma}2})\|_{\sL(L^2(G))},
	\end{align*}
	so this quantity  is finite by comparison of Rockland operators, see Corollary \ref{cor_exRgeneral}). In other words, $\id+\sigma_0$ is invertible for all the frequencies of $\widehat \cL_I$.
	By Theorem \ref{thm_Lparexists}, 
	this implies that, for any $A'_1\in \Psi^1(G)$, the operator
$$
\Op_G(\id+\sigma_0)+A'_1 = \cL_A -\Op_G(\sigma_1)+\id +A'_1
$$    
admits a left parametrix. 
\end{proof}

In the case of a compact nilmanifold $M$, 
we consider smooth  functions $a_{i,j}$, $1\leq i,j \leq n$,  on $M$.
Their derivatives with respect to $X_M^\alpha$ will be automatically bounded. 
We assume that the matrix $A(\dot x)=(a_{i,j}(\dot x))$ is symmetric.  
The differential operator 
	$$
	\cL_A:=-\sum_{1\leq i,j\leq n_1} X_{M,i} (a_{i,j}(\dot x) X_{M,j})
	= -\sum_{1\leq i,j\leq n_1} a_{i,j}(x) X_{M,i}X_{M,j} + (X_{M,i} a_{i,j}(\dot x)) X_{M,j}
	$$
	is a sub-Laplacian in divergence form on $M$. 
	We  assume that the matrix $A(\dot x)=(a_{i,j}(\dot x))$ is non-negative at every point $x\in G$.
The differential operator $\cL_A$ and the symbol $\sigma_0$ are then non-negative. Moreover, 
under a further hypothesis of ellipticity, 
similar properties as above hold:
	\begin{corollary}
	\label{cor_lem_LApar}
		We continue with the above setting on $M$.	
We assume that the matrix $A(\dot x)=(a_{i,j}(\dot x))$ is positive at every point $x\in G$.
Then the following properties hold:
\begin{enumerate}
	\item The symbol $\id+\sigma_0$ is invertible for all the frequencies of $\widehat \cL_I$, and $\cL_A$ admits a left parametrix. 
	\item For any $A_1\in \Psi^1(M)$, the operator $\cL_A+A_1$ admits a left parametrix, and
therefore the properties in Proposition \ref{prop_cqLpar} hold. 
\item The operator $\cL_A$ is essentially self-adjoint and has discrete spectrum with finite dimensional eigenspaces. It is also the case for $\cL_A+A_1$
for any symmetric  $A_1\in \Psi^1(M)$. 
\end{enumerate}
	\end{corollary}
	
	The proof of Corollary \ref{cor_lem_LApar} will follow readily  from generalisations given in 
	Section
\ref{subsec_generalisationLA} below, especially
	 Theorem \ref{thm_Lparexists} and Corollary \ref{cor_Lparexists}.

\subsection{Generalisations of sub-Laplacians in divergence form}
\label{subsec_generalisationLA}

The same reasonings as for Lemma \ref{lem_LApar} and  Corollary \ref{cor_lem_LApar} above regarding  sub-Laplacians in divergence form give the following classes of examples of operators non-negative differential operators whose symbols are non-negative and invertible for high frequencies
(and therefore admit left parametrices).

\begin{ex}
\label{ex_RAGstrat}
	Let $G$ be a stratified Lie group and let $X_1,\ldots,X_{n_1}$ a basis of the first stratum. 
	Fix $\nu_1\in \bN$ and set $\nu'_1=\#\{\alpha\in \bN_0^{n_1} \ : \ |\alpha|=\nu_1\}$.
		Let $a_{\alpha,\beta}\in C_{l,b}^\infty(G)$ with $\alpha,\beta\in \bN_0^{n_1}$.
		We assume that for each $x\in G$,
		the $\nu'_1\times \nu'_1$-matrix $A(x)=(a_{\alpha,\beta}(x))_{|\alpha|=|\beta|=\nu_1}$ is non-negative. 
		Then the symbol $\sigma_0$ 		given by
	$$
	\sigma_0(x,\pi):=
	\sum_{|\alpha|=|\beta|=\nu_1} a_{\alpha,\beta} (x)(\pi(X)^\beta)^t \pi(X)^\alpha,
	$$ 
is in $S^{\nu_1}(G\times \Gh)$,
and it is non-negative.

Assume  that $A(x)$ is positive for each $x\in G$ and that furthermore
		$$
		\inf_{x\in G} \lambda_{A(x),1} >0 ,\qquad\mbox{while}\qquad
		\sup_{x\in G} \lambda_{A(x),n}<\infty, 
		$$
		where   $\lambda_{A(x),1}$ and $\lambda_{A(x),n}$ are the  lowest and highest eigenvalues of $A(x)$. 
	Then proceeding as in the proof of Lemma \ref{lem_LApar}, 
		$\id+\sigma_0$
	is invertible in $S^{\nu_1}(G\times \Gh)$ for all the frequencies of the positive Rockland symbol $\sum_{|\alpha|=|\beta|=\nu_1} (\pi(X)^\beta)^t \pi(X)^\alpha$. 	
	Moreover, the non-negative differential operator 
	$$\cR_A:= \sum_{|\alpha|=|\beta|=\nu_1} (X^\beta)^t\  a_{\alpha,\beta} (x) \ X^\alpha\\
$$
	may be written as 
	$\Op_G(\sigma_0+\sigma_1)$ with $\sigma_1 \in S^{2\nu_1-1}(G\times \Gh)$. Consequently, $\cR_A$
	is in $\Psi^{2\nu_1}(G)$
and 	
	admits a left parametrix. Moreover, $\cR_A+A_1$ for any $A_1\in \Psi^{2\nu_1-1}(G)$, also admits a left parametrix. 
\end{ex}

\begin{ex}
\label{ex_RAMstrat}
	Let $M=\Gamma \backslash G$ be a compact nilmanifold with $G$, $X_1,\ldots,X_{n_1}$, $\nu_1$ and $\nu'_1$ as in Example \ref{ex_RAGstrat}.
		Let $a_{\alpha,\beta}\in C^\infty(M)$ with $\alpha,\beta\in \bN_0^{n_1}$, $|\alpha|=|\beta|=\nu_1$. 
		We assume that for each $x\in G$,
		the $\nu'_1\times \nu'_1$-matrix $A(\dot x)=(a_{\alpha,\beta}(\dot x))_{|\alpha|=|\beta|=\nu_1}$ is non-negative. 
		Then the symbol $\sigma_0$ 		given by
	$$
	\sigma_0(\dot x,\pi):=
	\sum_{|\alpha|=|\beta|=\nu_1} a_{\alpha,\beta} (\dot x)(\pi(X)^\beta)^t \pi(X)^\alpha,
	$$ 
is in $S^{\nu_1}(M\times \Gh)$ 
and it is non-negative.

Assume  that $A(\dot x)$ is positive for each $\dot x\in M$. Proceeding as in the proof of Lemma \ref{lem_LApar}, 
		$\id+\sigma_0$
	is invertible in $S^{\nu_1}(M\times \Gh)$ for all the frequencies of the positive Rockland symbol $\widehat \cR_I:=\sum_{|\alpha|=|\beta|=\nu_1} (\pi(X)^\beta)^t \pi(X)^\alpha$. 	
	Moreover, the differential operator 
	$$
	\cR_A:= \sum_{|\alpha|=|\beta|=\nu_1} (X_M^\beta)^t a_{\alpha,\beta} (\dot x) \, X_M^\alpha ,
	$$
		may be written as 
	$\Op_M(\sigma_0+\sigma_1)$ with $\sigma_1 \in S^{2\nu_1-1}(M\times \Gh)$. Consequently, $\cR_A$
	is in $\Psi^{2\nu_1}(M)$
and 	
	admits a left parametrix. Moreover, $\cR_A+A_1$ for any $A_1\in \Psi^{2\nu_1-1}(M)$, also admits a left parametrix. 
	Again, $\cR_A$ is non-negative. 
\end{ex}

\begin{ex}
\label{ex_RAGgrad}
		Let $G$ be a graded Lie group, and let $X_1,\ldots, X_n$ an adapted basis to its graded Lie algebra.
	Fix two common multiple $\nu_0,\nu_1$  of the weights $\upsilon_j$, $j=1,\ldots,n$ of the dilations.
	Set $Y_j:= X_j^{ \nu_0 /\upsilon_j}$, $j=1,\ldots, n$ and set $\nu'_1:=\#\{\alpha\in \bN_0^{n} \ : \ [\alpha]=\nu_1\}$.
	Let $a_{\alpha,\beta}\in C_{l,b}^\infty(G)$ with $\alpha,\beta\in \bN_0^{n}$, $[\alpha]=[\beta]=\nu_1$.
		We assume that for each $x\in G$,
		the $\nu'_1\times \nu'_1$-matrix $A(x)=(a_{\alpha,\beta}(x))_{[\alpha]=[\beta]=\nu_1}$ is non-negative. 
		Then the symbol $\sigma_0$ 		given by
	$$
	\sigma_0(x,\pi):=\sum_{[\alpha]=[\beta]=\nu_1}a_{\alpha,\beta}(x) (\pi(Y)^\beta)^t \pi(Y)^\alpha,
	\quad \mbox{where} \quad  Y^\alpha :=Y_1^\alpha \ldots Y_n^{\alpha_n},	
	$$
	is in $S^{\nu_1}(G\times \Gh)$ 
and it is non-negative.	
		
	Assume  that $A(x)$ is positive for each $x\in G$ and that furthermore	
		$$
		\inf_{x\in G} \lambda_{A(x),1} >0 ,\qquad\mbox{while}\qquad
		\sup_{x\in G} \lambda_{A(x),n}<\infty, 
		$$
		where   $\lambda_{A(x),1}$ and $\lambda_{A(x),n}$ are the  lowest and highest eigenvalues of $A(x)$. 
Then $\id+\sigma_0$
	is invertible in $S^{\nu_1}(G\times \Gh)$ for all the frequencies of the positive Rockland operator $\sum_{[\alpha]=[\beta]=\nu_1} (\pi(Y)^\beta)^t \pi(Y)^\alpha$. 
		Moreover, the non-negative differential operator 
	$$
	\cR_A:= 
	\sum_{[\alpha]=[\beta]=\nu_1}(Y^\beta)^t \ a_{\alpha,\beta}(x)  \ Y ^\alpha,
	$$		may be written as 
	$\Op_G(\sigma_0+\sigma_1)$ with $\sigma_1 \in S^{2\nu_1-1}(G\times \Gh)$. Consequently, $\cR_A$
	is in $\Psi^{2\nu_1}(G)$
and 	
	admits a left parametrix. Moreover, $\cR_A+A_1$ for any $A_1\in \Psi^{\nu_0\nu_1-1}(G)$, also admits a left parametrix. 
\end{ex}

\begin{ex}\label{ex_RAMgrad}
	Let $M=\Gamma \backslash G$ be a compact nilmanifold with $G$, $X_1,\ldots,X_n$, $\nu_1$, $\nu'_1$, $\nu_0$, $Y_1,\ldots, Y_n$ as in Example \ref{ex_RAGgrad}.
	Let $a_{\alpha,\beta}\in C^\infty(M)$ with $\alpha,\beta\in \bN_0^{n}$, $[\alpha]=[\beta]=\nu_1$.
		We assume that for each $x\in M$,
		the $\nu'_1\times \nu'_1$-matrix $A(\dot x)=(a_{\alpha,\beta}(\dot x))_{[\alpha]=[\beta]=\nu_1}$ is non-negative. 
		Then the symbol $\sigma_0$ 		given by
	$$
	\sigma_0(\dot x,\pi):=\sum_{[\alpha]=[\beta]=\nu_1}a_{\alpha,\beta}(\dot x) (\pi(Y)^\beta)^t \pi(Y)^\alpha,	
	$$
	is in $S^{\nu_1}(M\times \Gh)$ 
and it is non-negative.	
		
	Assume  that $A(\dot x)$ is positive for each $\dot x\in M$.
Then $\id+\sigma_0$
	is invertible in $S^{\nu_1}(M\times \Gh)$ for all the frequencies of the positive Rockland operator $\sum_{[\alpha]=[\beta]=\nu_1} (\pi(Y)^\beta)^t \pi(Y)^\alpha$. 
		Moreover, the non-negative differential operator 
	$$
	\cR_A:= 
	\sum_{[\alpha]=[\beta]=\nu_1}(Y_M^\beta)^t \ a_{\alpha,\beta}(\dot x)  \ Y_M ^\alpha\ \in \Psi^{\nu_0\nu_1}(M),
	$$		
	may be written as 
	$\Op_M(\sigma_0+\sigma_1)$ with $\sigma_1 \in S^{2\nu_1-1}(M\times \Gh)$. Consequently, $\cR_A$
	is in $\Psi^{2\nu_1}(M)$
and 	
admits a left parametrix. Moreover, $\cR_A+A_1$ for any $A_1\in \Psi^{\nu_0\nu_1-1}(M)$, also admits a left parametrix. 
\end{ex}

\section{Semiclassical pseudo-differential calculi on $G$ and $M$}\label{sec:semiclassical_calculus_G_M}

This section is devoted to the semiclassical quantization and pseudo-differential calculi 
obtained on $G$ and $M$.

\subsection{Semiclassical quantizations}

In this section, we consider a small parameter $\eps\in (0,1]$ and the semiclassical quantizations $\Op_G^{(\eps)}$ and $\Op_M^{(\eps)}$ on $G$ and $M$ given by dilations of the Fourier variable, that is,
$$
\Op_G^{(\eps)}(\sigma) = 
\Op_G(\sigma^{(\eps)})
\quad\mbox{and}\quad
\Op_M^{(\eps)}(\sigma) = 
\Op_M(\sigma^{(\eps)})
$$
where
\begin{equation}
\label{eq_sigmaeps}
\sigma^{(\eps)} (x,\pi)= \sigma(x,\eps\cdot  \pi), 
\quad\mbox{for almost all} \ (x,\pi)\in G\times \Gh
\ \mbox{or} \ M\times \Gh. 
\end{equation}
Note that if $\kappa$ is the convolution kernel of $\sigma$, then $\kappa^{(\eps)}$ given by 
\begin{equation}
\label{eq_kappaeps}	
\kappa^{(\eps)}_{x}(z):= \eps^{-Q} \kappa_{x}(\eps^{-1}\cdot y), \quad y\in G,
\end{equation}
is  the convolution kernel of $\sigma^{(\eps)}$.

In order to motivate our semiclassical calculus, we will start with the study of the asymptotics obtained by composition and adjoint for these semiclassical quantizations. 
The proof can be found in Appendix  \ref{secA_pfthm_sclexp_prod+adj}.
\begin{theorem}
\label{thm_sclexp_prod+adj}
Let $m_1,m_2,m\in \bR$.
\begin{enumerate}
	\item If $\sigma_1 \in S^{m_1}(G\times \Gh)$ 
and $\sigma_2 \in S^{m_2}(G\times \Gh)$,  denoting by 
$$
\sigma := \sigma_1 \diamond_\eps \sigma_2
$$
 the ($\eps$-dependent) symbol such that  
$$
\Op_G^{(\eps)}(\sigma)=
\Op_G^{(\eps)}(\sigma_1)\Op_G^{(\eps)}(\sigma_2)
\in \Psi^{m_1+m_2}(G\times \Gh),
$$
then for any $N\in \bN_0$ and 
for any semi-norm $\|\cdot\|_{S^{m_1+m_2-(N+1)},a,b,c}$,
the following quantity is finite:
$$
\sup_{\eps\in (0,1]} \eps^{-(N+1)} \|\sigma - \sum_{[\alpha]\leq N } \eps^{[\alpha]} \Delta^\alpha \sigma_1 \, X^\alpha \sigma_2 \| _{S^{m_1+m_2-(N+1)},a,b,c}<\infty .
$$
In fact, this quantity is bounded, up to a constant of $G,N,m,a,b,c$ by 
semi-norms (depending only on $G,N,m,a,b,c$) in $\sigma_1 \in S^{m_1}(G\times \Gh)$ 
and $\sigma_2 \in S^{m_2}(G\times \Gh)$.

\item If $\sigma\in S^m(G\times \Gh)$ and denoting by $\sigma^{(\eps,*)}$ the ($\eps$-dependent) symbol such that  
$$
\Op_G^{(\eps)}(\sigma^{(\eps,*)})=(\Op_G^{(\eps)}(\sigma))^*\in \Psi^{m}(G\times \Gh),
$$
then for any $N\in \bN_0$ and 
for any semi-norm $\|\cdot\|_{S^{m-(N+1)},a,b,c}$,
the following quantity is finite:
$$
\sup_{ \eps\in (0,1]}\eps^{-(N+1)}
\|\sigma^{(\eps,*)}- \sum_{[\alpha]\leq N } \eps^{[\alpha]}
 \Delta^\alpha  X^\alpha \sigma^*
\|_{S^{m-(N+1)},a,b,c}<\infty .
$$
In fact, this quantity is bounded, up to a constant of $G,N,m,a,b,c$ by a
semi-norm (depending only on $G,N,m,a,b,c$) in $\sigma\in S^m(G\times \Gh)$.
\end{enumerate}
\end{theorem}

Naturally, we have a similar statement for symbols on $M$, that is, in $S^\infty(M\times \Gh)$.

We observe that Theorem \ref{thm_sclexp_prod+adj} would also 
hold under the weaker hypothesis that the symbols $\sigma_1,\sigma_2,\sigma$ depend on $\eps$ uniformly in the following sense:

\begin{definition}
\label{def_unif}
Let $\sigma(\eps)$, $\eps\in (0,1]$, be a family of symbols in $S^m(G\times \Gh)$. 
It  is  uniformly in $S^m(G\times \Gh)$ for 
$\eps\in (0,1]$	 when for any semi-norm $\|\cdot \|_{S^m,a,b,c}$, the following quantity 
$$
\sup_{\eps\in (0,1]} \|\sigma(\eps)\|_{S^m,a,b,c}<\infty, 
$$
 is finite.
We have a similar definition in $S^m(M\times \Gh)$, and we adopt the same vocabulary for a family of operators being uniformly in $\Psi^m(G)$ or $\Psi^m(M)$ for 
$\eps\in (0,1]$. 
\end{definition}

In the rest of the paper, when considering a family of symbols uniformly in some $S^m(M\times \Gh)$, we will often omit to indicate the dependence of the symbols in $\eps\in (0,1]$.

The semiclassical quantization $\Op_G^{(\eps)}$ of a family 
of symbols depending on $\eps\in (0,1]$ uniformly in $S^0(G\times \Gh)$ act uniformly on $L^2(G)$, and this generalises to Sobolev spaces and to a similar property on $M$:
\begin{theorem}
\label{thm_sclcontSob}
	\begin{enumerate}
		\item If $\sigma$ is a family of symbols depending on $\eps\in (0,1]$ uniformly in $S^0(G\times \Gh)$, then 
		the operators $\Op^{(\eps)}_G (\sigma)$, $\eps\in (0,1]$, are bounded on $L^2(G)$; furthermore, they are uniformly bounded in the sense that  
$$
\sup_{\eps\in (0,1]} \|\Op^{(\eps)}_G (\sigma)\|_{\sL(L^2(G))}<\infty,
$$
is finite. In fact, this quantity is bounded up to a structural constant, by $\sup_{\eps\in (0,1]} \|\sigma\|_{S^0,a,b,c}$ for some $a,b,c\in \bN_0$ independent of $\sigma$.  
\item If $\sigma$ is a family of symbols depending on $\eps\in (0,1]$ uniformly in $S^m(G\times \Gh)$, then 
		the operators $\Op^{(\eps)}_G (\sigma)$, $\eps\in (0,1]$, are bounded $L^2_s(G)\to L^2_{s-m}(G)$; furthermore, they are uniformly bounded and there exists a constant $C>0$ and a semi-norm $\|\cdot \|_{S^m,a,b,c}$
		such that  
$$
\forall \eps\in (0,1],\ 
\forall f\in \cS(G)\qquad 
\|\Op^{(\eps)}_G (\sigma)f\|_{L^2_{s-m}(G),\eps^\nu \cR} 
\leq 
C \ \left(\sup_{\eps'\in (0,1]} \|\sigma\|_{S^m,a,b,c}\right)
\|f\|_{L^2_{s}(G),\eps^\nu \cR} .
$$
\item We have similar properties on $M$.
	\end{enumerate}
\end{theorem}

\begin{proof}
Let us prove Part (1). The $L^2$-boundedness of each operator $\Op^{(\eps)}_G (\sigma(\eps))$, $\eps\in (0,1]$
 follows from Theorem \ref{thm_PDOGcomp+adj} (1).
 By Remark \ref{rem_thm_PDOGcomp+adj_L2bddness} (with its notation), the operator norms are estimated by 
 \begin{align*}
 	\|\Op^{(\eps)}_G (\sigma)\|_{\sL(L^2(G))} &\leq C\left( \max_{[\beta]\leq 1+Q/2} \sup_{(x,\pi)\in G\times \Gh}\|X^\beta_x \sigma(x,\eps \cdot \pi)\|_{\sL(\cH_\pi)} \ + \ \sup_{x\in G} \||\cdot|_p^{pr}\kappa_x^{(\eps)} \|_{L^2(G)}\right)
 	\\& =C \left( \max_{[\beta]\leq 1+Q/2} \sup_{(x,\pi)\in G\times \Gh}\|X^\beta_x \sigma(x, \pi)\|_{\sL(\cH_\pi)} \ + \ \eps^{pr-Q/2}\sup_{x\in G} \||\cdot|_p^{pr}\kappa_x \|_{L^2(G)}\right).
 \end{align*}  As $pr-Q/2>0$ and each term on the last right-hand side defines a continuous semi-norm on $S^0(G\times \Gh)$, Part (1) follows.
 Part (2) follows from the properties of composition in Theorem \ref{thm_sclexp_prod+adj}  (1) and the properties of positive Rockland operators.
 
 For Part (3), it suffices to prove the $L^2$-boundedness for $m=0$, and this follows from \eqref{eq_OpMsigmaL2bdd} as above.
\end{proof}

We denote by
$$
\|f\|_{L^2_{s,\eps}(G)} := 
\|f\|_{L^2_{s}(G),\eps^\nu \cR}
=
\|(\id+\eps^\nu \cR)^{\frac s\nu }f\|_{L^2(G)}
$$
the semiclassical Sobolev norm on $G$ associated with $\eps^\nu \cR$ where $\cR$ is a positive Rockland operator $\cR$ of homogeneous degree $\nu$. 
Any two positive Rockland operator will yield equivalent norms with induced constant uniform in $\eps\in (0,1]$.
Hence we can describe the behaviour of $T(\eps):=\Op^{(\eps)}_G (\sigma)$ in Theorem \ref{thm_sclcontSob}  (2) with 
$$
\sup_{\eps\in (0,1]} \|T(\eps)\|_{\sL(L^2_{s
,\eps}(G),L^2_{s-m,\eps}(G)) }<\infty.
$$ 
We will say that the family of operators $T(\eps)$, $\eps\in (0,1])$ is bounded 
	$L^2_s(G)\to L^2_{s-m}(G)$  semiclassically $\eps$-uniformly.
We have a similar vocabulary on $M$.

\subsection{Semiclassical asymptotics}
The asymptotics expansions in 
Theorem \ref{thm_sclexp_prod+adj} 
 lead us to define:

\begin{definition}
\label{def_sclexp}
Let $\sigma(\eps)$, $\eps\in (0,1]$, be a family of symbols in $S^m(G\times \Gh)$.
We say that $\sigma(\eps)$ admits a uniform semiclassical  expansion in $S^m(G\times \Gh)$ at scale $\eps\in (0,1]$ when the following properties are satisfied:
\begin{enumerate}
	\item For every $\eps\in (0,1]$, 
	$\sigma(\eps)$ admits an asymptotic expansion in $S^m(G\times \Gh)$  (in the sense of Definition  \ref{def_asymp}) of the form:
	$$
\sigma(\eps) \sim \sum_{j\in \bN_0} \eps^j \, \tau_j(\eps).
$$
\item The family of symbols $\sigma(\eps)\in S^m(G\times \Gh)$, $\eps\in (0,1]$,  is uniformly in $S^{m}(G\times \Gh)$, 
and for each $j\in \bN_0$, the family of symbols $\tau_j(\eps)\in S^{m-j}(G\times \Gh)$, $\eps\in (0,1]$, is uniformly in $S^{m-j}(G\times \Gh)$.
\item For each $N\in \bN$, the family of remainders
$\eps^{-(N+1)}(\sigma(\eps)-\sum_{j\leq N}  \eps^j \, \tau_j (\eps))$, $\eps\in (0,1]$, is uniformly in $S^{m-(N+1)}(G\times \Gh)$.
\end{enumerate}
We then write 
$$
\sigma(\eps) \sim_{\eps} \sum_{j\in \bN_0} \eps^j \, \tau_j(\eps) \quad\mbox{uniformly in} \ S^m(G\times \Gh).
$$
The symbols $\tau_0(\eps)$ and $\tau_1(\eps)$  are called the principal and  subprincipal symbols of $\sigma(\eps)$. 

We have a similar definition in $S^m(M\times \Gh)$.\end{definition}

Proceeding as for \cite[Theorem 5.5.1]{R+F_monograph},
given an asymptotic expansion $\sum_{j\in \bN_0} \eps^j \tau_j(\eps)$, with $\tau_j(\eps)$ $\eps$-uniformly in $S^{m-j}(G\times \Gh)$, $j\in \bN_0$, then there exists a symbol $\sigma(\eps)$ $\eps$-uniformly in $S^m(G\times \Gh)$ admitting 
$\sum_{j\in \bN_0} \eps^j \tau_j(\eps)$
as uniform semiclassical expansion. Moreover, for $\eps\in (0,1]$ fixed, the symbol $\sigma(\eps)$ is unique modulo $S^{-\infty}(G\times \Gh)$.

Theorem \ref{thm_sclexp_prod+adj} provides our first examples of semiclassical expansions:
\begin{ex}
\begin{enumerate}
	\item 
If $\sigma_1 \in S^{m_1}(G\times \Gh)$ 
and $\sigma_2 \in S^{m_2}(G\times \Gh)$, 
then the family of symbols $\sigma_1 \diamond_\eps \sigma_2$, $\eps\in (0,1]$ admits the expansion
$$
 \sigma_1 \diamond_\eps \sigma_2\sim_\eps \sum_{\alpha\in \bN_0^n } \eps^{[\alpha]} \Delta^\alpha \sigma_1 \, X^\alpha \sigma_2 \quad\mbox{uniformly in} \, S^{m_1+m_2}(G\times \Gh).
$$
\item If $\sigma\in S^m(G\times \Gh)$ then the family of symbols $\sigma^{(\eps,*)}$ admits the expansion
$$
\sigma^{(\eps,*)}\sim_\eps  \sum_{\alpha\in \bN_0^n } \eps^{[\alpha]}
 \Delta^\alpha  X^\alpha \sigma^*
\quad\mbox{uniformly in} \, S^{m}(G\times \Gh).
$$
\end{enumerate}
\end{ex}

Another example is the case of sub-Laplacians  in horizontal divergence form:
\begin{ex}
Considering the operator $\cL_A$ from Section \ref{subsec_cL_a}, 
	we may write 
	$$
	\eps^2 \cL_A=-\eps^2 \sum_{1\leq i,j\leq n_1} X_i (a_{i,j}(x) X_j) = \Op^{(\eps)}(\sigma(\eps))	, \qquad 
\sigma(\eps) = \tau_0 +\eps\tau_1,
$$
where the principal symbol is given by
$$
 \tau_0 (x,\pi)= \sum_{1\leq i,j\leq n_1} a_{i,j}(x)\, \pi( X_i)   \pi( X_j),
 $$
 and the subprincipal symbol is given by
 $$
 \tau_1 (x,\pi)= \sum_{j=1}^{n_1} (\sum_{i=1}^{n_1} X_i a_{i,j})(x) \, \pi( X_j).
 $$
 This implies readily that the family $\sigma(\eps)$,  $\eps\in (0,1]$,  admits a uniform semiclassical expansion $\sigma(\eps) \sim_\eps \tau_0 +\eps\tau_1$  in $S^2(G\times \Gh)$.
 
 We may generalise this example with functions  $a_{i,j}$ that may depend uniformly on $\eps$, as well as their  left-invariant derivatives:
 $$
 \sup_{\eps\in (0,1]} \|X^\alpha a_{i,j}\|_{L^\infty(G)}<\infty, 
 \qquad \alpha\in \bN_0^{n}, \ 1\leq i,j \leq n.
 $$
 We have similar properties for sub-Laplacians 
 and their generalisations given in Section \ref{subsec_generalisationLA} on $G$ and $M$.
\end{ex}

As a consequence of  Theorem \ref{thm_sclexp_prod+adj}, the  class of operators $\Op^\eps \sigma(\eps)$ with $\sigma(\eps)$ admitting a semiclassical expansion is stable under composition and adjoint:
\begin{theorem}
\label{thm_sclasymp}
\begin{enumerate}
\item For $i=1,2,$ let $\sigma_i(\eps)$, $\eps\in (0,1]$,  be a family of symbols admitting a 
 semiclassical  expansion 
 $$
\sigma_i(\eps) \sim_{\eps} \sum_{j\in \bN_0} \eps^j \, \tau_{i,j}(\eps) \quad\mbox{uniformly in}\ S^{m_i}(G\times \Gh).
$$
Then 
$\sigma_1(\eps) \diamond_\eps \sigma_2(\eps)$, $\eps\in (0,1]$,  is a family of symbols admitting 
the semiclassical expansion:
 $$
 \sigma_1(\eps) \diamond_\eps \sigma_2(\eps)
 \sim_\eps 
 \sum_{\substack{j_1,j_2\in \bN_0,\\\alpha\in \bN_0^n}} \eps^{j_1+j_2+ [\alpha] }
 \Delta^{\alpha}  \tau_{1,j_1}(\eps)\, X^\alpha \tau_{2,j_2}(\eps)
 \quad\mbox{uniformly in}\ S^{m_1+m_2}(G\times \Gh).
 $$
	\item Let $\sigma(\eps)$, $\eps\in (0,1]$,  be a family of symbols admitting a  semiclassical asymptotic expansion:
 $$
\sigma(\eps) \sim_{\eps} \sum_{j\in \bN_0} \eps^{j} \, \tau_j(\eps)\quad\mbox{uniformly in} \ S^m(G\times \Gh).
$$
Then 
$\sigma(\eps)^{(\eps,*)}$, $\eps\in (0,1]$, 
 is a family of symbols admitting a
 semiclassical asymptotic:
$$
\sigma(\eps)^{(\eps,*)} \sim_{\eps} \sum_{j\in \bN_0, \alpha\in \bN_0^n} \eps^{j+[\alpha]}
\Delta^\alpha X^\alpha \tau_j(\eps)^*\quad\mbox{uniformly in}\ S^{m}(G\times \Gh).
$$
\end{enumerate}
 We have similar properties for the symbol classes $S^m(M\times \Gh)$.
\end{theorem}

\begin{proof}
For Part (1), we write for $i=1,2$ and any $N_i\in \bN_0$
$$
\sigma_i(\eps) = \sum_{j_i\leq N_i} \eps^{j_i} \, \tau_{i,j_i}(\eps)
\ + \ \eps^{N_i+1} \rho_{i,N_i+1}(\eps). 
$$
By linearity of the quantization and composition of operators, 
the operation  $\diamond_\eps$ is bilinear and
we have
 \begin{align*}
  \sigma_1(\eps) \diamond_\eps \sigma_2(\eps)
 &=
 \sum_{j_1\leq N}\sum_{j_2\leq N} \eps^{j_1+j_2} \, \tau_{1,j_1}(\eps)\diamond_\eps \tau_{2,j_2}(\eps)
 +
  \sum_{j_1\leq N} \eps^{j_1+N+1} \, \tau_{1,j_1}(\eps)\diamond_\eps \rho_{2,N+1}(\eps)
\\ & \qquad +
 \sum_{j_2\leq N} \eps^{j_2+N+1} \,\rho_{1,N+1} (\eps) \diamond_\eps  \tau_{2,j_2}(\eps).
 \end{align*}
 We conclude with routines checks and  Theorem \ref{thm_sclexp_prod+adj} (1).
  
For Part (2), we write for any $N\in \bN_0$
$$
\sigma(\eps)= \sum_{j\leq N} \eps^{j} \, \tau_{j}(\eps)
\ + \ \eps^{N+1} \rho_{N+1}(\eps). 
$$
By linearity of taking the adjoint,  the operation $\tau \mapsto \tau^{(\eps,*)}$ is linear, so
we have
 \begin{align*}
  \sigma(\eps)^{(\eps,*)}
   &=
 \sum_{j\leq N} \eps^{j} \, \tau_{j}(\eps)^{(\eps,*)}
\ + \ \eps^{N+1} \rho_{N+1}(\eps)^{(\eps,*)}.
 \end{align*}
 We conclude with routines checks and Theorem \ref{thm_sclexp_prod+adj} (2).
\end{proof} 

\subsection{Semiclassical smoothing symbols and operators}

In this paper, we distinguish between semiclassical smoothing symbols and operators. 

\begin{definition}
\label{def_sclsmoothing}
	A family of symbols $\sigma =\sigma(\eps)$, $\eps\in (0,1]$ or  its corresponding familly of operators $\Op_G^{(\eps)}$ is \emph{semiclassically smoothing} when each $\sigma(\eps)$ is smoothing, i.e. $\sigma(\eps)\in S^{-\infty}(G\times \Gh)$, with 
	$$
	\sup_{\eps\in (0,1]}
	\|\sigma(\eps)\|_{S^m,a,b,c} <\infty,
	\quad \mbox{for any}\ \|\cdot \|_{S^m,a,b,c}.
	$$  
		We have a similar definition on $M$.	
\end{definition}

\begin{definition}
\label{def_sclsmoothingSob}
Let $R=R(\eps)$, $\eps\in (0,1]$, be a family of  operators  bounded $L^2(G)$.
It is  \emph{semiclassically smoothing on the Sobolev scale} when
$$
	\forall s\in \bR, \, N\in \bN_0, \quad \exists C>0\quad \forall \eps\in (0,1]\qquad 
	\|R\|_{\sL(L_{s,\eps}^2 (G), L^2_{s+N,\eps }(G))}
	\leq C \eps^{N+1}.
	$$
	
	We have a similar definition on $M$.
\end{definition}

In fact, the improvement is in any power of $\eps$ as is usually the case in semiclassical analysis:
\begin{lemma}
Let $R=R(\eps)$, $\eps\in (0,1]$, be a family of  operators bounded on $L^2(G)$ that is semiclassically smoothing on the Sobolev scale. 
	We have
	$$
	\forall s\in \bR, \, N_1,N_2\in \bN_0, \quad \exists C>0\quad \forall \eps\in (0,1]\qquad 
	\|R\|_{\sL(L_{s,\eps}^2 (G), L^2_{s+N_1,\eps }(G))}
	\leq C \eps^{N_2}.
	$$
	We have a similar result on $M$.	
\end{lemma}
\begin{proof}
	This follows readily from 
	$$
	s_1\leq s_2 \Longrightarrow \|g\|_{L^2_{s_1,\eps}}\leq \|g\|_{L^2_{s_2,\eps}},
	$$
	applied to $Rf$ with $s_1=s+N_1$ and $s_2=s+N_2$ with $N_2\geq N_1$. 
\end{proof}

The quantization of semiclassical smoothing symbols gives  semiclassical smoothing operators on the Sobolev scale, and the converse is true on $M$:
\begin{lemma}
	
\begin{enumerate}
	\item 
	If  $\sigma =\sigma(\eps)$, $\eps\in (0,1]$ is semiclassically smoothing, 
	the $R(\eps)= \Op^{(\eps)}_G (\sigma)$ 	is semiclassically smoothing on the Sobolev scale.

	\item We have a similar result on $M$, where moreover the converse is true: if  $R=R(\eps)$, $\eps\in (0,1]$, is a family of  operators bounded on $L^2(M)$, then it may be written in the form $R(\eps)= \Op^{(\eps)}_G (\sigma)$ 	with $\sigma =\sigma(\eps)$, $\eps\in (0,1]$ semiclassically smoothing.
\end{enumerate}	
\end{lemma}

\begin{proof}
Part (1) and the similar result in Part (2) follow from the property of the semiclassical calculus. 

In the case of $M$, the proof of Proposition \ref{prop_char_smoothingM} (2) implies that the integral kernel $K=K(\eps)\in \cD'(M\times M)$ of $R(\eps)$ is smooth with for any $\alpha,\beta\in \bN_0^n$
\begin{align*}
 \eps^{[\alpha]+[\beta]}\|X^\alpha_M (X^\beta_M)^t K\|_{L^2(M\times M)}& = \eps^{[\alpha]+[\beta]}\|X^\alpha_M R X^\beta_M\|_{HS(L^2(M))}\\
 &\leq C_s\|(\id+\eps^\nu \cR_M)^{\frac s \nu}X^\alpha_M R X^\beta_M\|_{\sL(L^2(M))} \lesssim_{\alpha,\beta,N} \eps^N,   
\end{align*}
	for any $N$ and $\eps\in (0,1]$.
This implies in turn that the convolution kernel $\kappa_{\dot x}(y)$ defined in the proof of Proposition \ref{prop_char_smoothingM} (1) is smooth in $(\dot x,y)$ with $y$-support included in a compact subset independent of $\dot x$ or $\eps$  and such that
$$
\| X^\alpha_M X^\beta  \kappa\|_{L^2(M\times G)} \lesssim_{\alpha,\beta} \eps^N.
$$
This implies readily that the corresponding symbols $\sigma=\sigma(\eps)$ given by 
$$
\sigma(\dot x,\pi)=(\eps^{-1} \cdot \pi)(\kappa_{\dot x})
= \pi(\kappa_{\dot x}^{(\eps^{-1})}), 
\qquad 
\kappa_{\dot x}^{(\eps^{-1})} (y) := \eps^{-Q} \kappa_{\dot x}(\eps^{-1} y),
$$
 is semiclassically smoothing.
 Since $R=\Op^{(\eps)}_M (\sigma)$, this concludes the proof. 
\end{proof}

\begin{remark}
It is likely that the smoothing pseudo-differential calculi $\Psi^\infty(G)$ 
and $\Psi^\infty(M)$ can be characterised with suitable commutators and actions on Sobolev spaces, but this would be outside of the scope of our paper.	
\end{remark}

\subsection{The class $\cA_0$ and its asymptotics}
In this section, we recall the definition of the class of symbols $\cA_0$ used in \cite{FFchina,FFJST,FFFMilan,fischer2022semiclassicalanalysiscompactnilmanifolds,LinoCLoMe} and some of its immediate properties. 

\subsubsection{Definition}
 
On $M$, the class $\cA_0 = \cA_0(M)$ coincides with the class of smoothing symbols:
$$
\cA_0(M):=S^{-\infty}(M\times \Gh),
$$ 
while on $G$, the class $\cA_0 = \cA_0(G)$ is defined as the space of  smoothing symbols with $x$-compact support:
$$
\cA_0(G):=\{\sigma\in S^{-\infty}(G\times \Gh) \ \mbox{with}\ x-\mbox{compact support}\}.
$$
In particular, the convolution kernels $\kappa$ of symbols $\sigma$ in $\cA_0(M)$ or $\cA_0(G)$ are Schwartz in the group variable. 
Furthermore, the group Fourier transform yields a bijection $(\dot x\mapsto \kappa_{\dot x}) \mapsto (\dot x \mapsto  \sigma(\dot x,\cdot) = \cF_G(\kappa_{\dot x}))$ from $C^\infty(M:\cS(G))$ onto $\cA_0(M)$
and a bijection $(x\mapsto \kappa_{x}) \mapsto ( x \mapsto  \sigma(x,\cdot) = \cF_G(\kappa_{x}))$ from $C_c^\infty(G:\cS(G))$ onto $\cA_0(G)$.  We equip the vector spaces $\cA_0(M)$ and $\cA_0(G)$ of the topologies so that these mappings are isomorphisms of topological vector spaces. 

\subsubsection{First properties}
We observe that $\cA_0(G)$ and $\cA_0(M)$ are  $*$-algebras  for the usual composition and adjoint of symbols.

Proceeding as in \cite{FFchina,FFJST,FFFMilan,LinoCLoMe,fischer2022semiclassicalanalysiscompactnilmanifolds}, 
we set 
$$
\|\sigma \|_{\cA_0} := \int_{G} \sup_{\dot x\in G}  |\kappa_{\dot x}(y)|dy, 
$$
where $\kappa_{x}$ is the kernel associated with $\sigma \in \cA_0(G)$.
This defines a continuous seminorm $\|\cdot\|_{\cA_0}$ on $\cA_0(G)$.
We have
$$
\forall \sigma\in \cA_0(G)\qquad 
\sup_{(x,\pi)\in G\times \Gh} \|\sigma(x,\pi)\|_{\sL(\cH_\pi)}
 \leq  \|\sigma \|_{\cA_0},
$$
Moreover,
$$
\forall \sigma\in \cA_0(G), \ \forall \eps\in(0,1]\qquad 
\| \Op^{(\eps)} (\sigma)\|_{\sL(L^2(M))} \leq  \|\sigma^{(\eps)}\|_{\cA_0}
 =\|\sigma\|_{\cA_0}.
$$
With a similar norm on $\cA_0(M)$, similar properties have been proved on $M$
\cite{FFFMilan,fischer2022semiclassicalanalysiscompactnilmanifolds}.

For any $\sigma,\sigma_1,\sigma_2\in \cA_0(G)$, 
we have \cite{FFJST,FFchina,LinoCLoMe}
$$
\Op^{(\eps)} (\sigma_1)\Op^{(\eps)} (\sigma_2)
=\Op^{(\eps)} (\sigma_1\sigma_2)+O(\eps)
\qquad\mbox{and}\qquad
(\Op^{(\eps)} (\sigma))^*=\Op^{(\eps)} (\sigma^*)
+O(\eps),
$$
 in the sense that 
\begin{align*}
	\|\Op^{(\eps)} (\sigma_1)\Op^{(\eps)} (\sigma_2)
-\Op^{(\eps)} (\sigma_1\sigma_2)\|_{\sL(L^2(G))}
&\lesssim_{\sigma_1,\sigma_2} \eps,\\
\|(\Op^{(\eps)} (\sigma))^*-\Op^{(\eps)} (\sigma^*)\|_{\sL(L^2(G))}
&\lesssim_{\sigma} \eps,
\end{align*}
with  similar properties on $\cA_0(M)$
\cite{FFFMilan,LinoCLoMe,fischer2022semiclassicalanalysiscompactnilmanifolds}.
Theorems \ref{thm_sclcontSob} and \ref{thm_sclexp_prod+adj} imply the complete semiclassical expansions in the $\sL(L^2(G))$ or $\sL(L^2(M))$ sense, that is, it holds for any $N\in \bN_0$
\begin{align*}
	\Op^{(\eps)} (\sigma_1)\Op^{(\eps)} (\sigma_2)
&=
\sum_{[\alpha]\leq N } \eps^{[\alpha]} \Delta^\alpha \sigma_1 \, X^\alpha \sigma_2 
+O(\eps)^{N+1}, \\ 
(\Op^{(\eps)} (\sigma))^*
&=\sum_{[\alpha]\leq N } \eps^{[\alpha]}
 \Delta^\alpha  X^\alpha \sigma^*
+O(\eps)^{N+1}.
\end{align*}

\subsubsection{Integral kernels, Hilbert-Schmidt norms and traces}
\label{subsubsec_L2MGh} 
 From the kernel estimates in Theorem \ref{thm_kernelG} (2), 
 if $\sigma\in S^{-\infty}(G\times \Gh)$  with associated kernel $\kappa_x(z)$, 
 then the integral kernel  $K^{(\eps)}$ of  $\Op_G^{(\eps)}(\sigma)$ is smooth on $G\times G$ and satisfies:
$$ 
\forall \eps>0,\ 
\forall x,y\in G\qquad
K^{(\eps)} (x,y) =\kappa^{(\eps)}_{x}(y^{-1}x), 
\quad\mbox{so}\quad  K^{(\eps)} (x,x) =\eps^{-Q} \kappa_{x}(0),
$$
with 
$$
\kappa_{x}(0) = \int_{\Gh} \tr\left( \sigma(x,\pi) \right)d\mu(\pi). 
$$
 Now, from the consequences of the kernel estimates for traces and Hilbert-Schmidt norms in Corollary \ref{cor_trHSG}, it follows that if in addition 
 $\sigma$ has $x$-compact support, that is, if 
 $\sigma\in \cA_0(G)$, then 
 the operator $\Op_G^{(\eps)}(\sigma)$ is trace-class and Hilbert-Schmidt on $L^2(G)$ with
$$
\tr \left(\Op^{(\eps)}(\sigma)\right)
=\int_G K^{(\eps)} (x,x) dx
= \eps^{-Q} \int_G\kappa_{ x}(0)  dx, 
$$
with 
$$
\int_G\kappa_{x}(0)  d x = \iint_{G\times \Gh} \tr\left( \sigma( x,\pi) \right) dx  d\mu(\pi),
$$
and 
\begin{equation}
\label{eq_HSG}
	\|\Op^{(\eps)}(\sigma)\|_{HS(L^2(M))}^2 
=\|K^{(\eps)}\|_{L^2(G\times G)}^2
= \eps^{-Q} \|\kappa\|_{L^2(G\times G)}^2.
\end{equation}

Defining the tensor product of the Hilbert spaces $L^2(G)$ and $L^2(\Gh)$ defined in 
 Section \ref{subsec_aboutGnilp}:
$$
L^2(G\times \Gh) := \overline{L^2(M) \otimes  L^2(\Gh)},
$$
we may identify $L^2(G\times \Gh)$ with the space of measurable fields of Hilbert-Schmidt operators $\sigma = \{\sigma(x,\pi)\ : \ ( x,\pi) \in G\times \Gh\}$ such that 
$$
\|\sigma\|_{L^2(G\times \Gh)}^2 :=\iint_{G\times \Gh} \|\sigma(x,\pi)\|_{HS(\cH_\pi)}^2 dx d\mu(\pi)<\infty.
$$
Here $\mu$ is the Plancherel measure on $\Gh$, see Section \ref{subsec_aboutGnilp}.
The group Fourier transform yields an isomorphism between the Hilbert spaces $L^2(G\times \Gh)$ and $L^2(G\times G)$ since $\cF_G^{-1}\sigma (x,\cdot)= \kappa_{x}$.
By the Plancherel formula, we obtain a more pleasing writing for the  right-hand side in \eqref{eq_HSG}:
$$
\|\sigma\|_{L^2(G\times \Gh)} = \|\kappa\|_{L^2(G\times G)}.
$$
Naturally $\cA_0(G)\subset L^2(G\times \Gh)$, 
and we have a similar definition and properties for $L^2(M\times \Gh)$.

By homogeneity of nilmanifolds,  the above properties also hold \cite{fischer2022semiclassicalanalysiscompactnilmanifolds} for  $\cA_0(M)$
  asymptotically:
\begin{proposition}
	\label{prop_tr}
Let $\sigma\in \cA_0(M)$ with associated kernel $\kappa_{\dot x}(z)$.

\begin{enumerate}
	\item The integral kernel  $K^{(\eps)}$ of  $\Op^{(\eps)}(\sigma)$ is smooth on $M\times M$ and satisfies for $\eps$ small: 
$$ 
\forall \dot x\in M\qquad
K^{(\eps)} (\dot x,\dot x) = \eps^{-Q} \kappa_{\dot x}(0)\ + \ O(\eps^N), 
$$
for any $N\in \bN$, with 
$$
\kappa_{\dot x}(0) = \int_{\Gh} \tr\left( \sigma(\dot x,\pi) \right)d\mu(\pi).
$$
\item The operator $\Op^{(\eps)}(\sigma)$ is   trace-class  on $L^2(M)$ with
$$
\tr \left(\Op_M^{(\eps)}(\sigma)\right)
= \eps^{-Q} \int_{M}\kappa_{\dot x}(0)  d\dot x
\ + \ O(\eps^N), 
$$
for any $N\in \bN$,
with 
$$
\int_{M}\kappa_{\dot x}(0)  d\dot x = \iint_{M\times \Gh} \tr\left( \sigma(\dot x,\pi) \right) d\dot x  d\mu(\pi).
$$
\item The operator $\Op^{(\eps)}(\sigma)$ is  Hilbert-Schmidt on $L^2(M)$ with
$$
\|\Op_M^{(\eps)}(\sigma)\|_{HS(L^2(M))}^2 
= \eps^{-Q} \|\sigma\|_{L^2(M\times \Gh)}^2
\ + \ O(\eps).
$$
\end{enumerate}
Above, the implicit constants are bounded, up to constants 
depending on $G,\Gamma$ and possibly $N$,  by some continuous semi-norms of $S^{-\infty}(M\times \Gh)$ in $\sigma$. 
\end{proposition}

\section{Semiclassical functional calculus on $G$ and $M$}
\label{sec_FCG}

In this section, we develop semiclassical functional calculi inside $\Psi^\infty(G)$ and $\Psi^\infty(M)$. The proofs will be mainly about estimates allowing for the routine arguments in symbolic pseudo-differential calculi. The precise setting and hypotheses are presented in the next section.

\subsection{The main result}

\subsubsection{The hypotheses}
\label{subsec_settingFC}

\begin{setting}
\label{set_FC}
We consider a semiclassical family of pseudo-differential operators $T(\eps)\in \Psi^m(G)$, $\eps\in (0,1]$, whose corresponding symbols 
 $\sigma$ admit a  semiclassical  expansion  at scale $\eps\in (0,1]$:
 $$
T(\eps):=\Op^{(\eps)}_G(\sigma), \qquad 
\sigma \sim_{\eps} \sum_{j=0}^\infty \eps^j \, \sigma_j \quad\mbox{uniformly in} \ S^m(G\times \Gh).
$$
Moreover,  the  operators $T(\eps)$ and the principal symbols $\sigma_0$ are  non-negative  (in the sense of
\eqref{eq_nonnegT} and Definition \ref{def_nonnegsymb} for each $\eps\in (0,1]$):
$$
\sigma_0\geq 0
\qquad\mbox{and}\qquad 
T(\eps)\geq 0.
$$

We have a similar setting on $M$.
\end{setting}

We make the following two further assumptions, firstly on the order being positive and secondly on the principal symbol $\sigma_0$ and its invertibility:
\begin{hypothesis}
\label{hyp_m>0}
$m>0$.
\end{hypothesis}

\begin{hypothesis}
\label{hyp_FC}
For any $\gamma\in \bR$, there exists $C_\gamma>0$ such that for all $\eps\in (0,1]$, for almost all $(x,\pi)\in G\times \Gh$, and any $v\in \cH_\pi^\infty$, we have
\begin{equation}
\label{eq_hyp_FCG1}
 \|\pi(\id+\cR)^{\frac {\gamma} \nu} (\id + \sigma_0(x,\pi)) v\|_{\cH_\pi} \geq C_\gamma \|\pi(\id+\cR)^{\frac {m+\gamma} \nu}  v\|_{\cH_\pi},
\end{equation}
We have a similar hypothesis on $M$.
\end{hypothesis}

The above hypotheses imply by Theorem \ref{thm_Lparexists} and Corollary \ref{cor_Lparexists} that $T(\eps)$ is essentially self-adjoint (among other properties), and it will therefore make sense to define its functional calculus:
\begin{lemma}
\label{LE_essentially_self-adjoint}
Under Setting \ref{set_FC} and Hypothesis \ref{hyp_FC} on $G$,
each operator  $T(\eps)$ admits a left-parametrix, satisfies sub-elliptic estimates and is hypoelliptic.  
Assuming in addition Hypothesis \ref{hyp_m>0}, 
 $T(\eps)$ is also essentially self-adjoint on 		$\cS(G)\subset L^2(G)$. 
 
  We have similar properties in the nilmanifold setting $M$, where furthermore, $T(\eps)$ has  compact resolvent, its spectrum is a discrete subset of $[0,\infty)$ and its eigenspaces are finite dimensional. 
\end{lemma}	 

\subsubsection{Main example}
Our main example where the above hypotheses are satisfied is the sub-Laplcian in horizontal divergence form perturbed with a potential:

\begin{ex}
\label{ex_cLA+V_FChyp}
	Let $\cL_A$ be a non-negative sub-Laplacian   in horizontal divergence form on a stratified group $G$ as in Section \ref{subsec_cL_a}.
	We assume that it
	satisfies the hypothesis of uniform ellipticity of Lemma \ref{lem_LApar}.
Let $V\in C_{l,b}^\infty(G)$ be non-negative.
By Lemma \ref{lem_LApar},  the family of differential operators given (using the notation of Section \ref{subsec_cL_a}) by 
$$
\eps^2 (\cL_A +V) = \Op_G^{(\eps)}(\sigma), 
\quad \sigma= \sigma_0 +\eps \sigma_1 +\eps^2 V,
\quad\eps\in (0,1],
$$
	falls under Setting \ref{set_FC}  on $G$ and  satisfies Hypothesis \ref{hyp_m>0} with $m=2$ and Hypothesis \ref{hyp_FC} for $\cR=\cL_I$.
	
We have a similar property for  a sub-Laplacian in horizontal divergence form on a stratified nilmanifold $M$ as in Section \ref{subsec_cL_a}
perturbed by a non-negative potential $V\in C^\infty(M)$. 
\end{ex}

We can generalise the example of the sub-Laplacian in horizontal divergence form  above in the following way:

\begin{ex}
\label{ex_cRA+V_FChyp}
Consider the operators 
$$
\cR_A = \Op_G (\sigma_0+\sigma_1)
\quad\mbox{or respectively}\ 
\Op_M (\sigma_0+\sigma_1),
$$
 from Examples 
\ref{ex_RAGstrat} and \ref{ex_RAGgrad} on $G$,
 or  \ref{ex_RAMstrat} and \ref{ex_RAMgrad} on $M$. 
  Let $V\in C_{l,b}^\infty(G)$ be non-negative on $G$, 
 or let $V\in C^\infty (M)$ be non-negative on $M$.
Then 
 $\eps^\nu (\cR_A+V) $ falls under Setting \ref{set_FC}  on $G$ and  satisfies Hypothesis \ref{hyp_m>0} with $m$ being the homogeneous degree of $\cR_A$,  and Hypothesis \ref{hyp_FC} for $\cR=\cR_I$.
 We can replace $V$ with other non-negative pseudo-differential terms of order $<\nu$.
\end{ex}

\subsubsection{Statement of the main result} 
We can now state the main result of this paper.
\begin{theorem}
\label{thm_sclFCG}
We consider  Setting \ref{set_FC} and Hypotheses \ref{hyp_m>0} and \ref{hyp_FC}
on $G$. 
For any $\psi\in \cG^{m'}(\bR)$ with $m'\in \bR$,
$\psi(T)$ decomposes as 
$$
\psi(T) = \Op^{(\eps)}_G(s_\psi) +  R,
$$
with $s_\psi=s_\psi(\eps)$, $\eps\in (0,1]$, uniformly in $S^{mm'}(G\times \Gh)$ and $R=R(\eps)$, $\eps\in (0,1]$, semiclassically smoothing on the Sobolev scale
(in the sense of Definition \ref{def_sclsmoothingSob}). 
Moreover, $s_\psi$ admits a semiclassical expansion 
$$
s_\psi \sim_\eps \sum_{k=0}^\infty \eps^k \tau_k
\quad\mbox{uniformly in} \  S^{mm'}(G\times \Gh),
$$
with principal symbol 
$$
\tau_0 = \psi(\sigma_0)\in S^{mm'}(G).
$$
 
 We have a similar property on $M$  where furthermore  $R=R(\eps)$ is semiclassically smoothing (in the sense of Definition \ref{def_sclsmoothing}).
\end{theorem}

The fact mentioned in the statement that the principal symbol $\sigma_0$ has a functional calculus in the symbol classes is a consequence of Theorem \ref{thm_psi(sigma0)}.
The proof of Theorem \ref{thm_sclFCG}  relies  on the construction of a parametrix for  $z-T$ (see Section \ref{subsec_parametrix} below), it will be given in Section \ref{subsec_pfthm_sclFCG}.

\subsection{Parametrix for $z-T(\eps)$}
\label{subsec_parametrix}
Our strategy to study the functional calculus of $T=T(\eps)$  relies on  explicit expressions for the right parametrix of $z-T$ given as follows:

\begin{lemma}
\label{lem_resolvent_of_operator}
We consider  Setting \ref{set_FC} and Hypotheses \ref{hyp_m>0} and \ref{hyp_FC}
on $G$.
\begin{enumerate}
	\item For any $z\in\bC\setminus\bR$, 
	we set 
	$$
	b_{0,z} := (z-\sigma_0)^{-1},
	$$
	and recursively for $k=1,2,\ldots$,
$$
	b_{k,z}:= (z-\sigma_0)^{-1} d_{k,z}, \qquad
	d_{k,z}:=
	\sum_{\substack{j+[\alpha] + l=k \\ l<k}}  \Delta^\alpha \sigma_{j} \ X^\alpha b_{l,z}.
$$
For each $k\in \bN_0$, this defines a symbol $b_{k,z}\in S^{-m-k}(G\times\Gh)$ and a symbol $d_{k,z}\in S^{-k}(G\times\Gh)$.
Moreover, 
 for any semi-norms 
 $\|\cdot\|_{S^{-m-k},a,b,c}$ and  $\|\cdot\|_{S^{-k},a,b,c}$, there exist constant a $C>0$ and powers $p\in \bN$ such that we have for all $z\in \bC\setminus\bR$
 $$
\| b_{k,z}\|_{S^{-m-k},a,b,c}
\leq C \left(1+ \frac {1+|z|} {|\IM\, z|}\right)^{p+1} ,
\qquad
\| d_{k,z}\|_{S^{-k},a,b,c}
\leq C \left(1+ \frac {1+|z|} {|\IM\, z|}\right)^{p+1} .
$$ 
\item We construct 
the symbol  $p_z$   with asymptotics
 $$
p_z\sim_{\eps} \sum_{j\in \bN_0} \eps^j \, b_{j,z} \quad\mbox{uniformly in}\ S^{-m}(G\times \Gh)
$$
following the ideas of Borel's extension lemma (see e.g. \cite[Section 5.5.1]{R+F_monograph}).
Then for any $N\in \bN_0$ and for any semi-norm $\|\cdot\|_{S^{m-(N+1)},a,b,c}$, there exist a constant $C>0$ and a power $p\in \bN$ such that
$$
\forall z\in \bC\setminus \bR, \ \forall \eps\in (0,1]\qquad 
\|p_z
 - \sum_{j=0}^N \eps^j \, b_{j,z}\|_{S^{m-(N+1)},a,b,c} \leq C \left(1+\frac {1+|z|} {|\IM\, z|}\right)^{p+1}
  \eps^{N+1}
$$
\item 
The operator 
$P_z
:=\Op_G^{(\eps)} (p_z
)$ is a left parametrix for $z-T$ in the sense that  
$$
(z- T)P_z
 -\id \in \Psi^{-\infty}(G),
$$
is smoothing. 
Moreover, writing $R_z
:= \Op_G^{(\eps)}(r_z
) := (z- T)P_z
$, 
for any $N\in \bN_0$ and for any semi-norm $\|\cdot\|_{S^{-(N+1)},a,b,c}$, there exist a constant $C>0$ and a power $p\in \bN$ such that
$$
\forall z\in \bC\setminus \bR, \ \forall  \eps\in (0,1]\qquad 
\eps^{-(N+1)}\|r_z
 \|_{S^{-(N+1)},a,b,c} \leq 
 C \left( 1+\frac {1+|z|} {|\IM\, z|}\right)^{p+1}.
$$
\end{enumerate}

A similar property holds true on $M$.
\end{lemma}

\begin{proof}
Let us prove Part (1) inductively using the property of composition of symbols (see Section \ref{subsec_propSm}).
The seminorm estimates for $b_{0,z}$
follows from the resolvent bounds in Proposition \ref{prop_sigma0Res} together with Remark \ref{remprop_sigma0Res} (1).
Let $k_0\in \bN$. The estimates for $b_{0,z}$
imply that if suffices to prove the seminorm estimates for $d_{k_0,z}$:
\begin{align*}
\| d_{k_0,z}\|_{S^{-k_0},a,b,c}	
&\leq 
\sum_{\substack{j+[\alpha] + l=k_0 \\ l<k_0}}  \|\Delta^\alpha \sigma_{j} \ X^\alpha b_{l,z}\|_{S^{-k_0},a,b,c}\\
&\lesssim
\sum_{\substack{j+[\alpha] + l=k_0 \\ l<k_0}}  \|\Delta^\alpha \sigma_{j}\|_{S^{m-j-[\alpha]},a_1,b_1,c_1} \| X^\alpha b_{l,z}\|_{S^{-k_0-m+j+[\alpha]},a_2,b_2,c_2}\\
&\lesssim
\sum_{l=0}^{k_0-1}   \| X^\alpha b_{l,z}\|_{S^{-m-l},a_2,b_2,c_2}.
\end{align*}
Hence, the seminorm estimates for $b_{l,z}$, $l<k$, will imply the estimates for $d_{k,z}$ 
and therefore for $b_{k,z}$.
This 
concludes the proof of Part (1).

The construction of the parametrix symbol is classical, 
and the semi-norm estimates in Part (2) then follow from the construction and the estimates in Part (1).
For Part (3), we fix  $N\in\bN_0$. 
We  write
$$
(z- T)P_z
=
S_{N,z}
 +\eps^{N+1}\Op_G^{(\eps)}(r_{0,N,z}),
 $$
 where
 $$
 S_{N,z}:= \Op^{(\eps)}_G \big(z-\sum_{j=0}^{N} \eps^j \sigma_j\big) \Op^{(\eps)}_G \big( \sum_{l=0}^N \eps^l b_{l,z}\big) \in \Psi^0(G).
 $$
The asymptotic expansions for the symbols of $P_z
$ and $T$ together with the properties of composition imply that $r_{0,N,z}\in S^{-(N+1)}(G\times \Gh)$.
Moreover,  for any semi-norm $\|\cdot\|_{S^{-N},a,b,c}$, there exist a constant $C>0$ and a power $p\in \bN$ such that
$$
\forall z\in \bC\setminus \bR, \ \forall  \eps\in (0,1]\qquad 
\|r_{0,N,z} \|_{S^{-(N+1)},a,b,c} \leq 
C \left(1+ \frac {1+|z|} {|\IM\, z|}\right)^{p+1}.
$$
We now analyse the symbol of $S_{N,z}= \Op^{(\eps)}_G(s_{N,z})$:
\begin{align*}
s_{N,z}&= (z-\sum_{j=0}^{N} \eps^j \sigma_j) \diamond_\eps ( \sum_{l=0}^N \eps^l b_{l,z})\\	
&= \big(z-\sigma_0)\diamond_\eps ( \sum_{j=0}^N \eps^j b_{j,z}) \ - \ (\sum_{j=1}^{N} \eps^{j} \sigma_{j}) \diamond_\eps ( \sum_{l=0}^N \eps^{l} b_{l,z})
\\	
&= \sum_{[\alpha]+l\leq N} \eps^{[\alpha]+l} \Delta^\alpha (z-\sigma_0) \  X^\alpha b_{l,z}
 -  \sum_{[\alpha]+j+l\leq N, j>0} \eps^{[\alpha]+j+l}  \Delta^\alpha \sigma_{j}  \  X^\alpha b_{l,z}
\ + \ \eps^{N+1} r_{1,N,z},
\end{align*}
by the composition properties of the calculus (see Theorem \ref{thm_sclasymp}), 
with $r_{1,N,z}\in S^{-(N+1)}(G\times \Gh)$ satisfying 
a similar estimates as $r_{0,N,z}$ above. 
In the first sum of the right-hand side above, we observe
$$
	\sum_{[\alpha]+l\leq N} \eps^{[\alpha]+l} \Delta^\alpha (z-\sigma_0) \  X^\alpha b_{l,z}
=
(z-\sigma_0)\sum_{k=0}^N \eps^k b_{k,z}
\ - \sum_{[\alpha]+l\leq N, [\alpha]>0} \eps^{[\alpha]+l} \Delta^\alpha \sigma_0 \  X^\alpha b_{l,z}.
$$
This yields
\begin{align*}
s_{N,z}&= 
(z-\sigma_0)\sum_{k=0}^N \eps^k b_{k,z}
 -  \sum_{\substack{[\alpha]+j+l\leq N\\  j=0 \Rightarrow [\alpha]>0 }} \eps^{[\alpha]+j+l}  \Delta^\alpha \sigma_{j}  \  X^\alpha b_{l,z}
\ + \ \eps^{N+1} r_{1,N,z}
\\
&= \sum_{k=0}^N \eps^k
\sum_{\substack{j+[\alpha] + l=k \\ l<k}}  \Delta^\alpha \sigma_{j} \ X^\alpha b_{l,z}
 -  \sum_{\substack{[\alpha]+j+l\leq N\\ j=0 \Rightarrow [\alpha]>0}}\eps^{[\alpha]+j+l}  \Delta^\alpha \sigma_{j}  \  X^\alpha b_{l,z}
\ + \ \eps^{N+1} r_{1,N,z}
\end{align*}
having used the definition of $b_{k,z}$ from Part (1).
We have obtained
$s_{N,z}=\id +\eps^{N+1} r_{1,N,z}$.
This concludes the proof of Part (3).
\end{proof}

It is possible to also follow
the classical construction for a left parametrix instead of right parametrix for $z-T$
with similar property as the right one constructed in Lemma \ref{lem_resolvent_of_operator}.
However, the construction for the left one involves more complicated expressions, and as we can avoid using it, we will not present it here. 

We observe that the existence of parametrices implies  norm bounds on the true resolvent of $T$:
\begin{lemma}
\label{LE:Resolvent_norm_bound}
We consider  Setting \ref{set_FC} and Hypotheses \ref{hyp_m>0} and \ref{hyp_FC} on $G$. Then $(z-T)^{-1}$ is bounded $L^2_{s-m}(G) \to L^2_{s}(G)$ 
 semiclassically $\eps$-uniformly for any $s\in\bR$ and $z\in \bC\setminus \bR$. More precisely, let us fix a positive Rockland operator $\cR$ of homogeneous degree $\nu$; for any $s\in \bR$,  there exist a constant $C>0$ and an integer $p\in\bN$ such that
\begin{align*}
\forall z\in \bC\setminus \bR, \quad
\sup_{\eps\in (0,1]}
	\big\|  (z-T)^{-1} \big \|_{L^2_{s,\eps }(G)L^2_{s+m,\eps }(G)} 
	\leq C  \left (1+\frac{1+|z|}{|\IM\, z|} \right)^{p+1} .
\end{align*}

A similar property holds true on $M$.
\end{lemma}

\begin{proof}
Let $s\in \bR$.
By Lemma \ref{lem_resolvent_of_operator},  $(z- T)^{-1} = P_z
 -(z- T)^{-1}R_{z}$, so 
\begin{align*}
&\|(\id +\eps^\nu \cR_G)^{\frac{s}{\nu}} (z-T)^{-1} (\id + \eps^\nu \cR_G)^{-\frac{s-m}{\nu}}	\|_{\sL(L^2(G))}
\\&\qquad \leq 
\|(\id +\eps^\nu \cR_G)^{\frac{s}{\nu}} P_z
 (\id + \eps^\nu \cR_G)^{-\frac{s-m}{\nu}}	\|_{\sL(L^2(G))}
+ 
\\&\qquad \qquad+\|(\id +\eps^\nu \cR_G)^{\frac{s}{\nu}} (z- T)^{-1} R_z
(\id + \eps^\nu \cR_G)^{-\frac{s-m}{\nu}}	\|_{\sL(L^2(G))}.
\end{align*}
By the properties of the semiclassical calculus, 
the first $\sL(L^2(G))$-norm on the right-hand side is bounded up to a constant by a $S^{-m}$-semi-norm of $p_z
$.
We now assume  $s\leq 0$. Since $T$ is essentially self-adjoint, the second $\sL(L^2(G))$-norm is 
$$
\leq \frac 1 {|\IM\, z|}\|R_z
 (\id +\eps^\nu \cR_G)^{-\frac{s-m}{\nu}} 	\|_{\sL(L^2(G))}
$$
Using again the properties of the semiclassical calculus, this last $\sL(L^2(G))$-norm  is bounded up to a constant by a $S^{(s-m)}$-semi-norm of $R_z
$. The estimates for the semi-norms in $p_z
$ and $R_z
$ imply the result for $s\leq 0$.
We conclude the proof with an argument of duality  and interpolation to obtain the case $s>0$.
\end{proof}

\subsection{Proof of the main result}
\label{subsec_pfthm_sclFCG}

In this section, we prove Theorem \ref{thm_sclFCG}.
The main step consists in adapting the arguments given in Section \ref{subsubsec_pfthm_psi(sigma0)}
 regarding the functional properties of $\sigma_0$; we summarise them in the following statements:
	
\begin{lemma}
\label{lem_taukpsi}
We continue with the setting of Theorem \ref{thm_sclFCG}.
We fix $m_1\leq 0$.
Let $a_z(x,\pi)$, $(x,\pi)\in G\times \Gh$, be a field of bounded operators forming a symbol $a_z$ in $S^{m_1}(G\times \Gh)$ depending in $z\in \bC\setminus \bR$ in such a way that  $\bar \partial_z a_z =0$.
We also assume that for any seminorm 
$\|\cdot \|_{S^{m_1},a,b,c}$ , there exist $C=C_{k,a,b,c}>0$ and $p\in \bN$ such that
 $$
\forall z\in \bC\setminus\bR\qquad 
\| a_z \|_{S^{m_1},a,b,c}
\leq C \left(1+ \frac {1+|z|} {|\IM\, z|}\right)^{p} .
$$ 
Let $\psi\in \cG^{m'}(\bR)$ with (fixed) $m'<-1$.
We consider an almost analytic extension $\tilde \psi$ of $\psi$ satisfying the following properties:
\begin{enumerate}

\item $\tilde \psi$ is supported in $ \{x+iy \in \bC, |y|\leq 10(1+|x|)\}$,
    \item $(-i\partial_y)^p \tilde \psi|_\bR=  \psi^{(p)}$ for any $p\in \bN_0$, 
    \item   for any $p\in \bN$,
$$
\sup_{z\in \bC\setminus \bR}
(1+|z|)^{p-m'}
	|\partial_y^p \tilde{\psi}(z) | <\infty,
$$
\item for any $N\in \bN_0$
 $$
\int_{\bC} \big| \bar\partial \tilde \psi( z) \big| \left(\frac{1+|z| }{|\IM \, z|}\right)^N L(dz) <\infty,
$$
\end{enumerate}
Then the formula
	$$
	\tau(a_z,\psi) :=  \frac{1}{\pi} \int_{\bC}   \bar \partial \tilde \psi(z) \ a_z \,  L(dz),
	$$
	defines a symbol $\tau(a_z, \psi)$ that is $\eps$-uniformly in $ S^{m_1}(G\times \Gh)$.
	It is independent of the choice of such analytic extensions $\tilde \psi$, which exist by Proposition 
	\ref{propApp_tildepsicGm'}.
Moreover, for  any seminorm $\|\cdot\|_{S^{m_1},a,b,c} $, there exist $C>0$ and 
	a seminorm $\|\cdot\|_{\cG^{m'},N}$ such that 
	$$
	\forall \psi\in \cG^{m'}(\bR), \ \forall t\in (0,1]\qquad 
	\|\tau (a_z, \psi(t\, \cdot ))\|_{S^{m_1},a,b,c} \leq C t^{-1} \|\psi\|_{\cG^{m'},N}.
	$$

  	Similar properties hold true on $M$.
	\end{lemma}

\begin{remark}
\label{rem_lem_taukpsi}
The meaning of the dependence in $z\in \bC\setminus \bR$ such that 
	$\bar \partial_z a_z =0$ is that for any $(x,\pi)\in G\times \Gh$ and $v_1,v_2\in \cH_\pi$, the map
	$z\mapsto (a_z(x,\pi)v_1,v_2)_{\cH_\pi}$ is smooth and holomorphic  on $\bC \setminus \bR$. Moreover, 
	we observe that the hypotheses on $a_z$ in Lemma \ref{lem_taukpsi} are satisfied by the following symbols constructed in Lemma \ref{lem_resolvent_of_operator}:
 $b_{k,z}$  with $m_1=-m-k$,
 $d_{k,z}$  with $m_1=-k$,
 $p_z$  with $m_1=-m$ and
 $p_z -\sum_{k=0}^N \eps^k b_{k,z}$ with $m_1=-m-N$.
 Indeed, 
 $b_{k,z}$  is a (non-commutative but with $(z,\bar z)$-constant coefficients) polynomial expression in $(z-\sigma_0)^{-1}$ and $X^\beta \Delta_\alpha\sigma_0$ for some finite indices $\alpha,\beta$; consequently, $\bar \partial_z b_{k_0,z} =0$.\end{remark}

\begin{proof}
We start with observing that given a function $F\in C^\infty(\bC)$ satisfying for any $N\in \bN_0$
 $$
\int_{\bC} \big| F( z) \big| \left(\frac{1+|z| }{|\IM \, z|}\right)^N L(dz) <\infty,
$$
the formula
	$$
	\frac{1}{\pi} \int_{\bC}  F(z) \ a_{z} \,  L(dz),
	$$
	defines a symbol $\eps$-uniformly in $ S^{m_1}(G\times \Gh)$ 
	with
$$
 \frac{1}{\pi} \int_{\bC}  | F(z)|  \ \| a_z\|_{S^{m_1},a,b,c} \,  L(dz)
\leq \frac{C}{\pi}
 \int_{\bC}  | F(z)|  
 \left(1+\frac {1+|z|}{|\IM \, z|}\right)^{p}
L(dz)<\infty. 
$$
This implies in particular that  $\tau_k(a_z,\psi)$ is a well defined  symbol $\eps$-uniformly in $ S^{m_1}(G\times \Gh)$. 
To show that it depends on $\psi$ and not on  the choice of its almost analytic extension $\tilde \psi$ with the stated properties, we adapt the argument of \cite[Section 2]{davies_bk}.
For $\delta\in(0,1/4]$ and $R\geq 1$, we set 
$$
\Omega_{\delta,R} := \{z\in \bC, |\RE\,  z|< R
\ \mbox{and}\ \delta< |\IM \, z| \leq 20 R \}.
$$
Let $\tilde \psi_1$ and $\tilde \psi_2$ be two almost analytic extensions of $\psi$ satisfying the properties of the statement. 
Since $\bar \partial_z a_z=0$,
 Stokes' (or Green's) theorem yields:
$$
\frac{1}{\pi} \int_{\Omega_{\delta,R}}   
\bar \partial (\tilde \psi_1-\tilde \psi_2)(z) \ a_z\,  L(dz)
=-\frac{i}{2\pi} \int_{\partial \Omega_{\delta,R}}   (\tilde \psi_1(z)-\tilde \psi_2(z)) \ a_z \, dz,
$$
and therefore
\begin{align*}
	&\big\|\frac{1}{\pi} \int_{\Omega_{\delta,R}}   
\bar \partial (\tilde \psi_1-\tilde \psi_2)(z) \ a_z\,  L(dz)\big\|_{S^{m_!},a,b,c} 
\leq \frac1{2\pi} \int_{\partial \Omega_{\delta,R}}   |\tilde \psi_1(z)-\tilde \psi_2(z)| \ \| a_z \|_{S^{m_1},a,b,c}  dz\\
&\quad \lesssim 
\int_{\partial \Omega_{\delta,R}}   |\tilde \psi_1(z)-\tilde \psi_2(z)| 
 \left(1+\frac {1+|z|}{|\IM \, z|}\right)^{p}
 dz \lesssim  I_{1,\pm} +I_{2,\pm},
 \end{align*}
 where
 \begin{align*}
 I_{1,\pm} &= \int_{x=-R}^R   |\tilde \psi_1-\tilde \psi_2|(x\pm i\delta) 
 \left(1+\frac {R}{\delta}\right)^{p}
 dx,	\\
I_{2,\pm} &= \int_{\pm y\in (\delta,20 R)}   |\tilde \psi_1-\tilde \psi_2|(\pm R +i y) 
 \left(1+\frac {R}{|y|}\right)^{p}
 dx	;
 \end{align*}
 the integrals on the remaining two edges $y=\pm R$, $-R\leq x \leq R$, vanish because of the support properties of $\tilde \psi_1$ and $\tilde \psi_2$.
  To estimate $I_{1,\pm}$ and $I_{2,\pm}$, we perform a Taylor expansion of each $\tilde \psi_j$, $j=1,2$, about  $y\sim 0$:
 $$
\psi_j(z) = \sum_{p_1\leq N} \frac{(iy)^{p_1}}{p_1 !} [\partial_y^{p_1}\psi_j](x) + 
r_{N,j} (z),\quad z=x+iy,
$$
with remainder
$$
r_{N,j} (z):=
\frac{(iy)^{N+1}}{N !} \int_0^1 (1-s)^{N} [\partial_y^{N+1}\psi_j](x+isy) \,ds.
$$
The hypotheses on $\tilde \psi_j$, $j=1,2$, imply the remainder estimates
$$
|r_{N,j} (z)|
\lesssim_{N,\tilde \psi} 
|y|^{N+1}
\int_0^1  (1+|x+isy|)^{-(N+1-m')}
 ds,
 $$
as well as  the Taylor series of $\tilde \psi_1(x+iy)$ and  $\tilde \psi_2(x+iy)$   about $y\sim 0$ being identical,  so  we have:
$$
\tilde \psi_1(z)-\tilde \psi_2(z) = r_{N,1}(z)-r_{N,2}(z)
\quad \mbox{for any} \ N\in \bN_0.
$$
We can now go back to estimating
\begin{align*}
I_{1,\pm} &\leq 
\left(1+\frac {R}{\delta}\right)^{p}
\int_{x=-R}^R 
 (|r_{N,1}|+|r_{N,2}|)(x\pm i\delta) dx\lesssim \delta^{N+1}\left(1+\frac {R}{\delta}\right)^{p} R ,\\
 I_{2,\pm} &\leq  \int_{\pm y\in (\delta,20 R)}   (|r_{N,1}|+|r_{N,2}|)(\pm R +i y) 
 \left(1+\frac {R}{|y|}\right)^{p}
 dy
 \\
 &\lesssim  R^{-(N+1-m')}
  \int_{y=\delta}^{20 R} \left(y^{N+1} + R^p y^{N+1-p}\right) dy
 \lesssim  R^{m'+1} ,
  \end{align*}
  when $N\geq p$.
Choosing $R=\delta^{-1}$ and $N=p+1$ yields
$$
\lim_{\delta\to 0} I_{1,\pm}=\lim_{\delta\to 0} I_{2,\pm} =0,
\quad\mbox{so}\
\lim_{\delta\to 0}
\big\|\frac{1}{\pi} \int_{\Omega_{\delta,R}}   
\bar \partial (\tilde \psi_1-\tilde \psi_2)(z) \ a_z\,  L(dz)\big\|_{S^{m_1},a,b,c} =0,
$$
since $m'<-1$.
By Lebesgue's dominated convergence theorem, this implies: 
$$
\big\|\frac{1}{\pi} \int_{\bC}   
\bar \partial (\tilde \psi_1-\tilde \psi_2)(z) \ a_z\,  L(dz)\tau (a_z,\psi)\big\|_{S^{m_1},a,b,c} =0,
$$ 
for any $a,b,c$. This shows that $\tau(a_z, \psi)$  is a well defined symbol in $S^{m_1}(G\times \Gh)$ independently of the choice of  analytic extensions $\tilde \psi$ satisfying the stated properties.  
We may as well choose the almost analytic extension $\tilde \psi$ constructed in Proposition 
	\ref{propApp_tildepsicGm'}. 

We observe that if $\tilde \psi(z)$ is an extension constructed in Proposition 
	\ref{propApp_tildepsicGm'}, then $\tilde \psi(t z)$
is an almost analytic extension for $\psi(t\lambda)$ satisfying the  properties required in the statement, so we have
$$
\tau (a_z,\psi(t\, \cdot ))
=\frac{1}{\pi} \int_{\bC}  t (\bar \partial \tilde \psi)(tz) \ a_z \,  L(dz)
=  \frac{t^{-1}}{\pi} \int_{\bC}   \bar \partial \tilde \psi(z) \ a_{t^{-1}z} \,  L(dz) = t^{-1}\tau (a_{t^{-1}z}, \psi).
$$
Hence, proceeding as above, we obtain
\begin{align*}
&\|\tau(a_z,\psi(t\, \cdot ))\|_{S^{m_1},a,b,c}
\leq 
 \frac{t^{-1}}{\pi} \int_{\bC}  | \bar \partial \tilde \psi(z)|  \ \| a_{t^{-1}z}\|_{S^{m_1},a,b,c} \,  L(dz)
 \\&\qquad\lesssim t^{-1}
 \int_{\bC}  | \bar \partial \tilde \psi(z)|  
 \left(1+\frac {1+|t^{-1} z|}{|\IM \, t^{-1} z|}\right)^{p}L(dz)
 \leq  t^{-1} 
 \int_{\bC}  | \bar \partial \tilde \psi(z)|  
 \left(1+\frac {1+| z|}{|\IM \,  z|}\right)^{p}L(dz)
  \\&\qquad\lesssim t^{-1} \|\psi\|_{\cG^{m'}, N}.
\end{align*}
This concludes the proof. 
\end{proof}

\begin{corollary}
	\label{cor_lem_taukpsi}
	We continue with the setting of Theorem \ref{thm_sclFCG} and Lemma \ref{lem_taukpsi} with $m_1\leq -m$.
	Moreover we  assume that for any seminorm 
$\|\cdot \|_{S^{m_1+m},a,b,c}$ , there exist $C=C_{m_1,a,b,c}>0$ and $p\in \bN$ such that
 $$
\forall z\in \bC\setminus\bR\qquad 
\| z a_z \|_{S^{m_1+m},a,b,c}
\leq C \left(1+ \frac {1+|z|} {|\IM\, z|}\right)^{p} .
$$ 
\begin{enumerate}
	\item  For any $m_2 \in [-1,0]$ and any seminorm $\|\cdot\|_{S^{m_1+m_2 m} ,a,b,c}$, 
there exist a constant $C>0$ and a number $d_0\in \bN_0$ such that we have for any $\psi\in C_c^\infty(\frac 12,2)$ and $t\in (0,1]$
$$
	\|\tau(a_z,\psi(t \, \cdot))\|_{S^{m_1+m_2 m },a,b,c} \leq C t^{m_2} \max_{d=0,\ldots,d_0} 
	\sup_{\lambda \geq 0} |\psi^{(d)}(\lambda)|, 
$$
\item For any $m'_2 \geq -m $ and any seminorm $\|\cdot\|_{S^{m_1 + m_2'} ,a,b,c}$, 
there exist a constant $C>0$ and a number $d_0\in \bN_0$ such that 
$$
\forall \psi\in C_c^\infty(-1 ,1)\qquad
	\|\tau(a_z,\psi)\|_{S^{m_1 + m'_2},a,b,c} \leq C  \max_{d=0,\ldots,d_0} 
	\sup_{\lambda \geq 0} |\psi^{(d)}(\lambda)|.
$$
\end{enumerate}
\end{corollary}

\begin{remark}
\label{rem_cor_lem_taukpsi}
	We observe that the hypotheses on $a_z$ in Corollary \ref{cor_lem_taukpsi} are satisfied by 
	$b_{k,z}$. 
	Indeed, we have	
	$$
	z b_{k,z} = z(z-\sigma_0)^{-1} d_{k,z}
	= \left(1+\sigma_0(z-\sigma_0)^{-1} \right)d_{k,z},
	$$
	and 
	$b_{k,z}$ and $d_{k,z}$ satisfy the hypotheses on $a_z$ in Lemma \ref{lem_taukpsi} by Remark \ref{rem_lem_taukpsi}.
	Therefore, the hypotheses on $a_z$ in Corollary \ref{cor_lem_taukpsi} are also satisfied by 
	$p_z$ and
 $p_z -\sum_{k=0}^N \eps^k b_{k,z}$. 
\end{remark}

\begin{proof}[Proof of Corollary \ref{cor_lem_taukpsi}]
Part (2)  follows directly from Lemma \ref{lem_taukpsi} with the continuity of the inclusion of symbol classes.
Let us prove Part (1).
Let $\psi\in C_c^\infty(\frac 12,2)$.
Lemma \ref{lem_taukpsi} gives the case of $m_2=-1$.
For $m_2=0$, consider $\psi_1(\lambda):=\lambda^{-1}\psi(\lambda)$.
Let  $\tilde \psi_1$ be the almost analytic extension for $\psi_1$ constructed in 
Proposition \ref{propApp_tildepsicGm'}. 
We observe that $t z \tilde \psi_1 (tz)$ is an 
almost analytical extension of $\psi(t\lambda)$. 
We have
$$
\tau  (a_z,\psi(t \, \cdot)) = t \tau (z a_z,\psi_1 (t \, \cdot)).
$$
The  result follows from Lemma \ref{lem_taukpsi} for $m_2=0$, and for any $m_2\in (-1,0)$ by interpolation. 
\end{proof}

We can now show Theorem \ref{thm_sclFCG}. 

\begin{proof}[Proof of Theorem \ref{thm_sclFCG}]
By the properties of the semiclassical calculus,
 it suffices to show the case of $m'\in [-1/2,0)$.
	Let $\psi\in \cG^{m'}(\bR^n)$ with fixed $m'\in [-1/2,0)$. 
	Without loss of generality, 
	we may assume that $\psi$ is real-valued.
	Let $(\eta_j)$ be the dyadic decomposition considered in Section \ref{subsubsec_pfthm_psi(sigma0)}.
	We may write for any $\lambda\geq 0$
	$$
	\psi(\lambda) = \psi_{-1}(\lambda) + \sum_{j=0}^\infty 2^{jm'} \psi_j (2^{-j}\lambda),
	$$
	where
	$$
	 \psi_j (\mu):= 2^{-jm'} \psi(2^j \mu) \eta_0(\mu), \ j\geq 0 , \quad\mbox{and}\quad 
	\psi_{-1} (\mu):=  \psi(\mu) \eta_{-1}(\mu).
	$$
	We observe 
	$$
	\psi_{-1}\in C_c^\infty (-1,1) 
	\quad \mbox{while} \quad 
	\psi_j \in C_c^\infty (\frac 12,2), \ j\in \bN_0, 
	\qquad \mbox{and} \qquad
	\sup_{\lambda\geq 0} |\psi_j ^{(d)}(\lambda)| \lesssim_d \|\psi\|_{\cG^{m'},d}
		$$
		for any $d \in \bN_0$
with an implicit constant independent of $j=-1,0,1,2,\ldots$.

Considering the setting of Lemma \ref{lem_taukpsi}, we can proceed as in the proof of 
Theorem \ref{thm_psi(sigma0)} (2)  
replacing the role of the resolvent of $\sigma_0$ with $a_z$.
An application of the Cotlar-Stein Lemma together with Corollary \ref{cor_lem_taukpsi} (1) 
implies
$$
\sum_{j=0}^\infty 2^{jm'} \tau (a_z,\psi_j (2^{-j}\, \cdot)) \in S^{mm'+m_1}(G\times \Gh), 
$$
with seminorm estimates; note that this requires $m'\in [-1/2,0)$.
Using Corollary \ref{cor_lem_taukpsi} (2) for the first term corresponding to $j=-1$, 
we may define 
$$
\tilde \tau (a_z, \psi):= \tau(a_z,\psi_{-1} ) + \sum_{j=0}^\infty 2^{jm'} \tau(b_{k,z},\psi_j (2^{-j}\, \cdot))
\in S^{mm'+m_1}(G\times \Gh).
$$
Given a seminorm $\|\cdot \|_{S^{mm'+m_1}(G\times \Gh),a,b,c}$ there exist a constant $C'$ and a seminorm $\|\cdot\|_{\cG^{m'}, N}$ both independent of $\psi$ such that 
$$
\|\tilde \tau(a_z,\psi)\|_{S^{mm'+m_1}(G\times \Gh),a,b,c} \leq C \|\psi\|_{\cG^{m'}, N}.
$$
In fact, the constant $C'$ is up to a constant given by the constant $C_{k,a',b',c'}$ from Lemma \ref{lem_taukpsi} with $a',b',c'$ high enough. 
We can apply this to $a_z$ being $b_{k,z}$, 
$p_z$ and
 $p_z -\sum_{k=0}^N \eps^k b_{k,z}$, see 
 Remarks \ref{rem_lem_taukpsi} and 
\ref{rem_cor_lem_taukpsi}.
Consequently, the estimate in Lemma \ref{lem_resolvent_of_operator} implies  for the latter 
$$
\|\tilde \tau (p_z,\psi ) - \sum_{k=0}^N \eps^k \tilde\tau (b_{k,z},\psi)\|_{S^{mm'-N},a,b,c}
=
\|\tilde \tau (p_z- \sum_{k=0}^N \eps^kb_{k,z} ,\psi )\|_{S^{mm'-N},a,b,c}
\lesssim  \eps^{N+1} \|\psi\|_{\cG^{m'}, N}.
$$
In other words, the symbol
$$
s_\psi :=\tilde \tau(p_z,\psi ) \in S^{mm'}(G\times \Gh),
$$
admits the following semiclassical asymptotics expansion
$$
s_\psi \sim_\eps \sum_{k=0}^\infty \eps^k \tau_k, 
\qquad\mbox{where}\quad \tau_k:=\tilde\tau (b_{k,z},\psi), \ k=0,1,2,\ldots
$$

It remains to show that the family of operators
$$
R:=R(\eps):=\psi(T) -
\Op^{(\eps)}_G(s_\psi) , \quad \eps\in (0,1],
$$
is semiclassically smoothing on the Sobolev scale.
Above, $\psi(T)$ is defined by functional analysis and 
 we have with a sum in the sense of the strong operator topology of $L^2(G)$:
$$
	\psi(T) = \psi_{-1}(T) + \sum_{j=0}^\infty 2^{jm'} \psi_j (2^{-j}T).
	$$
	
We now use the right parametrix $P_z
$ constructed in Lemma \ref{lem_resolvent_of_operator}, to write 
$$
(z- T)P_z
 =\id+R_z,
 \quad\mbox{so}\quad 
(z- T)^{-1} = P_z
 - (z- T)^{-1}R_z.
 $$
 With the Helffer--Sj\"ostrand formula (see \eqref{eq_HS}), this leads to decomposing $\psi_0(T)$ as
$$
\psi_0(T) 
= \frac{1}{\pi} \int_{\bC}   \bar \partial \tilde \psi_0(z)\ (z-T)^{-1} L(dz)
= \frac{1}{\pi} \int_{\bC}   \bar \partial \tilde \psi_0(z)\ P_z \, L(dz)
+R_{0,\psi},
$$
where 
$$
R_{0,\psi}:=\frac{1}{\pi} \int_{\bC}   \bar \partial \tilde \psi_0(z)\ (z-T)^{-1} R_z\, L(dz).
$$
We recognise
$$
\frac{1}{\pi} \int_{\bC}   \bar \partial \tilde \psi_0(z)\ P_z \, L(dz)
=\Op_G^{(\eps)} (\tau(p_z,\psi_0)).
$$
More generally, we obtain for any $j\in \bN_0$ and for $j=-1$:
\begin{align*}
R_{j,\psi}&:=
\psi_j (2^{-j}T) -\Op_G^{(\eps)} (\tau(p_z,\psi_j(2^{-j}\cdot)))=
\frac{1}{\pi} \int_{\bC}  2^{-j} \bar \partial \tilde \psi_j(2^{-j}z)\ (z-T)^{-1} R_z\, L(dz),\\	
R_{-1,\psi}&:=\psi_{-1} (T) - \Op_G^{(\eps)} (\tau(p_z,\psi_{-1})) =\frac{1}{\pi} \int_{\bC}   \bar \partial \tilde \psi_{-1}(z)\ (z-T)^{-1} R_z\, L(dz).	
\end{align*}
With a sum in the sense of  the strong operator topology of $L^2(G)$, we can now write
$$
R= R_{-1,\psi} + \sum_{j=0}^\infty 2^{jm'} R_{j,\psi}.
$$

Let us analyse $R_{0,\psi}$.
The properties of the semiclassical calculus together with 
  Lemmata \ref{lem_resolvent_of_operator} (3) and \ref{LE:Resolvent_norm_bound} imply
  $$
\|(z-T)^{-1} R_z
(\eps)  \|_{\sL(L^2_{s,\eps }(G),L^2_{s+N+m,\eps }(G))} \lesssim_{N,m} \left (1+\frac{1+|z|}{|\IM\, z|} \right)^{p} \eps^{N+1},
$$
for some $p\in \bN$.
This yields
$$
\|R_{0,\psi}  \|_{\sL(L^2_{s,\eps }(G),L^2_{s+N+m,\eps }(G))} 
\lesssim_{N,m}
 \eps^{N+1} \|\psi\|_{\cG^{m'}, N_1},
$$
for some $N_1\in \bN$.
We have similar results for the $R_{j,\psi}$ for $j=-1$ and $j=1,2,\ldots$, with the implicit constants in the estimates  independent of $j$.
Therefore, we have obtained:
$$
\|R\|_{\sL(L^2_{s,\eps }(G),L^2_{s+N+m,\eps }(G))} 
\lesssim_{N,m} \eps^{N+1} \|\psi\|_{\cG^{m'}, N_1} 
(1+  \sum_{j=0}^\infty 2^{jm'} )
\lesssim \eps^{N+1} \|\psi\|_{\cG^{m'}, N_1}.
$$
This concludes the proof.
\end{proof}

In the Abelian case, it is possible to give a very neat formula for the $\tau_k$ in terms of derivatives of $\psi$. 
Although this is not possible in general in the non-commutative case, our proof still shows that the dependence of the $\tau_k$ is linear and continuous in $\psi$:
\begin{corollary}
We continue with the setting of Theorem \ref {thm_sclFCG}. 
Our construction is such that $s_\psi$, $\tau_k$ and $R$ depend linearly on $\psi$. Moreover, they satisfy the following estimates uniformly in $\eps\in (0,1]$
$$
	\|s_\psi\|_{S^{mm'},a,b,c} 
	\leq C \|\psi\|_{\cG^{m'},N}, \qquad 
	\|\tau_k\|_{S^{mm'-k},a,b,c} 
	\leq C \|\psi\|_{\cG^{m'},N},
$$
and on $G$
$$
\|R\|_{\sL(L^2_{s,\eps }(G),L^2_{s+N_1,\eps }(G))}  
	\leq C \eps^{N'+1}\|\psi\|_{\cG^{m'},N},
	$$
while on $M$, $R=\Op_M^{(\eps)}(r)$ with 
$$
\|r\|_{S^{-N_1},a,b,c} 
	\leq C\eps^{N'+1} \|\psi\|_{\cG^{m'},N}.
	$$
	Above,  $C,N$ depend on $N',N_1\in \bN_0$ and $a,b,c$ but not on $\psi,\eps $.	
\end{corollary}

\section{Weyl laws}
\label{sec_WeylR}

In this section, we apply  the functional calculus developed above to prove Weyl laws of certain semiclassical operators.

\subsection{Results}

Considering setting of the functional calculus above, 
we will prove  Weyl laws under the following hypotheses:
\begin{hypothesis}
\label{hyp_weylL}
We assume  that for two numbers $a<b$ there exists a $\delta_0>0$ such that on $G$
$$
C_{\sigma_0,a,b,\delta_0}:=
 \int_{G\times \Gh} \tr |1_{(a-\delta_0,b+\delta_0)}(\sigma_0(x,\pi))| d\mu(\pi)dx < \infty.
$$
We have a similar hypothesis on $M$ , with an integration over $M\times \Gh$. 
\end{hypothesis}

\begin{theorem}
\label{Thm:semiclasical_weyl_law}
We consider  Setting \ref{set_FC} and Hypotheses \ref{hyp_m>0}, \ref{hyp_FC} and \ref{hyp_weylL}. 
Then all the point of $[a,b]$ in the spectrum of $T$ are point spectrum.
Denote by $\lambda_{j,\eps}$, $j\in \bN$, the eigenvalues of $T$ contained in the interval $[a,b]$  (counted with multiplicities and in increasing order)
 and  the corresponding spectral counting function  for a fixed  $[a,b]\subset \bR$ by 
$$
N(\eps):=|\{ \lambda_{j,\eps}\in [a,b]\}|.
$$
 We have
$$
\lim_{\eps \to 0}\eps^{Q} N(\eps) = \int_{G\times \Gh} \tr \left(1_{[a,b]}( \sigma_0 (x,\pi) ) \right)dx d\mu(\pi).
$$

We have a similar result on $M$ with 
$$
\lim_{\eps \to 0}\eps^{Q} N(\eps) =  \int_{M\times \Gh} \tr \left(1_{[a,b]}( \sigma_0 (x,\pi) ) \right)dx d\mu(\pi).
$$
\end{theorem}

Before discussing the proof of Theorem \ref{Thm:semiclasical_weyl_law}, let us discuss its main application to a sub-Laplacian in divergence form perturbed by a potential:
\begin{corollary}
\label{cor1_Thm:semiclasical_weyl_law}
	Let $\cL_A$ be a non-negative sub-Laplacian   in horizontal divergence form on a stratified group $G$ as in Section \ref{subsec_cL_a}. We assume that it satisfies the hypothesis of uniform ellipticity of Lemma \ref{lem_LApar}.
Let $V\in C^\infty(G)$ be a non-negative function 
such that all its left-invariant derivatives are bounded. Moreover, suppose that there exists $a<b$ and $\delta>0$ such that $V^{-1}((a-\delta,b+\delta))$ is compact.
Then the operators $$
T (\eps):= \eps^2 \cL_A + V, \quad\eps\in (0,1],
$$
satisfy the Weyl law
$$
\lim_{\eps \to 0}\eps^{Q} \tr[1_{[a,b]}(T(\eps))] = \int_{G\times \Gh} \tr \left(1_{[a,b]}( \tau_0 (x,\pi) ) \right)dx d\mu(\pi),
$$
where $\tau_0=\sigma_0 +V$, that is,
$$
\tau_0 (x,\pi)= \sum_{1\leq i,j\leq n_1} a_{i,j}(x)\, \pi( X_i)   \pi( X_j) + V(x).
$$
We have a similar result on the nilmanifold $M$. 
\end{corollary}
\begin{proof}
With the notation of Section \ref{subsec_cL_a}, we have 
$$
T(\eps)=\Op_G^{(\eps)}(\tau(\eps)), 
\quad\mbox{where}\quad  
\tau(\eps) = \tau_0 +\eps\tau_1, 
$$
with principal and subprincipal symbols
$$
\tau_0 (x,\pi)=\sigma_0(x,\pi)
 + V(x) \qquad\text{and}\qquad
 \tau_1 (x,\pi)= \sigma_1 (x,\pi) .
 $$

As $V\geq 0$ and  $V^{-1}((a-\delta,b+\delta))$ is compact, $V^{-1}([0,b+\delta))$ is compact and contained $V^{-1}((a-\delta,b+\delta))$.
The uniform ellipticity implies:
$$
0\leq \sigma_0(x,\pi)\leq C \pi(\cL_\id),
$$
so that  
$$
  0\leq 1_{(a-\delta_0,b+\delta_0)}(\tau_0(x,\pi))  \leq 1_{[0,b+\delta)}(V(x)) \ |\psi|^2(C \pi( \cL_\id) ),
$$
where $\psi\in C_c^\infty(\bR)$ is valued in $[0,1]$ and satisfies $\psi=1$ on $(-\infty,b+\delta)$.
This implies that the quantity
\begin{align*}
	\int_{G\times \Gh} \tr |1_{(a-\delta_0,b+\delta_0)}(\tau_0(x,\pi))| d\mu(\pi)dx
& \leq 
\int_G 1_{[0,b+\delta)}(V(x)) dx 
\int_{\Gh} \tr  |\psi|^2(\widehat \cL_\id) d\mu
\\
& \quad = |V^{-1}([0,b+\delta))| \|\psi(\cL_\id)\delta_0\|_{L^2(G))}^2 <\infty,
\end{align*}
 is finite. In other words, Hypothesis \ref{hyp_weylL} is satisfied. 
 We conclude with an application of
Theorem~\ref{Thm:semiclasical_weyl_law}.
\end{proof}

Let us contrast Corollary \ref{cor1_Thm:semiclasical_weyl_law} with another application of Theorem~\ref{Thm:semiclasical_weyl_law}  valid  only on $M$:
\begin{corollary}
\label{cor2_Thm:semiclasical_weyl_law}
Let $\cL_A$ be a non-negative sub-Laplacian   in horizontal divergence form on a nilmanifold $M$ as in Section \ref{subsec_cL_a}. We assume that it satisfies the hypothesis of uniform ellipticity of Corollary \ref{cor_lem_LApar}.
Let $V\in C^\infty(M)$ be a non-negative function.
Then the operator $\cL_A + V$ is essentially self-adjoint on $C^\infty(M)\subset L^2(M)$,  has purely discrete spectrum with Weyl laws:
	$$
\lim_{\eps \to 0}\lambda^{-Q/2} \tr[1_{[0,\lambda]}(\cL_A + V)] = \int_{M\times \Gh} \tr \left(1_{[0,1]}( \sigma_0 (x,\pi) ) \right)dx d\mu(\pi).
$$
\end{corollary}

\begin{proof}
We observe 
$$
\tr[1_{[0,\lambda]}(\cL_A + V)] = \tr[1_{[0,1]}(T(\eps))] , 
\quad\mbox{where} \ \eps =\lambda^{-1/2},
$$
and
$$ 
 T(\eps) = \eps^2(\cL_A + V) = \Op^{(\eps)}_M(\tau),
$$ 	
with 
$$
\tau  = \tau_0 +\eps\tau_1 + \eps^2\tau_2, 
\qquad \tau_0 = \sigma_0, \ \tau_1=\sigma_1, \ \tau_2=V.
$$
As $\sigma_0 \leq C \widehat \cL_\id$, we have
$$
  0\leq 1_{[0,\lambda]}(\tau_0(\dot x,\pi))  \leq |\psi|^2(C \pi( \eps^2 \cL_\id) ),
$$
where $\psi\in C_c^\infty(\bR)$ is valued in $[0,1]$ and satisfies $\psi=1$ on $(-\infty,b+\delta)$.
This readily implies Hypothesis \ref{hyp_weylL} and we conclude with an application of
Theorem~\ref{Thm:semiclasical_weyl_law}.
\end{proof}

We can obtain similar results to Corollaries \ref{cor1_Thm:semiclasical_weyl_law}
and \ref{cor2_Thm:semiclasical_weyl_law} by applying Theorem~\ref{Thm:semiclasical_weyl_law} to the operators described  in Example \ref{ex_cRA+V_FChyp}.

\subsection{Proof of Theorem \ref{Thm:semiclasical_weyl_law}}

We will start the proof of Theorem \ref{Thm:semiclasical_weyl_law} with properties that are  of interest of their own. 

\begin{proposition}
	\label{prop_Thm:semiclasical_weyl_law}
	We consider  Setting \ref{set_FC} and Hypotheses \ref{hyp_m>0}, \ref{hyp_FC} and \ref{hyp_weylL} on $G$. 
For any a compact interval $I\subset (a-\delta_0,b+\delta_0)$, 
there exists $\eps_0\in (0,1]$ and $C>0$ 
 such that 
\begin{equation}
\label{eq_lem_Thm:semiclasical_weyl_law}
\forall \psi\in C_c^\infty(I), \  
\forall \eps\in (0,\eps_0]
\qquad 
\tr |\psi(T)| \leq C \sup_{\bR} |\psi| \ \eps^{-Q}.
\end{equation}
Moreover, the norm
$$
\|\psi (\sigma_0 )\|_{L^2(G\times \Gh)}^2=
\int_{G\times\Gh}  \|\psi (\sigma_0 (x,\pi)) \|_{HS(\cH_\pi)}^2 dx d\mu(\pi)
\leq C_{\sigma_0,a,b,\delta_0} \sup_{\bR} |\psi|^2 
,
  \quad 
$$
is finite, and we have
$$
\|\psi(T)\|_{HS(L^2(G))}^2 = \eps^{-Q} \int_{G\times\Gh}  \|\psi (\sigma_0 (x,\pi)) \|_{HS(\cH_\pi)}^2 dx d\mu(\pi) + \cO(\eps^{1-Q}).
$$
A similar result holds  on $M$.
\end{proposition}

This will require the following lemma:
\begin{lemma}
	\label{lem_inttrfsigma0}
		We consider  Setting \ref{set_FC} and Hypotheses \ref{hyp_m>0}, \ref{hyp_FC} and \ref{hyp_weylL} on $G$. We have
for any $f\in C_c (a-\delta_0,b+\delta_0)$:
$$
\forall \eps\in (0,1]\qquad
\|\Op^{(\eps)}_G(f(\sigma_0))\|^2_{HS(L^2(G))}
= \eps^{-Q} \|f(\sigma_0)\|_{L^2(G\times \Gh)}^2,
$$
and 
$$
\|f(\sigma_0)\|_{L^2(G\times \Gh)}
\leq \sup_\bR |f| \, C_{\sigma_0,a,b,\delta_0} .
$$

We have a similar result on $M$.
\end{lemma}
\begin{proof}[Proof of Lemma \ref{lem_inttrfsigma0}]
Let $f\in C_c(a-\delta_0,b+\delta_0)$.
By the Plancherel formula (see also the proof of Corollary~\ref{cor_trHSG} (2)), 
we have
\begin{align*}
\|\Op^{(\eps)}_G(f(\sigma_0))\|^2_{HS(L^2(G))} 
&= \int_{G\times \Gh} \|f(\sigma_0)(x,\eps\cdot \pi)\|_{HS(\cH_\pi)}^2 dxd\mu(\pi)\\
&=
\eps^{-Q}\int_{G\times \Gh} \|f(\sigma_0)(x, \pi')\|_{HS(\cH_{\pi'})}^2 dxd\mu(\pi')
\end{align*}
with the change of variable $\pi'=\eps\cdot \pi$, see \eqref{eq_muQhom}. This gives the equality. 
By functional calculus, we have
$$
\|f(\sigma_0)(x, \pi)\|_{HS(\cH_{\pi})}^2 
\leq \sup_\bR |f|^2 \, \tr |1_{(a-\delta_0,b+\delta_0)}(\sigma(x,\pi))|,
$$
and this implies the inequality. 
\end{proof}

We can now prove Proposition \ref{prop_Thm:semiclasical_weyl_law}.
\begin{proof}[Proof of Proposition \ref{prop_Thm:semiclasical_weyl_law}]
We fix an interval $I\subset (a-\delta_0,b+\delta_0)$.and  a function $f\in  C_c^\infty(a-\delta,b+\delta)$ valued in $[0,1]$ and such that $f=1$ on $I$.
The operator defined via
$$
R_f := R_f(\eps):=\eps^{-1}\left(f^2(T) - \Op^{(\eps)}_G(f(\sigma_0))\left( \Op^{(\eps)}_G(f(\sigma_0))\right)^*\right),
$$
satisfies for any $\eps\in (0,1]$
$$
\|R_f\|_{\sL(L^2(G))}
\leq \|R_f\|_{\sL(L^2(G), L^2_{1,\eps}(G))}
\lesssim \eps,
$$
by Theorem~\ref{thm_sclFCG} and the properties of the semiclassical pseudo-differential calculus.
Hence there exists $\eps_0\in (0,1]$ sufficiently small such that the operator $\id-\eps R_f(\eps))$ has an inverse in $\sL(L^2(G))$ for every $\eps\in(0,\eps_0]$ with 
\begin{equation}
	\label{eq1_lem_Thm:semiclasical_weyl_law}
	\|(\id-\eps R_f))^{-1}\|_{\sL(L^2(G))}
\leq \sum_{j=0}^\infty \|\eps R_f\|_{\sL(L^2(G))}^j
\lesssim 1.
\end{equation}

Let $\psi\in C_c^\infty(I)$.
Since $\psi=\psi f^2 $, we have by functional calculus
$$
A= AB, \quad\mbox{with}\ A:=\psi(T), \ B:=f^2(T),
$$
while  by Theorem~\ref{thm_sclFCG} and the properties of the semiclassical function calculus, 
$$
B=B_0+\eps R_f,  \quad\mbox{with}\ 
B_0:=\Op^{(\eps)}_G(f(\sigma_0))\, (\Op^{(\eps)}_G(f(\sigma_0)))^*.
$$
Therefore, we have for any $\eps\in (0,1]$,
$$
A=A(B_0+\eps R_f), \quad\mbox{or equivalently} \ 
A(\id -\eps R_f) =AB_0,
$$
yielding 	for $\eps \in (0,\eps_0]$,
$$
A= AB_0(\id-\eps R_f)^{-1},
$$
hence, 
$$
\tr |A| \leq \|A\|_{\sL(L^2(G))} \|(\id-\eps R_f)^{-1}\|_{\sL(L^2(G))} \tr |B_0|
\lesssim \|\psi\|_{L^\infty(\bR)} \|\Op^{(\eps)}_G(f(\sigma_0))\|^2_{HS(L^2(G))},
$$
having used the \eqref{eq1_lem_Thm:semiclasical_weyl_law} and functional analysis for the $\sL(L^2(G))$-norm. 
Using Lemma \ref{lem_inttrfsigma0} on the Hilbert-Schmidt norm, 
	we obtain \eqref{eq_lem_Thm:semiclasical_weyl_law}.
	
Proceeding as above for $f=\psi$ and $\eps=1$, we obtain the formula for 	
the norm $\|\psi (\sigma_0 )\|_{L^2(G\times \Gh)}$ and its estimate. 
By Theorem~\ref{thm_sclFCG} and the properties of the semiclassical function calculus, we have
\begin{equation}
	\label{eq_dev|psi|2}
	|\psi|^2 (T) = \Op^{(\eps)}_G(|\psi|^2(\sigma_0)) +\eps R_1, 
\end{equation}
and similarly for $|f|^2(T)$.
Since $|f|^2(T)$ is self-adjoint,
we have
\begin{align*}
	\|\psi(T)\|_{HS(L^2(G))}^2
	&= 
	\tr ( |\psi|^2 (T))
	= \tr ( |\psi|^2 (T)(|f|^2 (T))^*)
	\\
	&= \tr \left( |\psi|^2 (T)	\left(\Op^{(\eps)}_G(|f|^2(\sigma_0))\right)^*\right )
	+O(\eps^{1-Q})
	\\
	&= \tr \left( \Op^{(\eps)}_G(|\psi|^2(\sigma_0))
	\left(\Op^{(\eps)}_G(|f|^2(\sigma_0))\right)^*\right )
	+O(\eps^{1-Q}),
\end{align*}
having used 
\eqref{eq_dev|psi|2} for $f$ and $\psi$  together with \eqref{eq_lem_Thm:semiclasical_weyl_law}.
Bilinearising the first equality in \eqref{lem_inttrfsigma0} implies:
\begin{align*}
\tr \left( \Op^{(\eps)}_G(|\psi|^2(\sigma_0))
	\left(\Op^{(\eps)}_G(|f|^2(\sigma_0))\right)^*\right )
	&=\eps^{-Q}
	\int_{G\times\Gh}
	\tr \left( |\psi|^2(\sigma_0))
	\left(|f|^2(\sigma_0) \right)^*\right) dxd\mu
	\\&=
	\eps^{-Q} \int_{G\times\Gh}  \|\psi (\sigma_0 (x,\pi)) \|_{HS(\cH_\pi)}^2 dx d\mu(\pi).
\end{align*}
The statement follows. 
\end{proof}

We can now conclude the proof of Theorem \ref{Thm:semiclasical_weyl_law}.
\begin{proof}[Proof of Theorem \ref{Thm:semiclasical_weyl_law}]
Let $\underline \chi$ and $\overline \chi$ be two smooth functions on $\bR$, valued in $[0,1]$, satisfying 
$$
\overline{ \chi} 1_{[a,b]}=1_{[a,b]}, 
\quad\mbox{and}\quad 
\underline{ \chi}1_{[a,b]}=\underline{ \chi}.
$$
Then 
$$
\tr \left(|\underline \chi|^2(T)\right)\leq 
N(\eps) = \tr \left(1_{[a,b]}(T)\right)
\leq \tr \left(|\overline \chi|^2(T)\right).
$$
By Proposition \ref{prop_Thm:semiclasical_weyl_law}, this implies
\begin{align*}
	\liminf_{\eps\to 0} \eps^{-Q} N(\eps) &\leq \int_{G\times\Gh}  \|\overline \chi (\sigma_0 (x,\pi)) \|_{HS(\cH_\pi)}^2 dx d\mu(\pi),\\
\limsup_{\eps\to 0} \eps^{-Q} N(\eps) &\geq \int_{G\times\Gh}  \|\underline \chi (\sigma_0 (x,\pi)) \|_{HS(\cH_\pi)}^2 dx d\mu(\pi).
\end{align*}
This is true for any $\underline \chi$, $\overline \chi$ as above. By considering sequences of such functions converging point-wise to $1_{[a,b]}$, the result follows by Lebesgue dominated convergence.  
\end{proof}

 \appendix

 \section{Proof of the semiclassical composition and adjoint}
 \label{secA_pfthm_sclexp_prod+adj}
 In this section, we give a detailed proof of Theorem \ref{thm_sclexp_prod+adj}.
 We start with some tools required in the proof.
 
 \subsection{Adapted Taylor estimates and Leibniz properties}
 
 \subsubsection{Adapted Taylor estimates}
 Our analysis will require Taylor estimates adapted to graded groups and due to Folland and Stein \cite{folland+stein_82}, see also   Theorem 3.1.51 in~\cite{R+F_monograph}.

\begin{theorem}
\label{thm_MV+TaylorG}
Let $G$ be a graded Lie group, with adapted basis $(X_1,\ldots,X_n)$ for its Lie algebra. 
We fix  a quasinorm $|\cdot|$ on $G$.
\begin{itemize}
    \item {\rm Mean value theorem}. There exists $C_0>0$ and $\eta>1$ such that for any $f\in C^1(G)$ we have
    $$
    \forall x,y\in G,\qquad |f(xy) - f(x)| \leq C_0 \sum_{j=1}^n |y|^{\upsilon_j} \sup_{|z| \leq \eta |y|} |(X_j f)(xz)|.
    $$
    \item {\rm Taylor estimate}. 
    More generally, 
    with the constant $\eta$ of point (1), 
    for any $N\in \bN_0$, 
    there exists $C_N>0$ 
    such that for any $f\in C^{\lceil N\rfloor}(G)$ we have
    \begin{equation*}\label{taylor_1}
    \forall x,y\in G,\qquad |f(xy) - \bP_{G,f,x,N} (y)| 
    \leq C_N \sum_{\substack{|\alpha|\leq \lceil N\rfloor+1\\ [\alpha]>N} }|y|^{[\alpha]}
    \sup_{|z| \leq \eta^{\lceil N \rfloor+1} |y|} |(\bV^\alpha f)(xz)|.
    \end{equation*}
Above,  $\lceil N\rfloor$ denotes  $ \max\{|\alpha| : \alpha\in \bN_0^n$ with $[\alpha]\leq N\}$
and
$\bP_{G,f,x,N}$ denotes the Taylor polynomial of $f$ at $x$ of order $N$ for the graded group $G$, i.e. the unique linear combination of monomials of homogeneous degree $\leq N$ satisfying 
$X^\beta\bP_{G,f,x,N} (0)=X^\beta f(x)$  for any $\beta\in \bN_0^n$ with $[\beta]\leq N$.
\end{itemize}
\end{theorem}

\subsubsection{Leibniz properties}\label{subsubsec_Leibniz}.
The Leibnitz properties of vector fields readily imply for the product of symbols:
$$
X^\beta (\tau_1\tau_2) = \sum_{[\beta_1]+[\beta_2]=\beta}
c'_{\alpha_1,\alpha_2,\alpha} 
X^{\beta_1}\tau_1 \ X^{\beta_2}\tau_2,
$$
as well as for their $\diamond$-products:
$$
X^\beta (\sigma_1 \diamond_\eps \sigma_2) = \sum_{[\beta_1]+[\beta_2]=\beta}
c'_{\alpha_1,\alpha_2,\alpha} 
\int_G \kappa_{X^{\beta_1}\sigma_1,x_1}(z) \, \pi(z)^* X^{\beta_2}_{x_2=x}\sigma_2(x_2 \, \eps\cdot z^{-1},\pi) \, dz 
$$
A similar property holds for the difference operators. 
Indeed, recall that the $(q_\alpha)$ is a homogeneous basis of polynomials, and therefore it satisfies \cite[Section 3.1.4]{R+F_monograph}
$$
q_\alpha (xy) = \sum_{[\alpha_1]+[\alpha_2]=[\alpha]}c_{\alpha_1,\alpha_2,\alpha} q_{\alpha_1}(y) q_{\alpha_2}(x).
$$
This implies the Leibniz properties for $\Delta^\alpha$
for the product of symbols
$$
\Delta^\alpha (\tau_1 \tau_2)= \sum_{[\alpha_1]+[\alpha_2]=[\alpha]}c_{\alpha_1,\alpha_2,\alpha} \Delta^{\alpha_1}\tau_1 \Delta^{\alpha_2}\tau_2,
$$
and also for the $\diamond$-product resulting in $\sigma$:
 $$
\Delta^\alpha (\sigma_1 \diamond_\eps \sigma_2)= \sum_{[\alpha_1]+[\alpha_2]=[\alpha]}c_{\alpha_1,\alpha_2,\alpha} \Delta^{\alpha_1}\sigma_1 \diamond_\eps \Delta^{\alpha_2}\sigma_2.
$$

\subsection{Proof for the composition}
\label{subsec_pfthm_sclexp_prod}
Here, we prove  Theorem \ref{thm_sclexp_prod+adj} (1).
Let $\sigma_1 \in S^{m_1}(G\times \Gh)$ 
and $\sigma_2 \in S^{m_1}(G\times \Gh)$.
By Theorem \ref{thm_PDOGcomp+adj}, 
$$
\Op_G^{(\eps)}(\sigma_1)\Op_G^{(\eps)}(\sigma_2)
=
\Op_G^{(\eps)}(\sigma)\in \Psi^{m_1+m_2}(G\times \Gh)
$$
where
$$
\sigma:=
\sigma_1 \diamond_\eps \sigma_2
:=(\sigma_1(\cdot, \eps\, \cdot )\diamond \sigma_2(\cdot, \eps \, \cdot )))(\cdot, \eps^{-1}\, \cdot ) ,
$$
depends on $\eps\in (0,1]$.
By \eqref{eq_kappaadj}, 
the convolution kernel of $\sigma$ is 
$$
\kappa_{x}(y)= \eps^Q \int_G \kappa^{(\eps)}_{\sigma_2, xz^{-1}}( (\eps \cdot y)z^{-1})\ \kappa^{(\eps)}_{1,x}(z) dz 
=\int_G \kappa_{\sigma_2, x(\eps \cdot z)^{-1}}( y z^{-1}) \ \kappa_{\sigma_1,x}(z) dz .
$$
Therefore
$$
\sigma(x,\pi) =\int_G \kappa_{x}(y) \pi(y)^* dy
=\int_G \kappa_{\sigma_1,x}(z) \, \pi(z)^* \sigma_2(x \, \eps\cdot z^{-1},\pi) \, dz .
$$

By the Taylor estimates due to Folland 
and Stein (see Theorem \ref{thm_MV+TaylorG}), 
$$
\sigma_2(x \, \eps\cdot z^{-1},\pi)
= 
\sum_{[\alpha]\leq N} q_\alpha((\eps \cdot z)^{-1}) 
X^\alpha_x \sigma_2(x,\pi)
+R_{x,N}^{\sigma_2(\cdot,\pi) }(\eps \cdot z^{-1}),
$$
with remainder estimate
\begin{align*}
	&\|(\id +\pi( \cR))^{-\frac {m_2+\gamma} \nu} R_{x,N}^{\sigma_2(\cdot, \pi) }(\eps \cdot z^{-1}) (\id +\pi(\cR))^{\frac {\gamma} \nu} \|_{\sL(\cH_\pi)}
	\\ &\quad \leq C_{N} \sum_{\substack{|\alpha|\leq \lceil N\rfloor\\ [\alpha]>N}}
(\eps |z|)^{[\alpha]}
\sup_{z'\in G } 
\|(\id +\pi(\cR))^{-\frac {m_2+\gamma} \nu}
\sigma_2(z', \pi)  (\id +\pi(\cR))^{\frac {\gamma} \nu} \|_{\sL(\cH_\pi)}\\
&\quad \leq C_{N,\sigma_2} \eps^{N+1}
\sum_{\substack{|\alpha|\leq \lceil N\rfloor\\ [\alpha]>N}}
|z|^{[\alpha]} .
\end{align*}
We are led to study
\begin{align*}
	& \int_G \kappa_{\sigma_1,x}(z) \, \pi(z)^* 
R_{x,N}^{\sigma_2(\cdot,\pi) }(\eps \cdot z^{-1})
 \, dz
\\
&\qquad=
\sigma(x, \pi) - \sum_{[\alpha]\leq N} \int_G \kappa_{\sigma_1,x}(z) \, \pi(z)^*  q_\alpha((\eps \cdot z)^{-1}) 
X^\alpha_x \sigma_2(x,\pi) \, dz\\
&\qquad=\sigma(x, \pi) - \sum_{[\alpha]\leq N}\eps^{[\alpha]} 
\Delta^\alpha \sigma(x,\pi) X^\alpha_x \sigma_2(x,\pi).
	\end{align*}
The $\sL(\cH_\pi)$-norm of this expression for $\eps\in (0,1]$ and $m_2\leq 0$	is estimated by
$$
\| \sigma(x, \pi) - \sum_{[\alpha]\leq N}\eps^{[\alpha]} 
\Delta^\alpha \sigma_1(x,\pi) X^\alpha_x \sigma_2(x,\pi)\|_{\sL(\cH_\pi)}
	 \leq C_{N,\sigma_2}
\eps^{N+1}
\sum_{\substack{|\alpha|\leq \lceil N\rfloor\\ [\alpha]>N}}
\int_G |z|^{[\alpha]} |\kappa_{\sigma_1,x}(z)| dz.
$$
If $N$ is large enough (more precisely, $N$ such that 
$m_1-N<-Q$), then the integrals on the right-hand side are finite by the kernel estimates (see Theorem \ref{thm_kernelG}). 
More generally, we have:	
\begin{align*}
&\pi(X)^{\alpha_0}\left ( \sigma(x, \pi) - \sum_{[\alpha]\leq N}\eps^{[\alpha]} 
\Delta^\alpha \sigma(x,\pi) X^\alpha_x \sigma_2(x,\pi)\right)
\\&\qquad=
	 \int_G \kappa_{\sigma_1,x}(z) \, (\widetilde X^{\alpha_0}_z \pi(z))^* 
R_{x,N}^{\sigma_2(\cdot,\pi) }(\eps \cdot z^{-1})
 \, dz
 \\&\qquad=(-1)^{|\alpha_0|}
	 \int_G\pi(z)^* \widetilde X^{\alpha_0}_z\left(
	  \kappa_{\sigma_1,x}(z) \, 
R_{x,N}^{\sigma_2(\cdot,\pi) }(\eps \cdot z^{-1})\right)
  dz
 	\end{align*}
 	which, by the Leibniz property of vector fields, is a linear combination over $[\alpha_{0,1}]+[\alpha_{0,2}]=[\alpha_0]$ of 
\begin{align*}	 
&\int_G (\widetilde X^{\alpha_{0,1}} \kappa_{\sigma_1,x})(z) \,  \pi(z)^* 
\eps^{[\alpha_{0,2}]}(X^{\alpha_{0,2}} R_{x,N}^{\sigma_2(\cdot,\pi) })(\eps \cdot z^{-1})
 \, dz
\\& \quad =
	 \int_G  \kappa_{\widehat X^{\alpha_{0,1}}\sigma_1,x}(z) \,  \pi(z)^* 
\eps^{[\alpha_{0,2}]} R_{x,N-[\alpha_{0,2}]}^{X^{\alpha_{0,2}}\sigma_2(\cdot,\pi) })(\eps \cdot z^{-1})
 \, dz,
	\end{align*}
	when $N>[\alpha_0]$.
	Proceeding as above, we obtain:
\begin{align*}
	&\Big  \|\pi(X)^{\alpha_0}\left ( \sigma(x, \pi) - \sum_{[\alpha]\leq N}\eps^{[\alpha]} 
\Delta^\alpha \sigma(x,\pi) X^\alpha_x \sigma_2(x,\pi)\right ) \Big  \|_{\sL(\cH_\pi)}
	\\ &\quad \leq C_{\alpha_0}
\sum_{[\alpha_{0,1}]+[\alpha_{0,2}]=[\alpha_0]}
\int_G  |\kappa_{\widehat X^{\alpha_{0,1}}\sigma_1,x}(z)| \| R_{x,N-[\alpha_{0,2}]}^{X^{\alpha_{0,2}}\sigma_2(\cdot,\pi) })(\eps \cdot z^{-1},\pi)\|_{\sL(\cH_\pi)}
 \, dz\\
 \\ &\quad \leq C_{N,\sigma_2,\alpha_0}
\sum_{[\alpha_{0,1}]+[\alpha_{0,2}]=[\alpha_0]}
\eps^{N+1}
\sum_{\substack{|\alpha|\leq \lceil N\rfloor\\ [\alpha]>N}}
\int_G |z|^{[\alpha]} |\kappa_{\widehat X^{\alpha_{0,1}}\sigma_1,x}(z)| dz.
		\end{align*}
		Again, 
if $m_2\leq 0$ and $N$ is large enough (this time, with $N$ such that 
$m_1+[\alpha_0]-N<-Q$), then the integrals on the right-hand side are finite. 

If $m_2>0$, we proceed as follows. 
We consider a positive Rockland operator of homogeneous degree $\nu$ of degree high enough and symmetric in the sense that $\cR (f (x^{-1})) = \widetilde \cR f (z^{-1})$; for instance, we take $\cR$ as in Example \ref{ex_Rockland} (3).
We then introduce $\id=(\id+\widehat \cR)(\id+\widehat \cR)^{-1}$ in the following way:
\begin{align*}
&\sigma(x, \pi) - \sum_{[\alpha]\leq N}\eps^{[\alpha]} 
\Delta^\alpha \sigma(x,\pi) X^\alpha_x \sigma_2(x,\pi)\\
&\qquad=
	 \int_G \kappa_{\sigma_1,x}(z) \, \pi(z)^* (\id+\widehat \cR)(\id+\widehat \cR)^{-1}
R_{x,N}^{\sigma_2(\cdot,\pi) }(\eps \cdot z^{-1},\pi)
 \, dz\\
 &\qquad=
	 \int_G \kappa_{\sigma_1,x}(z) \, \pi(z)^* R_{x,N}^{(\id+\widehat \cR)^{-1}
\sigma_2(\cdot,\pi) }(\eps \cdot z^{-1},\pi)
 \, dz
 \\
 &\qquad\qquad +
	 \int_G \kappa_{\sigma_1,x}(z) \, (\widetilde \cR_z \pi(z)^*) R_{x,N}^{(\id+\widehat \cR)^{-1}
\sigma_2(\cdot,\pi) }(\eps \cdot z^{-1},\pi)
 \, dz.
	\end{align*}
	For the first integral on the right hand side, we proceed as above, while for the second an integration by part yields a linear combination over $[\alpha'_1]+[\alpha_2']=\nu$ of 
	\begin{align*}
	 &\int_G (\widetilde X^{\alpha'_1}\kappa_{\sigma_1,x})(z) \, (\widetilde \cR_z \pi(z)^*) \widetilde X^{\alpha'_2}_{z'=\eps \cdot z^{-1}} R_{x,N}^{(\id+\widehat \cR)^{-1}
\sigma_2(\cdot,\pi) }(z',\pi)
 \, dz\\
 &\qquad =\int_G (\kappa_{\widehat X^{\alpha'_1}\sigma_1,x})(z) \,  \pi(z)^* {z'=\eps \cdot z^{-1}} R_{0,N-[\alpha'_2]}^{(\id+\widehat \cR)^{-1}
X^{\alpha'_2}_x \sigma_2(x \cdot,\pi) }(z',\pi)
 \, dz.
	\end{align*}
Again, for $N$ large enough, we can proceed as above. 

We have obtained that for any $m_1,m_2$, $\alpha_0\in \bN_0^n$, there exists $N_0$ such that for any $N\geq N_0$, there exists $C>0$ satisfying 
$$
\sup_{x\in G,\pi\in \Gh} 
\|\pi(X)^{\alpha_0}( \sigma(x, \pi) - \sum_{[\alpha]\leq N}\eps^{[\alpha]} 
\Delta^\alpha \sigma(x,\pi) X^\alpha_x \sigma_2(x,\pi)) \|_{\sL(\cH_\pi)}
	\leq  C\eps^{N+1}.
	$$
This implies the desired estimate for the norm $\|\cdot\|_{S^{m_1+m_2},0,0,0}$.
The estimates for the semi-norms $\|\cdot\|_{S^m,a,b,0}$ is then a consequence of the Leibniz properties presented in Section \ref{subsubsec_Leibniz}.This concludes the proof of Theorem \ref{thm_sclexp_prod+adj} (1).

\subsection{Proof for the adjoint}

Here, we prove  Theorem \ref{thm_sclexp_prod+adj} (2).
Let $\sigma\in S^m(G\times \Gh)$. Then 
 $T^{(\eps)}:=(\Op_G^{(\eps)}(\sigma))^*\in \Psi^{m}(G\times \Gh)$ by Theorem \ref{thm_PDOGcomp+adj}, and 
$$	T^{(\eps)} = \Op^{(\eps)}_G(\sigma)\quad
\mbox{where}\quad
\sigma:=
\sigma^{(\eps,*)} := 
(\sigma(\cdot , \eps \, \cdot))^{(*)}(\cdot , \eps^{-1} \, \cdot),
$$
with $\sigma$ dependent on $\eps\in (0,1]$. 
By \eqref{eq_kappaadj},
the convolution kernel of $\sigma$ is 
$$
\kappa_{x}(y)= \eps^{Q}\bar \kappa_{ \sigma(\cdot , \eps \, \cdot), x\eps \cdot y^{-1}}(\eps\cdot y^{-1}) 
= \bar \kappa_{ \sigma, x\eps\cdot y^{-1}}(y^{-1}) .
$$
By the Taylor estimates due to Folland and Stein (see Theorem \ref{thm_MV+TaylorG}, 
$$
\kappa_{\sigma, x\eps\cdot y^{-1}}(w) 
= 
\sum_{[\alpha]\leq N} q_\alpha(\eps\cdot y^{-1}) 
X^\alpha_x \kappa_{\sigma, x}(w)
+R_{x,N}^{\kappa_{\sigma, \cdot }(w) }(\eps\cdot y^{-1}),
$$
with remainder estimate
\begin{align*}
	|R_{x,N}^{\kappa_{\sigma, \cdot }(w) }(\eps\cdot y^{-1})|
	&\leq C_{N} \sum_{\substack{|\alpha|\leq \lceil N\rfloor\\ [\alpha]>N}} \eps^{[\alpha]}
|y|^{[\alpha]}
\sup_{|y'|\leq  \eta^{\lceil N\rfloor +1 } \eps |y|} |X^\alpha _{x'=xy'}\kappa_{\sigma, x'}(w)|.
\end{align*}
By the kernel estimates (see Theorem \ref{thm_MV+TaylorG}), this implies
$$
\sup_{\eps^{-(N+1)}} |R_{x,N}^{\kappa_{\sigma, \cdot }(w) }(\eps\cdot y^{-1})|
\leq C_{\sigma,N_1} \eps^{N_1} (1+|y|)^{-N_1}
$$
for any $N_1\in \bN$  if $m_1-N <-Q$.
Hence, we have
\begin{align*}
\|\sigma^{(\eps,*)}
- \sum_{[\alpha]\leq N} \Delta^\alpha \sigma^*
\|_{\sL(\cH_\pi)}
&=\|	\int_G 
R_{x,N}^{\kappa_{\sigma, \cdot }(w) }(\eps\cdot y^{-1}) 
\pi(y)^* dy \|_{\sL(\cH_\pi)}\\
&\leq 	\int_G 
\|R_{x,N}^{\kappa_{\sigma, \cdot }(w) }(\eps\cdot y^{-1}) \|_{\sL(\cH_\pi)}
 dy \\
 &\leq \eps^{N_1}	\int_G C_{\sigma,N_1} (1+|y|)^{-N_1} dy,
\end{align*}
is finite if $N_1>Q$. 
This implies the desired estimate for the norm $\|\cdot\|_{S^{m},0,0,0}$.
The estimates for the semi-norms $\|\cdot\|_{S^m,a,b,0}$ is then a consequence of the Leibniz properties presented in Section \ref{subsubsec_Leibniz}.This concludes the proof of Theorem \ref{thm_sclexp_prod+adj} (2).

\section{Proof of Lemma~\ref{lem_tildepsicGm'}}
\label{Appendix_almost_analytic}

This Appendix is devoted to proving Lemma~\ref{lem_tildepsicGm'}.
By density of $C_c^\infty(\bR)$ in $\cG^{m'}(\bR)$, 
Lemma~\ref{lem_tildepsicGm'}  is a consequence of 
Lemma \ref{lem_HSCcinfty}
and the following statement:
\begin{proposition}
\label{propApp_tildepsicGm'}	
Let $\psi\in \cG^{m'}(\bR)$ with $m'<-1$. Then we can construct an almost analytic extension $\tilde {\psi} \in C^\infty(\bC)$ to $\psi$ such that  we have for all $N\in\bN_0$,
$$
\int_{\bC} \big| \bar\partial \tilde \psi( z) \big| \left(\frac{1+|z| }{|\IM \, z|}\right)^N L(dz) \leq C_{N}  \|\psi\|_{\cG^{m'},N+3},
$$
and for any $N,p\in \bN$,
$$
	|\partial_y^p \tilde{\psi}(z) | \leq C_{N,p}\frac{|\IM \, z|^N}{
	(1+|z|)^{p-m'}} \|\psi\|_{\cG^{m'},N+p+2}.
$$
Above the constants  $C_{N},C_{N,p}>0$ depend on $N$ and on $N,p$  respectively
(and on the construction and on $m'$), but not on $\psi$.
Moreover, we have
$$
\forall k\in \bN_0,\qquad 
\partial_x^k \tilde \psi|_\bR =(-i\partial_y)^k \tilde \psi|_\bR= \partial_x^k  \psi,
$$ 	
and
$$
\supp (\tilde \psi)\subset 
\{x+iy\in \bC \,| \, {\rm dist}( x,  \supp(\psi))\leq 1 \text{ and } |y| \leq 10(1+|x|)\}.
$$ 
\end{proposition}

The proof Lemma~\ref{lem_tildepsicGm'} will be complete once we show Proposition \ref{propApp_tildepsicGm'}.
The proof and construction of the latter	is standard, 
although the particular estimates for the derivatives in $y$ may not always be given explicitely; so, 
we include them here. 
It is based on the ideas of Jensen and Nakamura in \cite{MR1275402}. 
The first author would  like to take this opportunity to correct a statement that has appeared in her paper \cite[Section 3.1.1]{mydocumenta}.

The proof of Proposition \ref{propApp_tildepsicGm'} will rely on the following auxiliary results. 
\begin{lemma}
	\label{lem_AAExt1}
We fix a function $\chi\in C_c^\infty(\bR : [0,1])$ such that $\chi=1$ on $[-2,2]$ and $\chi=0$ outside $[-4,4]$.
Let $\psi\in C_c^\infty(-2,2)$.
We set
$$
\tilde \psi(z)
:=
\int_\bR e^{2\pi i z\xi} \chi(y\xi) \widehat \psi(\xi) d\xi,
\quad z=x+iy.
$$
This defines a smooth function $\tilde \psi:\bC\to \bC$. It  
is an almost analytic extension  of $\psi$
that  satisfies 
$$
\partial_x^k \tilde \psi|_\bR =(-i\partial_y)^k \tilde \psi|_\bR= \partial_x^k  \psi \quad \mbox{for any}\ k\in \bN_0,
$$ 	
and
for any $N\in\bN_0$
$$
\big|\bar\partial \tilde \psi(z)  \big|  \leq  C_{N}  |y|^N  \max_{k=0,\ldots, N+3 } \sup_{x\in\bR} |\psi^{(k)}(x)|,
$$
and for any $N,p\in \bN$
$$
|\partial_y^p \tilde\psi(z) | \leq C_{N,p}
|y| ^N
\max_{k=0,\ldots, N+p+2 } \sup_{x\in\bR} |\psi^{(k)}(x)|.
$$
This last bounds also holds for $N\in \bN$ and 
 $p=0$ under the additional assumption that  
 $|y|\geq 1$ or ${\rm dist} (x,\supp\, \psi)\geq \eps_0$ for some  fixed $\eps_0>0$. 
 Above, the constants $C_{N}$ and  $C_{N,p}$  are  independent of the function $\psi$.
\end{lemma}
\begin{proof} 
By the properties of the Fourier transform, 
we check readily that $\tilde \psi\in C^\infty(\bC)$ is an almost analytic extension of $\psi$ whose derivatives satisfy the  equalities of the statement. 
Let us write 
$$
\chi_0(y_1) := e^{-2\pi y_1}\chi(y_1)
\quad\mbox{and}\quad
\eta(y_1) := e^{-2\pi y_1}\chi'(y_1).
$$
 We have
\begin{align*}
\bar \partial \tilde \psi(z)
&:=
\int_\bR e^{2\pi i z\xi}\xi  \chi'(y\xi) \,\widehat \psi(\xi) d\xi
= 	\int_\bR e^{2\pi i x\xi}  \eta(y\xi)\, \xi\widehat \psi(\xi) d\xi,
\\
\tilde\psi(z) 
&=	
\int_\bR e^{2\pi i x\xi} \chi_0(y\xi) \,\widehat \psi(\xi) d\xi,
\\
\partial_y^p
\tilde\psi(z) 
&=	
\int_\bR e^{2\pi i x\xi} \chi_0^{(p)}(y\xi)\, \xi^p\widehat \psi(\xi) d\xi.
\end{align*}
As $\eta(y_1)$ is  supported in $\{|y_1|\in [2,4]\}$,  we have
$$
|\bar \partial \tilde \psi(z)|
\leq
|y|^N \int_\bR |  \xi^{N+1}\widehat \psi(\xi)| d\xi.
$$
Since $\psi$ is compactly supported, we have
$$
\int_\bR |  \xi^{N_1}\widehat \psi(\xi)| d\xi
\lesssim 
\max_{k=N_1,N_1+2} \sup_{\xi\in \bR}|\xi|^k |\widehat \psi(\xi)|
\lesssim 
\max_{k=N_1,N_1+2} \|\psi^{(k)}\|_{L^1(\bR)}
\lesssim 
\max_{k=N_1,N_1+2} \|\psi^{(k)}\|_{L^\infty(\bR)},
$$
yielding the first estimate. 
The same argument gives the second estimate when $p>0$ since $\chi_0^{(p)}$ has the same support property as $\eta$.
For the case $p=0$, 
 we  develop the Fourier transform:
\begin{align*}
\tilde\psi(z) &=
\int_\bR  e^{2\pi i (z-w)\xi} 
	\chi_0(y\xi)
	\psi(w) dw d\xi
	\\&=y^{N} \int_\bR  e^{2\pi i (z-w)\xi} 
	\chi_0^{(N)}(y\xi)
	\frac{\psi(w)}{(-2\pi i (z-w))^{N}} dw d\xi,
\end{align*}
after $N$ integrations by parts in $\xi$.
Since $\chi^{(N)}(y_1)$ is supported in $\{|y_1|\in [2,4]\}$,
we  readily estimate the integrand  as long as $|z-w|$ is bounded below away from zero. 
This is the case  when $|y|\geq 1$ or 
when ${\rm dist} (x,\supp\, \psi)\geq \eps_0$. 
\end{proof}

\begin{corollary}
\label{cor_lem_AAExt1}
We continue with the hypothesis and notation of Lemma \ref{lem_AAExt1}.
We also fix the following:
\begin{itemize}
	\item  a function $\chi_1\in C_c^\infty(\bR : [0,1]$ such that $\chi_1=1$ on $[-1,1]$ and $\chi_1=0$ outside $[-2,2]$, 
	\item an interval $I\subset [-2,2]$ and a function $\chi_2\in C_c^\infty(\bR : [0,1])$ such that $\chi_2(x)=1$ if ${\rm dist} (x,I) \leq \eps_1$ and $\chi_2=0$ outside $I':=\{x\in \bR , {\rm dist} (x,I) \leq 2\eps_1\}$ with $\eps_1\in (0,0.1)$ fixed.  
\end{itemize} 
We set
$$
\phi(z) :=\chi_1(y) \chi_2(x) \tilde \psi(z), \qquad z=x+iy.
$$
This defines a smooth function $\phi:\bC\to \bC$ supported in $I'\times [-2,2]$. 
Assuming that $\supp\, \psi\subset I$, 
$\phi$
is an almost analytic extension  of $\psi$ that  satisfies for any  $k\in \bN_0$ $$
\partial_x^k \phi|_\bR =(-i\partial_y)^k \phi|_\bR= \partial_x^k  \psi.
$$ 	
Moreover,  for any $N\in\bN_0$
$$
\big|\bar\partial \phi(z)  \big|  \leq  C_{N}  |y|^N  \max_{k=0,\ldots, N+3 } \sup_{x\in\bR} |\psi^{(k)}(x)|,
$$
and for any $N,p\in \bN$
$$
	|\partial_y^p \phi(z) | \leq C_{N,p}
|y| ^N
\max_{k=0,\ldots, N+p+2} \sup_{x\in\bR} |\psi^{(k)}(x)|.
$$
Above, the constants $C_{N}$ and  $C_{N,p}$  are  independent of the function $\psi$. 
\end{corollary}
\begin{proof}
The statement follows from Lemma \ref{lem_AAExt1} and the computations:
\begin{align*}
	\bar \partial \phi (z) 
	&= 
\chi_1 (y)\chi_2(x) \bar \partial \tilde \psi
+\frac 12 \chi'_1(y) \chi_2(x)\tilde \psi(z)
 +\frac i2 \chi_1 (y)\chi_2'(x)\tilde \psi(z) ,
 \\
 \partial_y^p \phi(z) 
 &=
 \sum_{p_1=0}^p \binom p {p_1} \chi_1(x) \chi_2^{(p_1)}(y) \partial_y^{p-p_1} \tilde \psi(z). 
\end{align*}	
\end{proof}

We are now ready to prove Proposition \ref{propApp_tildepsicGm'}:

\begin{proof}[Proof of Proposition \ref{propApp_tildepsicGm'}]
Applying  Corollary \ref{cor_lem_AAExt1}
with $I=[-2,2]$ gives the property for a function $\psi$ with compact support in $[-3/2,3/2]$. Hence, 
 we may assume that $\psi$ is supported outside $[-1,1]$.
Let us show the case of $\supp\, \psi\subset [1,+\infty)$, the case of $(-\infty,-1]$ following readily after sign modifications. 
 
We fix a dyadic decomposition  $(\eta_j)$ of $[1/2,+\infty)$, that is, $\eta_0\in C_c^\infty (\frac 12,2) $ with 
	$$
	\sum_{j=-1}^\infty \eta_j(\lambda)=1 \ \mbox{for all}\ \lambda\geq 1/2, \quad\mbox{where}\quad \eta_j(\lambda):=\eta_0(2^{-j}\lambda).
	$$
	As $\supp\, \psi\subset [1,+\infty)$, 
	we may write for any $\lambda\in \bR$
	$$
	\psi(\lambda) = \sum_{j=0}^\infty 2^{jm'} \psi_j (2^{-j}\lambda), 
	\quad\mbox{where}\quad 
	 \psi_j (\lambda)\coloneqq 2^{-jm'} \psi(2^j \lambda) \eta_0(\lambda).
	$$
	We observe that $\psi_j \in C_c^\infty (\frac 12,2)$ satisfies for all $j\in\bN_0$ and any $k\in \bN_0$
	\begin{equation}
		\label{eq_estphijpsi}
		\sup_{\lambda\in \bR} |\phi_j ^{(k)}(\lambda)| \lesssim_k \|\psi\|_{\cG^{m'},k},
	\end{equation}
	 with an implicit constant independent of $j$. 
	 
Let  $\phi_j$ be the almost analytic extension   constructed in Corollary \ref{cor_lem_AAExt1} for $I=[1/2,2]$ and $\eps_1=0,001$.
We set at least formally:
		\begin{equation}\label{EQ:Almost_analytic_lemma_2_1}
		\tilde{\psi}(z) :=  \sum_{j=0}^\infty 2^{jm'} \phi_j (2^{-j} z).
	\end{equation}
This sum is in fact locally finite. 
Indeed, let us write   $z=x+iy$. 
Each term in the sum in \eqref{EQ:Almost_analytic_lemma_2_1}
 vanishes when  $x<1/4$, so we may assume   $x\geq 1/4 $. We  set $j_0\in \bN_0$ such that  
$x\sim 2^{j_0}$, in the sense that  $j_0\in \bN_0$ is the 
largest non-negative integer smaller than $\ln x / \ln 2$ when $\ln x / \ln 2>0$, 
and $j_0=0$ otherwise.
We have 
$$
		\tilde{\psi}(z) := 
		\sum_{j=\max(0,j_0-1)}^{j_0+1} \phi_j (2^{-j} z).
		$$
This finite sum  vanishes when  $|y| > 2^{j_0+2}$, implying  the stated support  property  for $\tilde \psi$. 

As the sum over $j\in \bN_0$ is locally finite, we check readily that the function $\tilde \psi$ is smooth.
It follows from
 Corollary \ref{cor_lem_AAExt1}  that it is an almost analytic extension of $\psi$ with derivatives satisfying the  stated equalities and supported in $\{\RE\, z\geq 1/4\}$.
We also have with $\RE\, z = x \geq  1/4$ and $j_0$ as above, 
\begin{align*}
\big|\partial_y^p \tilde{\psi}(z) \big| 
&\leq \sum_{j=\max(0,j_0-1)}^{j_0+1} 2^{j(m'-p)} |(\partial_y^p \phi_j) (2^{-j} z)| 
\\
&\lesssim_{N,p} |y|^N
\sum_{j=\max(0,j_0-1)}^{j_0+1} 2^{j(m'-p)}   \max_{k=0,\ldots, N+p+2 } \sup_{x\in\bR} |\psi_j^{(k)}(x)| 
\\
&\lesssim_{N,p} \|\psi\|_{\cG^{m'},N+p+2}|y|^N  (1+|z|)^{m'-p},
\end{align*}
by \eqref{eq_estphijpsi} and
since $(1+|z|) \sim 2^{j_0}$.

It  remains to show the  stated integral estimate. 
We have for any $N\in\bN_0$
\begin{align*}
	&\int_{\bC} \big| \bar\partial \tilde \psi( z) \big| \left(\frac{1+|z|}{|y|}  \right)^N L(dz)
\leq    
\sum_{j=0}^\infty 2^{j(m'+1-N)} 
\int_\bC
\big| \bar\partial \phi_j( z) \big| 
\left(\frac{1+2^j |z|}{|y|}  \right)^{N}
  L(dz)\\
 &\qquad \lesssim_N  \sum_{j=0}^\infty 
 2^{j(m'+1)} 
\int_{|z|\leq 10}
 \frac{\big| \bar\partial \phi_j( z) \big|}{|y|^N}    L(dz)
 \\
 &\qquad \lesssim_N  \sum_{j=0}^\infty 2^{j(m'+1)}
  \max_{k=0,\ldots, N+3 } \sup_{x\in\bR} |\phi_j^{(k)}(x)| 
  \lesssim_N  \|\psi\|_{\cG^{m'},N+3}
  \sum_{j=0}^\infty 2^{j(m'+1)},
\end{align*}
having used the properties of $\phi_j$ 
from Corollary \ref{cor_lem_AAExt1} and in \eqref{eq_estphijpsi} . The last  sum is convergent since $m'<-1$. 
\end{proof}

\bibliographystyle{alpha.bst}

\bibliography{bibfile.bib}

\end{document}